\documentclass[12pt,a4paper]{article}
\usepackage[utf8]{inputenc}
\usepackage[T1]{fontenc}
\usepackage{amsmath}
\usepackage{amsthm}
\usepackage{amssymb}
\usepackage{graphicx}
\usepackage{hyperref}
\usepackage{url}
\usepackage{stmaryrd}
\hypersetup{
colorlinks=true,
linkcolor=red,
filecolor=magenta,      
urlcolor=cyan,
}

\usepackage{minitoc}
\usepackage[left=2cm,right=2cm,top=2cm,bottom=2cm]{geometry}
\newtheorem{theorem}{Theorem}[section]

\newtheorem{proposition}[theorem]{Proposition}
\newtheorem{definition}[theorem]{Definition}
\numberwithin{equation}{section}
\newtheorem{lemma}[theorem]{Lemma}
\newtheorem{remark}[theorem]{Remark}

\DeclareMathOperator{\n}{\Vert}
\DeclareMathOperator{\br}{\vert}
\newcommand{\Dp}{{\sf D}_\rho}
\newcommand{\Ds}{{\sf D}_\rho^2}
\newcommand{\Dr}{\widetilde{{\sf D}}_\rho} 
\newcommand{\rd}{{\sf R}_{\rm d}}
\newcommand{\ru}{{\sf R}_{\rm u}}
\newcommand{\Norm}[1]{\boldsymbol\lVert #1 \boldsymbol\rVert}

\DeclareMathOperator{\R}{\mathbb{R}}
\DeclareMathOperator{\N}{\mathbb{N}}

\newcommand{\bott}{{\rm bott}}
\newcommand{\surf}{{\rm surf}}

\title{Approximating a continuously stratified hydrostatic system by the multi-layer shallow water system}
\author{ Mahieddine Adim 
\thanks{Univ Rennes, CNRS, IRMAR - UMR 6625
F-35000 Rennes, France. \href{mailto:mahieddine.adim@univ-rennes1.fr }{mahieddine.adim@univ-rennes1.fr}}}
\date{\today}

\begin{document}
\maketitle
\thispagestyle{empty}

\begin{abstract}
In this article we consider the multi-layer shallow water system for the propagation of gravity waves in density-stratified flows, with additional terms introduced by the oceanographers Gent and McWilliams \cite{GentMcWilliams} in order to take into account large-scale isopycnal diffusivity induced by small-scale unresolved eddies. We establish a bridge between the multi-layer shallow water system and the corresponding system for continuously stratified flows, that is the incompressible Euler equations with eddy-induced diffusivity under the hydrostatic approximation.

Specifically we prove that, under an assumption of stable stratification, sufficiently regular solutions to the incompressible Euler equations can be approximated by solutions to multi-layer shallow water systems as the number of layers, $N$, increases. Moreover, we provide a convergence rate of order $1/N^2$.

A key ingredient in the proof is a stability estimate for the multi-layer system which relies on suitable energy estimates mimicking the ones recently established by Bianchini and Duchêne \cite{DB} on the continuously stratified system. This requires to compile a dictionary that translates continuous operations (differentiation, integration, {\em etc.}) into corresponding discrete operations.
\end{abstract} 

\section{Introduction}

\paragraph{Equations at stake}
This work is concerned with the multi-layer shallow water system
\begin{equation}\label{mltc-intro}
\left\{\begin{array}{l}
\partial_tH_i+\partial_x((\underline{H}_i+H_i)(\underline{U}_i+U_i))=\kappa\partial_x^2H_i\\ 
\partial_tU_i+\left(\underline{U}_i+U_i-\kappa\frac{\partial_xH_i}{\underline{H}_i+H_i}\right)\partial_xU_i+g\sum\limits_{j=1}^N\frac{1}{N}\frac{\min(\boldsymbol{\rho}_i,\boldsymbol{\rho}_j)}{\boldsymbol{\rho}_i}\partial_xH_j=0,
\end{array}     \quad \forall i\in\{1,\cdots,N\}
\right.  \end{equation}
as an approximation to the continuously stratified system
\begin{equation}\label{contsystem-intro}
\left\{\begin{array}{l}
\partial_th+\partial_x((\underline{h}+h)(\underline{u}+u))=\kappa\partial_x^2h,\\ 
\partial_tu+\left(\underline{u}+u-\kappa\frac{\partial_xh}{\underline{h}+h}\right)\partial_xu+\frac{g}{\varrho}\partial_x\psi=0,
\end{array}     
\right.  \end{equation}
where 
$	\partial_x\psi(t,x,\varrho)=\rho_{\surf}\int\limits_{\rho_{\surf}}^{\rho_{\bott}}\partial_xh(t,x,\varrho')d\varrho'+\int\limits_{\rho_{\surf}}^{\varrho}\int\limits_{\varrho'}^{\rho_{\bott}}\partial_xh(t,x,\varrho'')d\varrho''d\varrho'$.

When $\kappa=0$, system \eqref{contsystem-intro} is a reformulation of the incompressible Euler equation with hydrostatic approximation using isopycnal coordinates. Such reformulation is possible when the fluid is stratified, by which we mean that sheets of equal densities realize a foliation of the fluid domain; see {\em e.g.} \cite{HLB,DB}. When the stratification is stable, that is the density is increasing with depth for all horizontal space locations $x\in\R$, the reformulation with isopycnal coordinates makes use of the inverse of the density function $z\mapsto \rho(\cdot,z)$, which we denote $\varrho\mapsto \eta(\cdot,\varrho)$ for $\varrho\in (\rho_{\surf},\rho_{\bott}) $ where $\rho_{\surf}$ is the (constant) density at the free surface and $\rho_{\bott}$ is the (constant) density at the rigid flat bottom. Since the graph of $\eta(\cdot,\varrho)$ represents in Eulerian coordinates the isopycnal sheet of density $\varrho$, the function  defined as $\underline{h}+h:=-\partial_\varrho \eta$ is the infinitesimal depth of this isopycnal sheet. Then, $\underline{u}+u$ is the horizontal velocity component of fluid particles at the isopycnal sheet. We decompose the depth and horizontal velocities as $\underline{h}+h$ and $\underline{u}+u$ where $(\underline{h},\underline{u})$ represent the background shear flow and are given functions depending only on the density variable $\varrho\in (\rho_{\surf},\rho_{\bott}) $, and $(h,u)$ are the unknowns representing the deviations from the equilibrium and depending on time $t$, horizontal space $x\in\R$, and density variable  $\varrho\in (\rho_{\surf},\rho_{\bott}) $. Finally, $g$ is the constant gravity acceleration and $\psi$ is the so-called Montgomery potential. It is responsible for the interaction between isopycnal sheets.

Still with $\kappa=0$, the system \eqref{mltc-intro} corresponds to a situation of $N$ layers of immiscible fluids with constant densities $\boldsymbol{\rho}_i$, $ i\in\{1,\cdots,N\}$. Applying the hydrostatic approximation and columnar assumption, we arrive (see {\em e.g.} \cite{LeeSu77,Baines88}) at the multi-layer shallow water system \eqref{mltc-intro} where $\underline{H}_i+H_i$ represents the depth of the $i^{\rm th}$ layer ($\underline{H}_i$ being the depth at rest and $H_i$ the deviation) and $\underline{U}_i+U_i$ being the layer-averaged horizontal velocity ($\underline{U}_i$ being the background velocity and $U_i$ the deviation).

In \eqref{mltc-intro} (respectively \eqref{contsystem-intro}), the terms proportional to $\kappa$ have been introduced by Gent and McWilliams \cite{GentMcWilliams} so as to represent the large-scale contribution of unresolved eddies. They appear as additional effective velocities $U_i^\star:=-\kappa\frac{\partial_xH_i}{\underline{H}_i+H_i}$ (respectively $u^\star:=-\kappa\frac{\partial_xh}{\underline{h}+h}$) and act as diffusivity contributions in the mass conservation equations. Interestingly, similar terms have been introduced in the work of  
Duran, Vila, and Baraille \cite{ARN} so as to control the discrete energy of a semi-implicit numerical scheme for the multilayer system \eqref{mltc-intro}.
From the mathematical viewpoint, as discussed below, the regularizing effect of diffusivity contributions is essential to our analysis as it provides appropriate stability estimates, and we shall always assume $\kappa>0$.
\medskip

\paragraph{Stability aspects}
In the situation $\kappa=0$, system \eqref{mltc-intro} is a system of conservation equations and the well-posedness theory of the initial-value problem (and more generally stability properties) relies on hyperbolicity conditions; see \cite{BG}. Yet as soon as $N\geq 2$, explicit formula for the hyperbolic domain of the system become out of reach \cite{Ovsjannikov79,BarrosChoi08,VirissimoMilewski20}.
In \cite{duchene:hal-00922045,Monjarreta}, the authors provide sufficient conditions for strong hyperbolicity when the fluid is stably stratified (that is $\boldsymbol{\rho}_1<\boldsymbol{\rho}_2<\cdots<\boldsymbol{\rho}_N$) but these conditions are obtained using perturbative arguments with respect to the situation without shear velocities and degenerate as $N\to\infty$ to
\[ \big\{ (\underline{H}_i+H_i,\underline{U}_i+U_i)\in \R^{2N}\ : \ \underline{H}_i+H_i>0,\ \underline{U}_1+U_1=\underline{U}_2+U_2=\cdots =\underline{U}_N+U_N \big\},\]
preventing any study at the nonlinear level or including shear velocities. We let the reader refer to the interesting discussion in \cite[\S5]{Ripa91} concerning stability results for the multi-layer systems when increasing the number of layers, in relation with stability results for the continuously stratified system. It is recalled therein the celebrated stability criterion of Miles \cite{Miles61} and Howard \cite{Howard61}, preventing normal mode instability of continuously stratified shear flows if the local Richardson number is everywhere greater than $1/4$. Its is also clarified that this notion of stability (as well as other results, obtained by Holm and Long \cite{HLB} and Arbanel et al. \cite{AbarbanelHolmMarsdenEtAl86}) is too weak to imply the control of deviations from equilibria in particular at the nonlinear level.

Consistently, the well-posedness of the initial-value problem in finite-regularity spaces for the continuously stratified system \eqref{contsystem-intro} in the absence of diffusivity ($\kappa=0$) is an open problem. We let the reader refer to the work of Kukavica et al. \cite{Kukavica} for the existence and uniqueness of a solution in spaces of analytical functions (in the rigid-lid setting), and to Cao, Li \& Titi \cite{CaoTiti} among many other works (see \cite{LiTitiYuan22} for a recent account) for the situation with horizontal viscosity and diffusivity contributions. While the aforementioned works deal with the hydrostatic Euler equations written with Eulerian coordinates and do not rely on the stable stratification assumption, the work of Bianchini and Duchêne \cite{DB} is closer to our setting as it specifically deals with the system in isopycnal coordinates with the additional terms of Gent and McWillams, that is \eqref{contsystem-intro}. They show that sufficiently regular initial data satisfying the non-cavitation assumption $(\underline{h}+h)\vert_{t=0}\geq h_*>0$ (which in fact represents a stable stratification assumption) gives rise to a unique solution on a time interval $[0,T]$ with $T^{-1}\lesssim 1+\kappa^{-1}(|\underline{u}'|_{L^2((\rho_{\surf},\rho_{\bott}))}^2+M_0^2) $ where $M_0$ is the size of the initial deviations from the equilibrium.
\medskip

\paragraph{Main results} Our first main result is the analogous conclusion for the multi-layer system, namely that under natural hypotheses and in particular the stable stratification assumption (in fact we assume for simplicity that densities are equidistributed, {\em i.e.} $\boldsymbol{\rho}_{i+1}-\boldsymbol{\rho}_i = \frac{\rho_{\rm bott}-\rho_{\rm surf}}{N}$) and the non-cavitation assumption $(\underline{H}_i+H_i)\vert_{t=0}\geq h_*>0$, the solutions to \eqref{mltc-intro} are unique and exist on a time interval analogous to the one of the continuously stratified system, which in particular is {\em uniform with respect to $N$}. 

Our second main result states that for any sufficiently regular solution to the continuously stratified system \eqref{contsystem-intro} satisfying the non-cavitation assumption and appropriate bounds, the solutions to the multi-layer systems  \eqref{mltc-intro} with suitably chosen densities $\boldsymbol{\rho}_{i}$, reference depths $\underline{H}_i$, background velocities $\underline{U}_{i}$ and initial deviation $(H_i,U_i)\vert_{t=0}$ are at a distance $\mathcal{O}(1/N^2)$ to the continuously stratified solution.
\medskip

In order to achieve these goals, we rely mostly on two ingredients:
\begin{enumerate}
\item a {\em consistency} result, stating that sufficiently regular solutions to the continuously stratified system may be projected into $N$-dimensional valued functions satisfying the multi-layer shallow water systems  up to a remainder term of size $\mathcal{O}(1/N^2)$;
\item a suitable {\em stability} estimate on the linearized multi-layer shallow water systems, which is uniform with respect to $N$.
\end{enumerate}
The consistency result amounts to controlling the difference between the $N$-dimensional projection of the contribution of the Montgomery potential, $\frac1\varrho \partial_x\psi$ in \eqref{contsystem-intro}, and the corresponding contribution in \eqref{mltc-intro} when $(H_i)_{i\in\{1,\cdots,N\}}$ is the  $N$-dimensional projection of $h$. The $\mathcal{O}(1/N^2)$ estimate is obtained by exploiting Taylor expansions at suitable values of densities $\boldsymbol{\rho}_i$ in the spirit of the midpoint rule in numerical quadratures.

The stability estimate on \eqref{mltc-intro} relies on the energy method, for a carefully constructed energy functional. In order to obtain suitable bounds on the correct time scale, we make use of a partial symmetric structure of the hyperbolic system (when $\kappa=0$) and rely on the regularization effect of diffusivity only when necessary. This partial symmetric structure and the construction of the associated energy functional relies on a decomposition of the linear operator
\[ {\sf \Gamma}_N:(H_i)_{i\in\{1,\cdots,N\}} \mapsto \big(\sum\limits_{j=1}^N\frac{1}{N}\frac{\min(\boldsymbol{\rho}_i,\boldsymbol{\rho}_j)}{\boldsymbol{\rho}_i}H_j\big)_{i\in\{1,\cdots,N\}}\]
which echoes directly an analogous decomposition of continuous operator
\[ \mathcal M: h \mapsto \frac1{\varrho}\left(\rho_{\surf}\int\limits_{\rho_{\surf}}^{\rho_{\bott}}h(\varrho')d\varrho'+\int\limits_{\rho_{\surf}}^{\varrho}\int\limits_{\varrho'}^{\rho_{\bott}}h(\varrho'')d\varrho''d\varrho'\right)\]
and which is a key ingredient in the study of Bianchini and Duchêne. More generally, our energy estimates mirrors the ones in \cite{DB} after introducing a dictionnary between continuous operators and analogous discrete operators, with several adaptations to fit our framework.
\bigskip

\paragraph{Related works} To the best of our knowledge, this is the first time that the convergence between solutions to multi-layer systems and continuously stratified equations is rigorously proved (while formal connections are discussed for instance in the pioneering works of Benton~\cite{Benton53}, Su~\cite{Su76}, Killworth \cite{Killworth92}), with the notable exception of the work of Chen and Walsh~\cite{CW}. In the latter work, the authors consider the framework of periodic traveling waves without the hydrostatic approximation (and without diffusivity contributions). They prove that for sufficiently small periodic traveling wave solutions to the continuously stratified equations satisfying natural assumptions (in particular that the stratification is stable, sufficient regularity of the streamline density, and the absence of stagnation points), there exist corresponding traveling wave solutions associated with any piecewise smooth streamline density function (hence in particular solutions to multi-layer systems which correspond to piecewise constant streamline density functions, the velocity field being irrotational within each layer) in an $L^\infty$ neighborhood, and that the mapping from streamline density functions to the corresponding traveling wave is Lipschitz continuous in suitable functional spaces, the distance between two streamline densities being measured through the $L^\infty$ norm. 

This result, which is quite strong and versatile, does not directly compare with ours. Most importantly, the framework of traveling waves is of course very different from our framework allowing non-trivial dynamics. Roughly speaking, the traveling wave problem can be recast into a problem of elliptic nature, while our systems of equations are hyperbolic/parabolic. Consistently the tools used in our work, with the exception of a semi-Lagrangian change of coordinates to fix the fluid domain (called the Dubreil-Jacotin transformation in the context of traveling waves, which is equivalent to the isopycnal change of coordinates) are very different from the tools used in \cite{CW}. Moreover, the hydrostatic approximation also modifies the nature of the problem, as it discards dispersive effects. As such, nonlinear traveling waves are inexistent within the hydrostatic approximation framework, and could be replaced ---if one desires to study simplified dynamics--- with simple waves \cite{ChumakovaEtAl09,OstrovskyHelfrich11}. It would be very natural and interesting to extend our results to the non-hydrostatic framework. Let us notice however that multi-layer equations without the hydrostatic approximation suffer from Kelvin--Helmholtz instabilities, and hence require a regularizing mechanism (possibly offered by the Gent and McWilliams contributions) to allow for solutions in larger functional spaces than analytical functions \cite{SulemSulem85,IguchiTanakaTani97,KamotskiLebeau05,Lannes13}.

Finally, as discussed in \cite{CW}, the ``reverse'' limit consisting in approaching a bi-layer (or multi-layer) situation with continuously stratified problems has been studied (in the framework of traveling waves) in particular by James~\cite{James01}, after Turner and Amick~\cite{Turner81,AmickTurner86}. This limit is only apparently related to the limit considered here. Let us mention in particular that negative lower and upper bounds on the vertical derivative of the density function (through the non-cavitation assumption) are crucially used in our analysis as well as in \cite{DB}. We leave the study of this interesting problem to a later work.

\paragraph{Outline of the paper}
In section \ref{section2} we introduce all the notations and conventions used throughout this work. In section \ref{section3} we provide preliminary results, including detailed product and commutators estimates that will be used in the proofs of our main results.
Section \ref{section4} is dedicated to the proof of our first main result, namely the local well-posedness of our multi-layer systems on a time interval independent of  the number of layers, $N$. We first provide a preliminary well-posedness result on a short time interval (depending on $N$) in Proposition \ref{WPML}. The next step is to extract the quasilinear structure of the equations (Section \ref{section4.1}) and to provide energy estimates associated with the extracted linear equations (Section \ref{section4.2}). Finally, in Section \ref{section4.3}, we prove of our large-time well-posedness result, Theorem \ref{timeindepof}.
In section \ref{section5} we prove our second main result, Theorem \ref{mainth}, stating that solutions to the multi-layer system \eqref{mltc-intro} converge towards sufficiently regular solutions to the continuously stratified hydrostatic system \eqref{contsystem-intro} satisfying the non-cavitation assumption, with a $\mathcal{O}(1/N^2)$ convergence rate. As already mentioned, this results relies in part on a consistency result, obtained in Section \ref{section5.1}, and on the other part on stability estimates which we collect in Section \ref{section5.2}.

\section{Notations and conventions}\label{section2}

{\em In what follows, we use as convention that $g=1$ and $\rho_{\bott}-\rho_{\surf}=1$. This choice can be enforced without loss of generality through a suitable rescaling of the variables and unknowns. We also set $(\boldsymbol{\rho}_i)_{i\in\{1,\cdots,N\}}$ as
\[ \forall i\in\{1,\cdots,N\},\qquad \boldsymbol{\rho}_i=\rho_{\surf}+(i-\tfrac12) \frac{\rho_{\bott}-\rho_{\surf}}{N}.\]
Again this choice does not convey a restriction on admissible density profiles but only that we decide to discretize a continuous streamline density with a piecewise constant function with equidistributed values; see Figure \ref{dessin}. Notice however that the upper and lower bound assumption we later on impose on the depth variables do express negative lower and upper bounds on the vertical derivative of the density profiles we can consider.
}
\medskip

\begin{figure}[htb]
\begin{center}
\includegraphics[width=.7\textwidth]{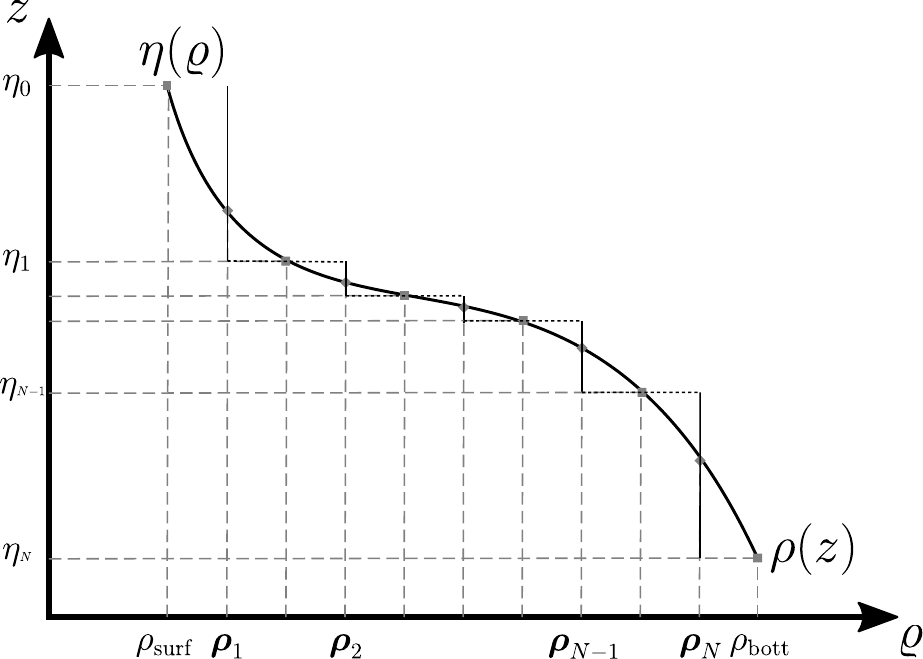}
\end{center}
\caption{A density profile and its piecewise constant approximation}
\label{dessin}
\end{figure}

We shall use matrix and vector formulations when dealing with the multi-layer system \eqref{mltc-intro}. We then use capital letters without indices for $N$-dimensional vectors, sans serif fonts for $N$-by-$N$ matrices, and indices to denote each components. 

For $F=(F_1,\cdots,F_N)^t$, $G=(G_1,\cdots,G_N)^t$ two vectors, we define the product $FG:=(F_1G_1,\cdots,F_NG_N)^t$,
and for a function $\varphi:\R\to\R$, $\varphi(F):=(\varphi(F_1),\cdots,\varphi(F_N))^t$. 

With these conventions, \eqref{mltc-intro} reads
\begin{equation}\label{mlct-notation}
	\left\{\begin{array}{l}
		\partial_tH+(\underline{U}+U)\partial_x H+(\underline{H}+H)\partial_xU=
		\kappa\partial_x^2 H,\\
		
		\partial_t U+\left(\underline{U}+U-\kappa\frac{\partial_xH}{\underline{H}+H}\right)\partial_xU+{\sf \Gamma}\partial_xH=0,
	\end{array}\right.
\end{equation}
with  $\underline{H}=(\underline{H}_1,\cdots,\underline{H}_N)^t$, $\underline{U}=(\underline{U}_1,\cdots,\underline{U}_N)^t$,
$H=(H_1,H_2,\cdots,H_N)^t$, $U=(U_1,U_2,\cdots,U_N)^t$, and ${\sf \Gamma}_{i,j}:=\frac1N\frac{\min(\boldsymbol{\rho}_i,\boldsymbol{\rho}_j)}{\boldsymbol{\rho}_i}$, for all $1\leq i,j\leq N$.

We now introduce matrix operations which will be used in our proofs.

\begin{itemize}
	\item Denote ${\sf Id}\in M_N(\R)$ the identity matrix, ${\sf P}:={\sf Diag}(1,0,\cdots,0)\in M_N(\R)$ the projection onto the first component, and ${\sf C }={\sf Id}-{\sf P}$. 
	
	\item We define the  operator ${\sf T}:=\sqrt{N}{\sf P}\in M_N(\R)$ which can be interpreted as the discrete analogue of the trace operator on the surface (acting on the density variable in the continuously stratified case).
	\item Let ${\sf S}\in M_N(\R)$ be the discrete integration operator defined as 
	\[{\sf S}:=\frac1N\begin{pmatrix} 
1 & 1 &\cdots& 1\\
0 & \ddots & \ddots & \vdots\\
\vdots&\ddots&\ddots&1\\
0&\cdots&0&1
\end{pmatrix}\]
We also define ${\sf S}_0\in M_{N,N-1}(\R)$ to be the matrix ${\sf S}$ without the last column.

\item Let $\Dp\in M_{N-1,N}(\R)$ be the discrete differentiation operator defined as
\[ \Dp:=N\begin{pmatrix}
1 & -1 & 0 &\cdots &0\\
0&   1 &-1&\ddots&\vdots\\\
\vdots&\ddots&\ddots&\ddots&0\\
0&\cdots & 0& 1& -1&\\
\end{pmatrix}.\] 
With an abuse of notation we will write  $\Ds=\Dr \Dp\in M_{N-2,N}(\R)$,  where $\Dr$ is the matrix $\Dp$ without the last line and last column.   Moreover we will use the convention $\Dp^0={\sf Id}$.

\item  Let ${\sf M}\in M_{N-1,N}(\R)$ be defined as
\[{\sf M}:=\frac{1}{2}\begin{pmatrix}
1 & 1 & 0 &\cdots &0\\
0&   1 &1&\ddots&\vdots\\\
\vdots&\ddots&\ddots&\ddots&0\\
0&\cdots & 0& 1& 1&\\
\end{pmatrix}.\]
  With an abuse of notation we will write  ${\sf M}^2=\widetilde{{\sf M}} {\sf M}\in M_{N-2,N}(\R)$,  where $\widetilde{{\sf M}}$ is the matrix ${\sf M}$ without the last line and last column, hence it is in the matrix space $ M_{N-2,N-1}(\R)$.   Moreover we will use the convention ${\sf M}^0={\sf Id}$.


\item Finally, we denote $\ru,\rd\in M_{N-1,N}(\R)$ the upwards and downwards reduction operators defined for $F=(F_1,\cdots,F_N)^t\in \R^N$, by 
\[\ru F=(F_2,\cdots,F_N)^t \quad \text{ and } \quad \rd F=(F_1,\cdots,F_{N-1})^t.\]
\end{itemize}

We shall map functions defined on $[\rho_{\surf},\rho_{\bott}]$ (associated with the continuously stratified framework) to $N$-dimensional vectors (associated with the multi-layer framework) with the following operator:
\[
P_N:	\begin{array}{ccc}
	\mathcal{C}([\rho_{\surf},\rho_{\bott}])& \rightarrow & \mathbb{R}^N\\
	f&\mapsto &\left(f(\boldsymbol{\rho}_i\right))_{1\leq i\leq N}
\end{array}
\]
\medskip

Let us now describe some functional spaces we use in this work.
\begin{itemize}
\item In order to describe tensored functional spaces for functions with variables in the strip $\Omega:=\R\times(\rho_{\surf},\rho_{\bott})$, we use the equivalent notations $H^s(\mathbb{R})=H^s_x$ (the usual $L^2$-based Sobolev space on $\mathbb{R}$) and $W^{k,\infty}(\R)=W^{k,\infty}_x$ (the $L^\infty$-based Sobolev space on $\mathbb{R}$), and similarly $L^2((\rho_{\surf},\rho_{\bott}))=L^2_\varrho$ and $W^{k,\infty}((\rho_{\surf},\rho_{\bott}))=W^{k,\infty}_\varrho$. We denote for instance 
\begin{equation*}
L^2_\varrho L^\infty_x=L^2(\rho_{\surf},\rho_{\bott};L^\infty(\R))=\{f : \text{ess sup}_{x\in\R}\br f(x,\cdot)\br\in L^2((\rho_{\surf},\rho_{\bott}))\}.
\end{equation*}
Notice $L^2_\varrho L^2_x=L^2_xL^2_\varrho=L^2(\Omega)$ and $L^\infty_\varrho L^\infty_x=L^\infty_xL^\infty_\varrho =L^\infty(\Omega)$. 

\item Let $s\in\R$, $k\in\N$, where $k\leq s$, we define the functional space 
\begin{equation*}
X^{\infty,s,k}(\Omega)=X^{\infty,s,k}_{\varrho,x}=\left\{f: \Omega \longrightarrow \R;\hspace{0.2cm}\forall j\in\N,\hspace{0.2cm} j\leq k,\hspace{0.2cm}  \partial_\varrho^jf\in L^\infty_\varrho(H^{s-j}(\R))\right\}, 
\end{equation*}
endowed with the topology of the norm
$$\n f\n_{\infty,s,k}^2=\sum\limits_{j=0}^k\n \partial_\varrho^jf\n_{L^\infty_\varrho(H^{s-j}_x)}^2.$$

\item For any $q\geq1$ we define endow $\R^N$ with the following normalized $l^q$ norms:
\[\br F\br_{l^q}:=\frac{\left(\sum\limits_{i=1}^N\br F_i\br^q\right)^\frac{1}{q}}{{N}^\frac{1}{q}} \quad \text{  and }\quad \br F\br_{l^\infty}:=\underset{i\in\{1,\cdots,N\}}{\sup}\br F_i\br\]
Moreover for $p,q\geq 1$, we introduce the following functional spaces
\begin{equation*}
L^p_x(l^q)=\{F=(F_1,\cdots,F_N)\in (L^p(\R^d))^N;\n\br F\br_{l^q}\n_{L^p_x}<\infty\}
\end{equation*}
\begin{equation*}
l^q(L^p_x)=\{F=(F_1,\cdots,F_N)\in (L^p(\R^d))^N;\br\left(\n F_i\n_{L^p_x}\right)_{1\leq i\leq N}\br_{l^q}<\infty\}
\end{equation*}
endowed with the following norms 
$\n F\n_{L^p_x(l^q)}:=\n\br F\br_{l^q}\n_{L^p_x}$ and $\n F\n_{l^q(L^p_x)}:=\br\left(\n F_i\n_{L^p_x}\right)_{1\leq i\leq N}\br_{l^q}$.
Notice that $\n F\n_{L^2_x(l^2)}=\n F\n_{l^2(L^2_x)}$.
\item Let $\big<\cdot,\cdot\big>_{L^2_xl^2}$ be the scalar product of the Hilbert space $l^2(L_x^2(\R))$ such that for $F,G\in l^2(L_x^2(\R))$ we have 
\begin{align*}
\big<F,G\big>_{L^2_xl^2}:=\int\limits_{\R}<F(x),G(x)>_{l^2}dx=\frac1N\sum_{i=1}^N \int\limits_{\R}F_i(x)G_i(x)dx.
\end{align*}

\item For $F\in\mathbb{R}^N$ and for $k\in\{0,1,2\}$ we define the following norm $\br F\br_{w^{k,\infty}}:=\sum\limits_{l=0}^k\br \Dp^{l} F\br_{l^\infty}$.

\item Let  $s\in\R$, $k\in\left\{0,1,2\right\}$, where $k\leq s$, we define the functional space 
\begin{equation*}
H^{s,k}(\R)=\left\{F: \R \longrightarrow \R^N;\hspace{0.2cm}\forall j\in\N,\hspace{0.2cm} j\leq k,\hspace{0.2cm} \Dp^j F\in H^{s-j}(\R)^N\right\}, 
\end{equation*}

endowed with the topology of the norm 
\begin{equation*}
\n F\n_{H^{s,k}}^2=\sum\limits_{j=0}^k\n \Dp^j F\n_{l^2(H^{s-j}_x)}^2.
\end{equation*}

\item We use standard notations for functions depending on time. For instance, for $k\in\N$ and $X$ a Banach space as above, $C^k([0,T];X)$ is the space of functions with values in $X$ which are continuously differentiable up to order $k$, and $L^p(0,T;X)$ the \textit{p}-integrable $X$ valued functions. All these spaces are endowed
with their natural norms.
\end{itemize}

We conclude with a few additional notations.
\begin{itemize}

\item For any $a,\,b\in\R$, we use the notation $a\lesssim b$ if there exists $C > 0$, independent of
relevant parameters, such that $a \leq Cb $.
\item We generically denote by $C(\cdots)$ some positive function that has a non decreasing dependence on its arguments.

\item We set 

\begin{equation*}
\left<B_a\right>_{a>b}=\left\{\begin{array}{l}0\quad \text{if}\,a\leq b,\\
B_a\quad\text{otherwise},\end{array}\right.
\end{equation*}

\item We denote $\Lambda=({\rm Id}-\partial_x^2)^{1/2}$, so that for all $s\geq0$,  $H^s(\R)=\{f\in L^2(\R); \ \Lambda^s f\in L^2(\R)\}$.

\item  Let ${\sf A}=(A_{ij})_{1\leq i,j\leq N}$ and $U=(U_1,\cdots,U_N)^t$ where $A_{ij},U_i$ are in a suitable functional spaces such that the linear operator $\textup{P}$ is well defined on that space (most of the time $\textup{P}=\Lambda^s$ with $s\geq0$ or $\textup{P}=\partial^\alpha_x$ with $\alpha\in\N$) we define the commutators
\begin{align*}
[\textup{P},{\sf A}]U &:=\textup{P}({\sf A}U)-{\sf A}\textup{P}(U)\\
[\textup{P};{\sf A},U]&:=\textup{P}({\sf A}U)-{\sf A}\textup{P}(U)-\textup{P}({\sf A})U.
\end{align*}

\end{itemize}

We conclude by introducing a new type of commutator adapted to the Leibnitz formula displayed in Lemma \ref{IL} below. 

\begin{definition}
Let  $A,B\in \R^N$, $j\in\{0,1,2\}$, we define, the following commutators:
\begin{align*}
\llbracket\partial_x^\alpha  \Dp^j,A\rrbracket B&:= \partial_x^\alpha( \Dp^j(AB))-  ({\sf M}^jA) (\partial^\alpha_x \Dp^j B),\\
 \llbracket\partial_x^\alpha  \Dp^j;A,B\rrbracket&:= \partial_x^\alpha(  \Dp^j(AB))-({\sf M}^jA)( \partial^\alpha_x  \Dp^j B)-({\sf M}^jB)( \partial^\alpha_x \Dp^j A).
\end{align*}
\end{definition}
Recalling the convention that  ${\sf M}^0={\sf Id}$, we have 
\begin{align*}
	\llbracket\partial_x^\alpha  \Dp^0,A\rrbracket B&= [\partial_x^\alpha, A]B,\\
	\llbracket\partial_x^\alpha  \Dp^0;A,B\rrbracket&= [\partial_x^\alpha;A,B].
\end{align*}

\section{Technical tools}\label{section3}
In this section we provide several preliminary results that will be used throughout this paper. In particular, we provide embeddings, product and commutator estimates adapted to our functional framework, and which mostly follow from standard estimates in Sobolev spaces. 
\begin{lemma}\label{IL}
Let $N\in\N^*$, $F,G\in\R^N$,  we have the following identities:
\begin{enumerate}
\item (Abel's summation)
\begin{equation*}
{\sf S}(FG)=F({\sf S}G)-{\sf S}_0 ((\Dp F) (\ru{\sf S}G)).\end{equation*}

\item (first order Leibnitz formula)
\begin{align*}
&\Dp(FG)=(\Dp F)({\sf M}G)+({\sf M}F)(\Dp G).
\end{align*}

\item (second order Leibnitz formula) 
\begin{equation*}
\Ds(FG)=(\Ds F) ({\sf M}^2G)+({\sf M}^2F)(\Ds G)+2({\sf M}\Dp F)({\sf M}\Dp G).
\end{equation*}
\end{enumerate}
\end{lemma}
\begin{proof}
\begin{enumerate}
\item 
Let $F=(F_1,\cdots,F_N)^t\in\R^N$, $G=(G_1,\cdots,G_N)^t\in\R^N$, the identity results from the fact that
\begin{align*}
\sum\limits_{j=i}^{N}F_jG_j-F_i\sum\limits_{k=i}^NG_k=-\sum\limits_{j=i}^{N-1}(F_j-F_{j+1})\sum\limits_{k=j+1}^NG_k,\quad \forall i\in\{1,\cdots,N-1\}.
\end{align*}

\item It results from the fact that \[(FG)_{i+1}-(FG)_i=\left(\frac{F_i+F_{i+1}}{2}\right)(G_{i+1}-G_i)+\left(\frac{G_i+G_{i+1}}{2}\right)(F_{i+1}-F_i).\]

\item It results directly from multiplying the first order Leibnitz formula by $\Dr$ and reapplying this formula a second time using the fact that $\Dr {\sf M}={\sf M}\Dp $.
\end{enumerate}

\end{proof}
\begin{lemma}\label{contop}
Let $N\in\N^*$, $F\in\R^N$. For any $p,q\in[1,\infty]$ the following estimates hold true:
\begin{align*}
\br F\br_{l^p}\leq \br F\br_{l^q},\, \text{when}\; p\leq q.
\end{align*}
\begin{align*}
&\n {\sf T}\n_{l^\infty\rightarrow l^2}\leq 1;\quad\n {\sf S}\n_{l^p\rightarrow l^q}\leq 1.
\end{align*}
The same continuity estimates hold for ${\sf S}_0$ instead of ${\sf S}$.

\end{lemma}

\begin{lemma}\label{A1}
Let $N\in\N^*$, $s\in\R$, $F\in  H^{s+\frac{1}{2},1}(\R)$. There exists $C>0$ independent of $N$ such that 
\begin{equation*}
\n\Lambda^s F\n_{l^\infty(L^2_x)}\leq C \n F\n_{H^{s+\frac{1}{2},1}}.
\end{equation*}
\end{lemma}

\begin{proof}

Let us assume that $N$ is even, and define $\phi:=(\overset{\frac{N}{2}}{\overbrace{1,\cdots,1}},\frac{N-2}{N},\cdots,\frac{4}{N},\frac{2}{N},0)^t\in\R^N$, hence $\Dp(\phi)=N(\overset{\frac{N}{2}-1}{\overbrace{0,\cdots,0}},\frac{2}{N},\cdots,\frac{2}{N})^t
$. Furthermore we notice that 
\begin{equation*}
\br \phi\br_{l^\infty}\leq 1,\quad \br \Dp(\phi)\br_{l^\infty}\leq 2.
\end{equation*}
Let $i\in\left\{1,\cdots,\frac{N}{2}\right\}$. Using Lemma \ref{contop} with Parseval equality we have
\begin{align*}
\n\Lambda^s F_i\n_{L^2_x}^2&=\int_{\R}\br\Lambda^s F_i\br^2\\
&=\int_{\R}  \left({\sf S}_0\Big(\Dp(\phi(\Lambda^s F)^2)\Big)\right)_i\\
&=\int_{\R} \left({\sf S}_0\Big(({\sf M}\phi)(\Dp((\Lambda^s F)^2)\Big)\right)_i+\int_{\R} \left({\sf S}_0\Big(({\sf M}(\Lambda^s F)^2)(\Dp(\phi))\Big)\right)_i\\
&=2\int_{\R} \left ({\sf S}_0\Big(({\sf M}\phi)(\Lambda^{s-\frac{1}{2}}\Dp(F))({\sf M}\Lambda^{s+\frac{1}{2}}F)\Big)\right)_i+\int_{\R} \left({\sf S}_0\Big(({\sf M}(\Lambda^s F)^2)(\Dp(\phi))\Big)\right)_i\\
&\leq 2\br\phi\br_{l^\infty}\int_{\R} \br(\Lambda^{s-\frac{1}{2}}(\Dp(F)))( \Lambda^{s+\frac{1}{2}}F)\br_{l^1}+\int_{\R} \br{\sf M}(\Lambda^s F)^2\br_{l^1}\br\Dp(\phi)\br_{l^\infty}\\
&\lesssim \n F\n_{H^{s+\frac{1}{2},1}}^2.
\end{align*}

Let $\tilde{F}_i:=F_{N-i+1}$ for $i\in\left\{1,\cdots,N\right\}$ and from what proceeds and for $i\in\left\{1,\cdots,\frac{N}{2}\right\}$ we have \begin{align*}
\n\Lambda^s \tilde{F}_i\n_{L^2_x}^2\lesssim \n\Lambda^{s+\frac{1}{2}}\tilde{F}\n_{l^2(L^2_x)}^2+ \n\Lambda^{s-\frac{1}{2}}( \Dp \tilde{F})\n_{l^2(L^2_x)}^2.
\end{align*}
Hence the result follows immediately since $\n\Lambda^{s-\frac{1}{2}}( \Dp\tilde{F})\n_{l^2(L^2_x)}=\n\Lambda^{s-\frac{1}{2}}( \Dp F)\n_{l^2(L^2_x)}$ and\\
$\n\Lambda^{s+\frac{1}{2}}\tilde{F}\n_{l^2(L^2_x)}=\n\Lambda^{s+\frac{1}{2}}F\n_{l^2(L^2_x)}$. We can easily adapt this proof when $N$ is odd.
\end{proof}
\begin{lemma}\label{propD}
Let $N\in\N^*$, $F\in\R^N$, $h_*>0$ and $\underline{H}\in\R^N$ such that $\underset{i\in\{1,\cdots,N\}}{\inf}\underline{H}_i\geq h_*$. Then there exists $C(h_*^{-1})$ independent of $N$ such that
\begin{enumerate}

\item $\br\frac{\Dp}{N}F\br_{l^2}\leq 2\br F\br_{l^2}$ and $\br\frac{\Ds}{N}F\br_{l^2}\leq 2 \br\Dp F\br_{l^2}$.

\item $\br\Ds\frac{1}{\underline{H}}\br_{l^\infty}\leq C(h_*^{-1})\left(\br\Ds\underline{H}\br_{l^\infty}+\br\Dp\underline{H}\br^2_{l^\infty}\right).$
\end{enumerate}
\end{lemma}

\begin{proof}
\begin{enumerate}

\item It follows immediately from the fact that $\frac{\Dp}{N}F=(F_1-F_2,\cdots,F_{N-1}-F_N)^t$ and that $\frac{\Ds}{N}F=N(F_1-2F_2+F_3,\cdots,F_{N-2}-2F_{N-1}+F_{N})^t$.

\item We have for $i\in \{1,\cdots,N-2\}$
\begin{align*}
\left(\Ds\frac{1}{\underline{H}}\right)_i&=N^2\left(
\frac{1}{\underline{H}_i}-\frac{2}{\underline{H}_{i+1}}+\frac{1}{\underline{H}_{i+2}}\right)
\\&=N^2\left(\frac{2(\underline{H}_{i+2}-\underline{H}_{i+1})(\underline{H}_{i+1}-\underline{H}_{i})-\underline{H}_{i+1}(\underline{H}_{i+2}-2\underline{H}_{i+1}+\underline{H}_{i})}{\underline{H}_{i}\underline{H}_{i+1}\underline{H}_{i+2}}\right)
\end{align*}
and the result follows immediately.
\end{enumerate}
\end{proof}
\begin{lemma}\label{comnatu}
Let $N\in\N^*$, $t_0>\frac{1}{2}$.

\begin{enumerate}
\item  For any $s,\,s_1,\,s_2\in\R$ such that $s_1\geq s$, $s_2\geq s$, $s_1+s_2\geq 0$ and $s_1+s_2\geq s+t_0$, there exists $C>0$ independent of $N$ such that for any $(p,q)\in\{(2,\infty),(\infty,2)\}$ and any $F=(F_1,\cdots,F_N)^t\in \left(H^{s_1}(\R^d)\right)^N$ and $G=(G_1,\cdots,G_N)^t\in \left(H^{s_2}(\R^d)\right)^N$, $FG\in l^2(H^s(\mathbb{R}^d))$ and
\begin{align*}
\n FG\n_{l^2(H^s)}\leq C \n F\n_{l^p(H^{s_1}_x)}\n G\n_{l^q(H^{s_2}_x)}\leq \left\{\begin{array}{l}
C \n F\n_{H^{s_1+\frac{1}{2},1}}\n G\n_{H^{s_2,0}}\\
\text{or}\\
C \n F\n_{H^{s_1,0}}\n G\n_{H^{s_2+\frac{1}{2},1}}
\end{array}
\right.
\end{align*}

\item For any $s\geq -t_0$, there exists $C>0$ independent of $N$ such that for any $(p,q),(\tilde{p},\tilde{q})\in\{(2,\infty),(\infty,2)\}$ and any $F=(F_1,\cdots,F_N)^t\in \left(H^{s}(\R^d)\right)^N$ and $G=(G_1,\cdots,G_N)^t\in \left(H^{s}(\R^d)\right)^N\cap\left(H^{t_0}(\R^d)\right)^N$, $FG\in l^2(H^s(\mathbb{R}^d))$ and
\begin{align*}
\n FG\n_{l^2(H^s_x)}&\leq C\n F\n_{l^{p}(H^{t_0}_x)}\n G\n_{l^{q}(H^{s}_x)}+C\left<\n F\n_{l^{\tilde{p}}(H^{s}_x)}\n G\n_{l^{\tilde{q}}(H^{t_0}_x)}\right>_{s>t_0}\\
&\leq C\left\{\begin{array}{l}
\n F\n_{H^{\max(t_0,s)+\frac{1}{2},1}}\n G\n_{H^{s,0}}
\\ \text{or}
\\\n F\n_{H^{\max(t_0+\frac{1}{2},s),1}}\n G\n_{H^{\max(t_0+\frac{1}{2},s),1}}
\end{array}
\right.    
\end{align*}

\item For any $s\in\R$ and $s_1$, $s_2\in\R$ such that $s_1\geq s$, $s_2\geq s-1$ and $s_1+s_2\geq s+t_0$, there exists $C>0$ independent of $N$ such that for any $(p,q)\in\{(2,\infty),(\infty,2)\}$ and any $F=(F_1,\cdots,F_N)^t\in \left(H^{s_1}(\R^d)\right)^N$ and $G=(G_1,\cdots,G_N)^t\in \left(H^{s_2}(\R^d)\right)^N$, $[\Lambda^s,F]G\in\left(L^2(\R^d)\right)^N$ and
\begin{align*}
\n[\Lambda^s,F]G\n_{l^2(L^2_x)}&\leq C \n \Lambda^{s_1}F\n_{l^p(L^2_x)}\n
\Lambda^{s_2}G\n_{l^q(L^2_x)}\leq\left\{\begin{array}{l}
C \n F\n_{H^{s_1+\frac{1}{2},1}}\n G\n_{H^{s_2,0}}\\
\text{or}\\
C \n F\n_{H^{s_1,0}}\n G\n_{H^{s_2+\frac{1}{2},1}}
\end{array}
\right.
\end{align*}
\item For any $s\geq0$, there exists $C>0$ independent of $N$ such that for any $(p,q),(\tilde{p},\tilde{q})\in\{(2,\infty),(\infty,2)\}$ and any $G=(G_1,\cdots,G_N)^t\in \left(H^{s-1}(\R^d)\right)^N $ and any $F=(F_1,\cdots,F_N)^t\in \left(L^\infty(\R^d)\right)^N$ with $\partial_x F_i\in H^{s-1}(\R^d)\cap H^{t_0}(\R^d)$, one has $[\Lambda^s,F]G\in\left(L^2(\R^d)\right)^N$ , and
\begin{align*}
\n\br[\Lambda^s,F]G\br_{l^2}\n_{L^2_x}&\leq C \n \Lambda^{t_0}(\partial_x F)\n_{l^p(L^2_x)}\n
\Lambda^{s-1}G\n_{l^q(L^2_x)}\\&\hspace{1cm}+C\left<\n\Lambda^{s-1}(\partial_x F)\n_{l^{\tilde{p}}(L^2_x)}\n
\Lambda^{t_0}G\n_{l^{\tilde{q}}(L^2_x)}\right>_{s>t_0+1}\\
&\leq C\left\{\begin{array}{l}
\n \partial_xF\n_{H^{\max(t_0+\frac{1}{2},s-\frac{1}{2}),1}}\n G\n_{H^{s-1,0}}
\\ \text{or}
\\\n \partial_xF\n_{H^{\max(t_0+\frac{1}{2},s-1),1}}\n G\n_{H^{\max(s-1,t_0+\frac{1}{2}),1}}
\end{array}
\right.
\end{align*}

\item For any $s\geq0$, there exists $C>0$ independent of $N$ such that for any $(p,q),(\tilde{p},\tilde{q})\in\{(2,\infty),(\infty,2)\}$ and for any $F=(F_1,\cdots,F_N)^t$, $
G=(G_1,\cdots,G_N)^t$ with $F_i,G_i\in H^{s}(\R^d)\cap H^{t_0+1}(\R^d) $ for all $i\in\{1,\cdots,N\}$, one has $[\Lambda^s;F,G]\in \left(L^2(\mathbb{R}^d)\right)^N$ and
\begin{align*}
\n\br [\Lambda^s;F,G]\br_{l^2}\n_{L^2_x}&\leq C\n \Lambda^{t_0+1}F\n_{l^p(L^2_x)}\n\Lambda^{s-1}G
\n_{l^q(L^2_x)}+\n\Lambda^{s-1}F\n_{l^{\tilde{p}}(L^2_x)}\n\Lambda^{t_0+1}G
\n_{l^{\tilde{q}}(L^2_x)}\\
&\leq C\left\{\begin{array}{l}
\n F\n_{H^{\max(t_0+\frac{3}{2},s-\frac{1}{2}),1}}\n G\n_{H^{\max(s-1,t_0+1),0}}
\\ \text{or}
\\\n F\n_{H^{\max(t_0+\frac{3}{2},s-1),1}}\n G\n_{H^{\max(t_0+\frac{3}{2},s-1),1}}
\end{array}
\right.
\end{align*}
The result holds true if we replace $\Lambda^s$ by $\partial_x^\alpha$ for $0\leq\alpha\leq s$.

\end{enumerate}

\end{lemma}
\begin{proof}
\begin{enumerate}
\item The proof results immediately from the following classical product estimate in Sobolev spaces for scalar function $f,g$
\begin{equation}
\n fg\n_{H^s_x}\leq C \n f\n_{H^{s_1}_x}\n g\n_{H^{s_2}_x},
\end{equation}
then conclude with Lemma \ref{A1}.
\item  The proof results immediately from the following classical tame estimate for products in Sobolev spaces and for scalar function $f,g$
\begin{equation}
\n fg\n_{H^s_x}\leq C\n f\n_{H^{t_0}_x}\n g\n_{H^{s}_x}+C\left<\n f\n_{H^{s}_x}\n g\n_{H^{t_0}_x}\right>_{s>t_0},
\end{equation}
then conclude with Lemma \ref{A1}.
\item We fix $(p,q)\in\{(2,\infty),(\infty,2)\}$ and we have $[\Lambda^s;F]G=\left([\Lambda^s,F_i]G_i\right)_i$, hence the proof results immediately using the following classical commutator estimate in Sobolev spaces for scalar functions $f,g$
\begin{align}
\n[\Lambda^s,f]g\n_{L^2_x}
&\lesssim \n \Lambda^{s_1}f\n_{L^2_x}\n
\Lambda^{s_2}g\n_{L^2_x},
\end{align}
then conclude with Lemma \ref{A1}.

\item The proof results immediately from using the following classical tame estimate for commutators in Sobolev spaces for scalar functions $f,g$
\begin{equation}
\n[\Lambda^s,f]g\n_{L^2_x}\leq C \n \Lambda^{t_0}(\partial_x f)\n_{L^2_x}\n
\Lambda^{s-1}g\n_{L^2_x}+C\left<\n\Lambda^{s-1}(\partial_x f)\n_{L^2_x}\n
\Lambda^{t_0}g\n_{L^2_x}\right>_{s>t_0+1},
\end{equation}
then conclude with Lemma \ref{A1}.

\item We have $[\Lambda^s;F,G]=\left(\Lambda^s(F_iG_i)-\Lambda^s(F_i)G_i-F_i\Lambda^s(G_i)\right)_i$, the result follows immediately by using the following classical estimate for the symmetric commutator and for scalar functions $f,g$, where $[\Lambda^s;f,g]:= \Lambda^s(fg) - f\Lambda^s g-g\Lambda^s f \in L^2(\R)$ 
\begin{align}\label{symopcommu}
\n [\Lambda^s ; f, g]\n_{L^2_x}\leq C \n f\n_{H^{t_0+1}_x} \n g\n_{H^{s-1}_x}+ C\n f\n_{H^{s-1}_x} \n g\n_{H^{t_0+1}_x},
\end{align}
then conclude by Lemma \ref{A1}.
\end{enumerate}
\end{proof}

\begin{lemma}\label{AlgebreBanach}

Let $N\in\N^*$, $t_0 >\frac{1}{2}$ then we have the following
\begin{enumerate}
\item Let $s\in\N$ such that $s \geq t_0 + \frac{1}{2}$, then there exists $C$ independent of $N$ such that for any $F, G\in H^{s,k}$, $k\in\left\{1,2\right\}$ we have
\begin{align*}
\n FG\n_{H^{s,k}}\leq C\n F\n_{H^{s,k}}\n G\n_{H^{s,k}}.
\end{align*}

\item  Let $s\in\N$ such that $s\geq t_0+\frac{3}{2}$, then there exists $C$ independent of $N$ such that for any $F,G\in H^{s,2}$ it holds
\begin{equation*}
\n FG\n_{H^{s,2}} \leq C\left(\n  F\n_{H^{s,2}} \n  G\n_{H^{s-1,1}} +\n  F\n_{H^{s-1,1}} \n  G\n_{H^{s,2}}\right).
\end{equation*}
\end{enumerate}
\end{lemma}
\begin{proof}

\begin{enumerate}
\item Let $F,G\in H^{s,1} $ then by Lemma \ref{IL}
\begin{align*}
\n FG\n^2_{H^{s,1}}&= \n \Lambda^s (FG)\n^2_{l^2(L^2_x)}+\n \Lambda^{s-1}(\Dp (FG))\n^2_{l^2(L^2_x)}\\
&\lesssim  \n \Lambda^s (FG)\n^2_{l^2(L^2_x)}+\n\Lambda^{s-1}\left(({\sf M}F)(\Dp G)\right)\n^2_{l^2(L^2_x)}+\n\Lambda^{s-1}\left(({\sf M}G)(\Dp F) \right)\n^2_{l^2(L^2_x)}.
\end{align*}

Since $s\geq t_0+\frac{1}{2}$, and using Lemma \ref{comnatu} $(2)$ we have
\begin{align*}
\n \Lambda^s (FG)\n_{l^2(L^2_x)}\lesssim \n F\n_{H^{s,1}}\n G\n_{H^{s,1}},
\end{align*}
and
\begin{align*}
\n\Lambda^{s-1}(({\sf M}F)\Dp G)\n_{l^2(L^2_x)}
&\lesssim \n F\n_{H^{s,1}}\n G\n_{H^{s,1}}.
\end{align*}
Symmetrically $\n\Lambda^{s-1}(({\sf M}G)(\Dp F )\n_{l^2(L^2_x)}\lesssim\n F\n_{H^{s,1}}\n G\n_{H^{s,1}}$. Consequently $H^{s,1}$ is a Banach algebra.\\

Let $F,G\in H^{s,2}$ then $\n FG\n^2_{H^{s,2}}= \n  FG\n^2_{H^{s,1}}+\n \Lambda^{s-2}(\Ds (FG))\n^2_{l^2(L^2)}$. Using the fact that $H^{s,1}$ is a Banach algebra we have 
\begin{align*}
\n  FG\n_{H^{s,1}}\lesssim \n  F\n_{H^{s,2}}\n  G\n_{H^{s,2}}.
\end{align*}

By Lemma \ref{IL}
\begin{equation*}
\Ds(FG)=({\sf M}^2F)(\Ds G)+({\sf M}^2G)(\Ds F)+2({\sf M} \Dp F)({\sf M}\Dp G).
\end{equation*}

Then by Lemma \ref{comnatu} $(2)$
\begin{align*}
\n \Lambda^{s-2}(({\sf M}^2F)(\Ds G))\n_{l^2(L^2_x)}\lesssim\n  F\n_{H^{s,1}}\n  G\n_{H^{s,2}}.
\end{align*}

Symmetrically we deduce that $\n\Lambda^{s-2}(({\sf M}^2G)(\Ds F))\n_{l^2(L^2_x)}\lesssim\n  F\n_{H^{s,2}}\n  G\n_{H^{s,1}}$. Moreover by Lemma \ref{comnatu} $(1)$ we have
\begin{align*}
\n \Lambda^{s-2}(({\sf M}\Dp G)({\sf M}\Dp F ))\n_{l^2(L^2_x)}
& \lesssim\n  F\n_{H^{s,2}}\n  G\n_{H^{s,2}}.
\end{align*}

\item The proof follows the steps seen previously in $1$ with this time using Lemma \ref{A1} and Lemma \ref{comnatu}  and choosing the convenient values of $p,\,q,\,\tilde{p},\, \tilde{q}$, in addition relying on the fact that $s\geq t_0+\frac{3}{2}$ when estimating.
\end{enumerate}  

\end{proof}

\begin{lemma}\label{comutgen}

Let $N\in\N^*$, $t_0 >\frac{1}{2}$, then the following holds

\begin{enumerate}
\item Let $ s\in\N$ with $s \geq t_0 + \frac{3}{2}$, then there exists $C>0$ independent of $N$ such that 
\begin{itemize}
\item For any $\alpha\in\N$, with $0\leq\alpha\leq s$
\begin{align*}
\n  [\partial_x^\alpha, F]G\n_{l^2(L^2_x)}\leq  C\n  F\n_{H^{s,2}} \n  G\n_{H^{s-1,1}}.
\end{align*}
The result holds true if we replace $\partial_x^\alpha$  by $\Lambda^{s}$.
\item For any $j\in\left\{1,2\right\}$ and $\alpha\in\N$ such that $0\leq\alpha\leq s-j$ 
\begin{align*}
\n  \llbracket\partial_x^\alpha  \Dp^j, F\rrbracket G\n_{l^2(L^2_x)}\leq C  \n F\n_{H^{s,2}} \n  G\n_{H^{s-1,j}}.
\end{align*}
The result holds true if we replace $\partial_x^\alpha$  by $\Lambda^{s-j}$.
\end{itemize}
\item Let $ s\in\N$ with $s > 2+\frac{1}{2}$, then there exists $C>0$ independent of $N$ such that for $\alpha\in\N$ with $0\leq\alpha\leq s-2$ we have
\begin{align*}
\n  \llbracket\partial_x^\alpha  \Dp, F\rrbracket G\n_{l^2(L^2_x)}\leq C\n F\n_{H^{s,2}}\n G\n_{H^{s-2,1}}.
\end{align*}
The result holds true if we replace $\partial_x^\alpha$  by $\Lambda^{s-2}$.
\item  Let $ s\in\N$ with $s \geq t_0 + \frac{5}{2}$, then there exists $C>0$ independent of $N$ such that for any $j\in\left\{1,2\right\}$ and $\alpha\in\N$ such that $0\leq\alpha\leq s-j$ then
\begin{align*}
& \n  \llbracket\partial_x^\alpha  \Dp^j, F,G\rrbracket\n_{l^2(L^2_x)}\leq C  \n F\n_{H^{s-1,2}} \n  G\n_{H^{s-1,2}}.
\end{align*}
The result holds true if we replace $\partial_x^\alpha$  by $\Lambda^{s-j}$.
\end{enumerate}
\end{lemma}

\begin{proof}~\\
\begin{enumerate}
\item  Using Lemma \ref{comnatu} $(4)$ we have 
\begin{align*}
\n  [\partial_x^\alpha, F]G\n_{l^2(L^2_x)}\leq C\left(\n \partial_x F\n_{H^{t_0+\frac{1}{2},1}}\n G\n_{H^{s-1,0}}+\n \partial_x F\n_{H^{s-1,0}}\n G\n_{H^{t_0+\frac{1}{2},1}}\right).
\end{align*}
For $j=1$ and $0\leq\alpha\leq s-1$, we have  by Lemma \ref{IL} and Lemma \ref{comnatu} ($(2)$ and $(4)$)
\begin{align*}
\llbracket\partial_x^\alpha&\Dp,F\rrbracket G= \partial_x^\alpha(\Dp(FG))-({\sf M}F)( \partial^\alpha_x(\Dp G))\\
&=\partial_x^\alpha\left(({\sf M}F)(\Dp G)+({\sf M}G)(\Dp F)\right)-({\sf M}F)( \partial^\alpha_x(\Dp G))\\
&=[\partial_x^\alpha,{\sf M}F]\Dp G+\partial^\alpha_x\left(({\sf M}G)(\Dp F)\right)\\
&\lesssim \n\partial_x{\sf M}F\n_{H^{\max(t_0+\frac{1}{2},s-\frac{3}{2}),1}}\n\Dp G\n_{H^{s-2,0}}+\n \Dp F\n_{H^{\max(t_0+\frac{1}{2},s-1),1}}\n{\sf M}G\n_{H^{\max(t_0+\frac{1}{2},s-1),1}}.
\end{align*}

For $j=2$ and $0\leq\alpha\leq s-2$, we have by Lemma \ref{IL} we have  by Lemma \ref{IL} and Lemma \ref{comnatu} ($(2)$ and $(4)$)
\begin{align*}
\llbracket\partial_x^\alpha&\Ds,F\rrbracket G= \partial_x^\alpha(\Ds(FG))-({\sf M}^2F)(\partial^\alpha_x(\Ds G))\\
&= [\partial_x^\alpha,{\sf M}^2F]\Ds G+2\partial_x^\alpha\left(({\sf M}\Dp G)({\sf M}\Dp F)\right)+\partial_x^\alpha\left((\Ds F)({\sf M}^2G)\right)\\
&\lesssim \n \partial_x{\sf M}^2F\n_{H^{\max(t_0+\frac{1}{2},s-\frac{5}{2}),1}} \n \Ds G\n_{H^{s-3,0}} +\n{\sf M}\Dp F\n_{H^{\max(t_0,s-2)+\frac{1}{2},1}}\n {\sf M}\Dp G\n_{H^{s-2,0}}\\
& \hspace{0,5cm}+ \n \Ds F\n_{H^{s-2,0}} \n {\sf M}^2G\n_{H^{\max(t_0,s-2)+\frac{1}{2},1}}. 
\end{align*}
\item The result follows from the fact that 
\begin{equation*}
\llbracket\partial_x^\alpha\Dp,F\rrbracket G=[\partial_x^\alpha,{\sf M}F]\Dp G+\partial^\alpha_x\left(({\sf M}G)(\Dp F)\right),
\end{equation*}
and using Lemma \ref{comnatu} $(2)$ and $(4)$ and choosing the corresponding $(p,q)\in\{2,\infty\}$ with $t_0=s-\frac{3}{2}$, then using Lemma \ref{A1}.
\item For $j=1$ and $0\leq\alpha\leq s-1$,we have  by Lemma \ref{IL} and Lemma \ref{comnatu} $(4)$ 
\begin{align*}
\llbracket\partial_x^\alpha\Dp;F,G\rrbracket&= \partial_x^\alpha(\Dp(FG))-({\sf M}F)( \partial^\alpha_x\Dp G))-({\sf M}G)(\partial^\alpha_x\Dp F)\\
&=\partial_x^\alpha\left(({\sf M}F)(\Dp G)\right)+\partial_x^\alpha\left(({\sf M}G)(\Dp F)\right)-({\sf M}F)(\partial^\alpha_x\Dp G)-({\sf M}G)( \partial^\alpha_x\Dp F)\\
& =[\partial_x^\alpha,{\sf M}F]\Dp G+[\partial_x^\alpha,{\sf M}G]\Dp F\\
&\lesssim \n\partial_x{\sf M}F\n_{H^{\max(t_0+\frac{1}{2},s-2),1}}\n\Dp G\n_{H^{\max(t_0+\frac{1}{2},s-2),1}}\\
&\hspace{5cm}+ \n\partial_x{\sf M}G\n_{H^{\max(t_0+\frac{1}{2},s-2),1}}\n\Dp F\n_{H^{\max(t_0+\frac{1}{2},s-2),1}}.
\end{align*}

For $j=2$ and $0\leq\alpha\leq s-2$, we have by Lemma \ref{IL} and Lemma \ref{comnatu} ($(2)$ and $(4)$)
\begin{align*}
\llbracket\partial_x^\alpha&\Ds;F,G\rrbracket= \partial_x^\alpha(\Ds(FG))-({\sf M}^2F)( \partial^\alpha_x\Ds G)-({\sf M}^2G)( \partial^\alpha_x\Ds F)\\
&=\partial_x^\alpha\left((\Ds F) ({\sf M}^2G)+(\Ds G)( {\sf M}^2F)+2({\sf M}\Dp  F)({\sf M}\Dp G)\right)\\
&\hspace{1cm}-({\sf M}^2F)( \partial^\alpha_x\Ds G)-({\sf M}^2G)(\partial^\alpha_x\Ds F)\\
&=[\partial_x^\alpha,{\sf M}^2G]\Ds F+[\partial_x^\alpha,{\sf M}^2F]\Ds G+2\partial_x^\alpha\left(({\sf M}\Dp  F)({\sf M}\Dp G)\right)\\
&\lesssim \n\partial_x{\sf M}^2G\n_{H^{\max(t_0+\frac{1}{2},s-3),1}}\n\Ds F\n_{H^{s-3,0}}+ \n\partial_x{\sf M}^2F\n_{H^{\max(t_0+\frac{1}{2},s-3),1}}\n\Ds G\n_{H^{s-3,0}}\\
&\hspace{3cm}+ \n{\sf M}\Dp F\n_{H^{\max(t_0+\frac{1}{2},s-2),1}}\n{\sf M}\Dp G\n_{H^{\max(t_0+\frac{1}{2},s-2),1}}.
\end{align*}

\end{enumerate}
This concludes the proof.
\end{proof}

\begin{lemma}\label{composHsk}
Let $h_*, \underline{M}, M^*>0$, $t_0>\frac{1}{2}$ and $s\in\N$ such that  $s\geq t_0+\frac{3}{2}$, there exists $C(h_*^{-1},\underline{M},M^*)>0$,  such that for any $N\in\N^*$ and any $\underline{H}\in w^{2,\infty}$, $H\in H^{s,2}(\R)$,  satisfying
\begin{equation*}
\br\underline{H}\br_{w^{2,\infty}}\leq \underline{M};\quad \n H\n_{H^{s-1,1}}\leq M^*,
\end{equation*}
\begin{equation*}
\underset{(x,i)\in\R\times\{1,\cdots,N\}}{\inf}\underline{H}_i+H_i(x)\geq h_*,
\end{equation*} 
the following holds 
\begin{equation}
\left\|\frac{1}{\underline{H}+H}-\frac{1}{\underline{H}}\right\|_{H^{s,2}}\leq C(h_*^{-1},\underline{M},M^*)\n H\n_{H^{s,2}}. 
\end{equation}
\begin{equation}
\left\|\frac{1}{\underline{H}+H}-\frac{1}{\underline{H}}\right\|_{H^{s-1,1}}\leq C(h_*^{-1},\underline{M},M^*)\n H\n_{H^{s-1,1}}. 
\end{equation}
\end{lemma}

\begin{proof}
For $i\in\{1,\cdots,N\}$ fixed, since $s$ is sufficiently large we notice by Sobolev injection that $H_i$ vanishes at infinity as a consequence we have $\underline{H}_i\geq h_*$. Let $\left(\frac{1}{\underline{H}+H}-\frac{1}{\underline{H}}\right)_i=\frac{1}{\underline{H}_i+H_i}-\frac{1}{\underline{H}_i}=\varphi_i(H_i)$, where $\varphi_i\in \mathcal{C}^\infty(\R;\R)$ such that $\varphi_i(x)=\frac{1}{\underline{H}_i+x}-\frac{1}{\underline{H}_i}$ when $\underline{H}_i+x\geq h_*$, moreover we notice  that there exists $C(\alpha,h_*^{-1})$(independent of i) such that $\n\varphi_i^{(\alpha)}\n_{L^\infty([h_*-\underline{H}_i,+\infty[)}\leq C(\alpha,h_*^{-1}) $, using the fact that $\underline{H}\in w^{2,\infty}$ we can choose $\varphi_i$ in such a way that there exists $C(\alpha,h_*^{-1})$ (independent of $i$) such that $\n\varphi_i^{(\alpha)}\n_{L^\infty(]-\infty,h_*-\underline{H}_i))}\leq C(\alpha,h_*^{-1})$. Hence using the composition Lemma in Sobolev spaces  (see Lemma A.4 in \cite{DB}) and Lemma \ref{A1} we have
\begin{align}
& \left\|\Lambda^s\left(\frac{1}{\underline{H}+H}-\frac{1}{\underline{H}}\right)\right\|_{l^2(L^2_x)}\leq C(h_*^{-1},\n H\n_{l^\infty(H^{t_0}_x)})\n H\n_{H^{s,0}}\label{composition1}.\\
&\left\|\Lambda^{s-1}\left(\frac{1}{\underline{H}+H}-\frac{1}{\underline{H}}\right)\right\|_{l^2(L^2_x)}\leq C(h_*^{-1},\n H\n_{l^\infty(H^{t_0}_x)})\n H\n_{H^{s-1,0}}\label{composition2}.
\end{align}
By Lemma \ref{IL} and Lemma \ref{propD} we find that
\begin{align*}
&\left\|\frac{1}{\underline{H}+H}-\frac{1}{\underline{H}}\right\|_{H^{s,2}}\leq C(h_*^{-1},\br\underline{H}\br_{w^{2,\infty}})\left\|\frac{H}{\underline{H}+H}\right\|_{H^{s,2}}.\\
&\left\|\frac{1}{\underline{H}+H}-\frac{1}{\underline{H}}\right\|_{H^{s-1,1}}\leq C(h_*^{-1},\br\underline{H}\br_{w^{1,\infty}})\left\|\frac{H}{\underline{H}+H}\right\|_{H^{s-1,1}}.
\end{align*}
We have 
for $i\in\{1,\cdots,N-1\}$
\begin{align*}
&\left( \Dp\left(\frac{H}{\underline{H}+H}\right)\right)_i=\left[-N(\underline{H}_i-\underline{H}_{i+1})\left(\frac{H_i+H_{i+1}}{2}\right)+N(H_i-H_{i+1})\left(\frac{\underline{H}_i+\underline{H}_{i+1}}{2}\right)\right]\\
&\hspace{6cm}\times \frac{1}{(\underline{H}_i+H_i)(\underline{H}_{i+1}+H_{i+1})}.
\end{align*}
then
\begin{align*}
\Dp\left(\frac{H}{\underline{H}+H}\right)=\left[-(\Dp\underline{H})({\sf M}H)+(\Dp H)({\sf M}\underline{H})\right]\rd\left(\frac{1}{\underline{H}+H}\right)\ru\left(\frac{1}{\underline{H}+H}\right). 
\end{align*}

We have for $i\in\{1,\cdots,N-2\}$
\begin{align*}
\left(\Dp\left(\rd\left(\frac{1}{\underline{H}+H}\right)\ru\left(\frac{1}{\underline{H}+H}\right)\right)\right)_i=\frac{-2\left({\sf M}\Dp(\underline{H}+H)\right)_i}{(\underline{H}_i+H_i)(\underline{H}_{i+1}+H_{i+1})(\underline{H}_{i+2}+H_{i+2})}.
\end{align*}

and hence 

\begin{align*}
\Ds\left(\frac{H}{\underline{H}+H}\right)=&\bigg[-(\Ds\underline{H})({\sf M}^2H)+(\Ds H)({\sf M}^2\underline{H})\bigg]{\sf M}\left(\rd\left(\frac{1}{\underline{H}+H}\right)\ru\left(\frac{1}{\underline{H}+H}\right)\right)\\
&-2\left[-{\sf M}((\Dp\underline{H})({\sf M}H))+{\sf M}((\Dp H)({\sf M}\underline{H}))\right]({\sf M}\Dp(H+\underline{H}))\\
&\times\rd\rd\left(\frac{1}{\underline{H}+H}\right)\rd\ru\left(\frac{1}{\underline{H}+H}\right)\ru\ru\left(\frac{1}{\underline{H}+H}\right).
\end{align*}

Using the previous equalities and the following arguments

\begin{itemize}
\item The spatial derivatives commute with the operators $\rd,\, \ru,\,,\, {\sf M},\,\Dp,\,\Ds$. 
\item The operators $\rd,\, \ru,\, {\sf M}$ are bounded in $l^2$ and $l^\infty$.
\item Lemma \ref{AlgebreBanach} and estimates  \eqref{composition1} and  \eqref{composition2}.
\end{itemize}
We obtain
\begin{equation*}
\left\|\frac{1}{\underline{H}+H}-\frac{1}{\underline{H}}\right\|_{H^{s,2}}\leq C(h_*^{-1},\br\underline{H}\br_{w^{2,\infty}},\n H\n_{H^{s-1,1}})\n H\n_{H^{s,2}}. 
\end{equation*}
\begin{equation*}
\left\|\frac{1}{\underline{H}+H}-\frac{1}{\underline{H}}\right\|_{H^{s-1,1}}\leq C(h_*^{-1},\br\underline{H}\br_{w^{1,\infty}},\n H\n_{H^{s-1,1}})\n H\n_{H^{s-1,1}}. 
\end{equation*}
Hence the proof is then completed since the constant $C$ is increasing in its arguments.
\end{proof}

We conclude this section by collecting continuity estimates on the operator $P_N$ defined by
\[
P_N:	\begin{array}{ccc}
	\mathcal{C}([\rho_{\surf},\rho_{\bott}])& \rightarrow & \mathbb{R}^N\\
	f&\mapsto &\left(f(\boldsymbol{\rho}_i\right))_{1\leq i\leq N}
\end{array}
\]
\begin{lemma}\label{op P_N}
Let $s\in\N$, $s\geq 2$ there exists $C>0$ such that for any  $N\in\N^*$ and any $f\in X^{\infty,s,2}(\Omega)$ the following holds
\begin{align*}
&\n P_Nf\n_{H^{s,0}}\leq C\n f\n_{\infty,s,0};\quad \n P_Nf\n_{H^{s,1}}\leq C \n f\n_{\infty,s,1}\quad
\n P_Nf\n_{H^{s,2}}\leq C\n f\n_{\infty,s,2}.
\end{align*}
\end{lemma}
\begin{proof}
It is immediate that
\begin{equation*}
\n P_Nf\n_{H^{s,0}}^2\leq \n f\n_{L^\infty_{\varrho}(H^s_x)}^2.
\end{equation*}
For $i\in\{1,\cdots N\}$, using the mean value theorem  
\begin{align*}
\int_{\R}\br N(\Lambda^{s-1}f(x,\boldsymbol{\rho}_i)-\Lambda^{s-1}f(x,\boldsymbol{\rho}_{i+1})\br^2dx&\leq \n\partial_\varrho f\n_{L^\infty_\varrho({H^{s-1}_x})}^2,
\end{align*}
hence $\n\Lambda^{s-1}\Dp P_Nf\n_{l^2(L_x^2)}\lesssim \n\partial_\varrho f\n_{L^\infty_\varrho({H^{s-1}_x})}$, and
$\n P_Nf\n_{H^{s,1}}\lesssim (\n f\n_{L^\infty_{\varrho}(H^s_x)}+\n\partial_\varrho f\n_{L^\infty_\varrho({H^{s-1}_x})})$.
Using the Taylor expansion for the map $\varrho\longmapsto\Lambda^{s-2}f(\cdot,\varrho)$  
we have 
\begin{align*}
\int_{\R}\br N^2(\Lambda^{s-2}f(x,\boldsymbol{\rho}_i)-2\Lambda^{s-2}f(x,\boldsymbol{\rho}_{i+1})+\Lambda^{s-2}f(x,\boldsymbol{\rho}_{i+2}))\br^2dx&\leq 2\n\partial_\varrho^2\Lambda^{s-2}f\n_{L^\infty_\varrho(L^2_x)}^2,
\end{align*}
consequently,
\begin{align*}
\n\Lambda^{s-2}\Ds P_Nf\n_{l^2(L_x^2)}\leq 2\n\partial_\varrho^2\Lambda^{s-2}f\n_{L^\infty_\varrho(L^2_x)}.
\end{align*}
This concludes the proof.
\end{proof}

\section{Large-time well-posedness}\label{section4}

This section is dedicated to the proof of our first main result, Theorem \ref{timeindepof}, concerning the well-posedness of the initial-value problem for the multilayer system \eqref{mltc-intro}. Let us for convenience rewrite the system here, using conventions and notations introduced in Section \ref{section2}.
\begin{equation}\label{multicouches}
	\left\{\begin{array}{l}
		\partial_tH+(\underline{U}+U)\partial_x H+(\underline{H}+H)\partial_xU=
		\kappa\partial_x^2 H,\\
		
		\partial_t U+\left(\underline{U}+U-\kappa\frac{\partial_xH}{\underline{H}+H}\right)\partial_xU+{\sf \Gamma}\partial_xH=0,
	\end{array}
	\right.
\end{equation}
where we recall that for all $i\in \{1,\cdots,N-1\}$, $\boldsymbol{\rho}_{i+1}-\boldsymbol{\rho}_i=\frac1N$ and ${\sf \Gamma}_{i,j}=\frac1N\frac{\min(\boldsymbol{\rho}_i,\boldsymbol{\rho}_j)}{\boldsymbol{\rho}_i}$.

We first state in Proposition \ref{WPML} below the easy-to-prove local well-posedness of multi-layer system on a ``short'' period of time, that is {\em a priori} vanishing as the number of considered layers $N$ goes to infinity. The goal in this section is to improve this result by obtaining a time of existence uniform with respect to $N$. In order to achieve this goal we shall rely on the energy method. We extract the quasilinear structure of the equations in Section \ref{section4.1} and provide useful estimates on the extracted linearized equations in Section \ref{section4.2}, while the completion of the proof of Theorem \ref{timeindepof} is achieved in Section \ref{section4.3}. 

A crucial ingredient in our energy estimates, dictating our choice of the energy functional, is the following decomposition
\begin{equation}\label{rhogammadecomp}
\boldsymbol{\rho}{\sf \Gamma}=\boldsymbol{\rho}_1({\sf TS})^t{\sf TS}+{\sf S}^t{\sf C}{\sf S}
\end{equation}
(recall Section \ref{section2} for the definition of the matrices ${\sf C}$, ${\sf S}$ and ${\sf T}$). This decomposition mimics an analogous one of the operator
\[ \varrho\mathcal M: h \mapsto \rho_{\surf}\int\limits_{\rho_{\surf}}^{\rho_{\bott}}h(\varrho')d\varrho'+\int\limits_{\rho_{\surf}}^{\varrho}\int\limits_{\varrho'}^{\rho_{\bott}}h(\varrho'')d\varrho''d\varrho',\]
used by Bianchini and Duchêne in \cite{DB} to obtain the analogous well-posedness result for the continously stratified system \eqref{contsystem-intro}.

The following result concerns the well-posedness of the shallow water multi-layer with the $\kappa$ regularization system \eqref{multicouches}, where the time existence of the solution depends on the number of layers $N$. The strategy of the proof is fairly standard, and we only sketch the main steps.

\begin{proposition}\label{WPML}
Let $N\in\mathbb{N}^*$, $0<\kappa\leq 1$, $ s>1+\frac{1}{2}$, $ 0<h^*<1$, $ M_0>0$. There exists  $T_{N}>0$
such that for all $(H_{0},U_{0})\in \left(H^s(\mathbb{R})\right)^N\times \left(H^s(\mathbb{R})\right)^N$  such that  $\|H_{0}\|_{H^{s,0}}+\|U_{0}\|_{H^{s,0}}\leq M_0$  and  $H_{0}+\underline{H}\geq h^{*}$, there exists a unique $(H,U)$ solution to \eqref{multicouches} with $(H_{0},U_{0})\vert_{t=0}=(H_{0},U_{0})$ such that
\[H\in C([0,T_{N}];H^s(\mathbb{R})^N)\cap L^2([0,T_{N}];H^{s+1}(\mathbb{R})^N),\quad 
U\in C([0,T_{N}];H^s(\mathbb{R})^N)
\]
and satisfying 
\begin{equation}\label{noncavii}
\underset{{i\in\{1,\cdots,N\}}}{\inf}H_i+\underline{H}_i\geq h^*/2\quad \text{on}\quad [0,T_{N}]\times  \mathbb{R}.
\end{equation}

\end{proposition}

\begin{proof}[Sketch of the proof] The solution can be constructed through a Picard’s iterative scheme, considering \eqref{multicouches} as transport diffusions equations (on the variable $H$) and transport equations (on the variable $U$), coupled through order-zero terms. Specifically, we define inductively $(H^n,U^n)$ for $n\in\N$ as the solutions to the decoupled equations on each $i-$th layer for $i\in\{1,\cdots,N\}$ as follows
\begin{equation}\label{decoupled}
\left\{\begin{array}{l}
\partial_tH_i^{n+1}+(\underline{U}_i^{n}+U_i^{n})\partial_xH_i^{n+1}-\kappa\partial^2_xH_i^{n+1}=f_i,\quad f_i=-(\underline{H}_i^{n}+H_i^{n})\partial_x U_i^{n},\\ 
\partial_tU_i^{n+1}+\left(\underline{U}_i^{n}+U_i^{n}-\kappa\frac{\partial_xH_i^{n}}{\underline{H}_i^{n}+H_i^{n}}\right)\partial_xU_i^{n+1}=g_i,\quad g_i=-\sum\limits_{j=1}^N{\sf \Gamma}_{i,j}\partial_xH_j^n,\\
{H_i^{n+1}}_{|t=0}={H_{i,0}}\\         
{U_i^{n+1}}_{|t=0}=U_{i,0}.
\end{array}
\right.
\end{equation}

One can find a corresponding time $T_{N}$, depending on $s,N,\kappa,h^*$ and $M_0$ but not on $n$ such that the sequence $(H^n,U^n)$ is uniquely defined on $[0,T_N]$ by \eqref{decoupled}, remains in a ball (in the functional space corresponding to the regularity of $(H,U)$ displayed in the theorem endowed with its natural topology) and satisfies the  non cavitation assumption \eqref{noncavii}. Moreover one shows that it is actually a Cauchy sequence in a weaker functional space, hence using weak convergence and interpolation in Sobolev spaces we can pass to the limit $n\to \infty $ in the equations and obtain the solution $(H,U)$ to \eqref{multicouches}. We infer consequently the continuity in time with respect to the strong topology, as well as uniqueness of solutions at this level of regularity. All these steps are justified using the well-posedness theory of both transport and transport diffusion equations (see for instance chapter 3 of \cite{BCD}), relying in particular on the energy estimates which are displayed in the forthcoming Lemma \ref{esttransdiff,trans}.
\end{proof}

\subsection{Quasilinearization}\label{section4.1}

In this subsection we focus on linearizing system \eqref{multicouches}, this is done by applying the operators ${\sf S},\,{\sf T},\, \partial_x,\, \Dp$ to the equations so as to obtain linear equations satisfied by the derivatives of our unknowns $H,U$.

\begin{lemma}\label{restes}
Let  $s\in\N$ such that $s>2+\frac{1}{2}$, and $M,\underline{M},h_*>0$. There exists ${C=C(s,M,\underline{M},h_*)>0}$, such that for any $N\in\N^*$, $\kappa>0$ and any $(\underline{H},\underline{U})\in w^{2,\infty}$ such that
\begin{equation*}
\br\boldsymbol{\rho}\br_{l^\infty}  +\br\boldsymbol{\rho}^{-1}\br_{l^\infty} + \br\underline{H}\br_{w^{2,\infty}}+\br\Dp\underline{U}\br_{w^{1,\infty}}\leq \underline{M},
\end{equation*}
and any $(H,U)\in C([0,T];H^{s,2}(\R))$ solution to $(\ref{multicouches})$ with some $T>0$ and satisfying \begin{equation*}
\n H(t,\cdot)\n_{H^{s-1,1}}+ \n {\sf S}H(t,\cdot)\n_{H^{s,1}}+ \n {\sf T}{\sf S}H(t,\cdot)\n_{H^{s,0}}+ \n U(t,\cdot)\n_{H^{s,2}}+\kappa^\frac{1}{2} \n H(t,\cdot)\n_{H^{s,2}}\leq M 
\end{equation*}

for all $t\in[0,T]$ and 
\begin{equation*}
\underset{(t,x,i)\in(0,T)\times\R\times\{1,\cdots,N\}}{\inf}\underline{H}_i+H_i(t,x)\geq h_*, \, \text{ the following holds.}
\end{equation*}
\begin{itemize}
\item For all $\alpha\in\N$, $j\in\{0,1\}$ with $0\leq\alpha\leq s-1-j$ , we have
\begin{align*}
\partial_t\partial_x^\alpha\Dp^{j}H+({\sf M}^{j}(\underline{U}+U))\partial_x\partial_x^\alpha\Dp^jH=\kappa\partial_x^2\partial_x^\alpha\Dp^{j}H+\mathcal{R}_{\alpha,j}
\end{align*}
where for every $t\in[0,T]$, $\mathcal{R}_{\alpha,j}(t,\cdot)\in l^2(L^2(\R))$ and $$\n\mathcal{R}_{\alpha,j}(t,\cdot)\n_{l^2(L^2_x)}\leq CM.$$

\item For all $\alpha\in\N$ with $0\leq\alpha\leq s$, we have
\begin{align*}
&\partial_t\partial_x^\alpha H+(\underline{U}+U)\partial_x\partial_x^\alpha H
+(\underline{H}+H)\partial_x\partial_x^\alpha U=\kappa\partial_x^2\partial_x^\alpha H+R_{\alpha,0}\\&
\partial_t\partial_x^\alpha U+{\sf \Gamma}\partial_x\partial_x^\alpha H+\left(\underline{U}+U-\kappa\frac{\partial_xH}{\underline{H}+H}\right)\partial_x \partial_x^\alpha U=\textbf{R}_{\alpha,0}
\end{align*}
where for every $t\in[0,T]$, $({\sf S} R_{\alpha,0}(t,\cdot),{\sf T}{\sf S} R_{\alpha,0}(t,\cdot),\textbf{R}_{\alpha,0}(t,\cdot))\in l^2(L^2(\R))^2\times l^2(L^2(\R))$ and $$\n {\sf S} R_{\alpha,0}(t,\cdot)\n_{l^2(L^2_x)}+\n {\sf T}{\sf S} R_{\alpha,0}(t,\cdot)\n_{l^2(L^2_x)}+\n \textbf{R}_{\alpha,0}(t,\cdot)\n_{l^2(L^2_x)}\leq C M(1+\kappa\n \partial_xH\n_{H^{s,2}}).$$

\item For any $\alpha\in \N$, $j\in\{1,2\}$, such that $0\leq\alpha\leq s-j$, it holds
\begin{align*}
&\partial_t \Dp^j {\sf S}\partial_x^\alpha H+\rd({\sf M}^{j-1}(\underline{U}+U))\partial_x\left(\Dp^j {\sf S}\partial_x^\alpha H\right)=\kappa\partial_x^2( \Dp^j {\sf S}\partial_x^\alpha H)+R_{\alpha,j}\\
&\partial_t\Dp^j\partial_x^\alpha U+ \left({\sf M}^j\left(\underline{U}+U-\kappa\frac{\partial_x H}{\underline{H}+H}\right)\right)\partial_x (\Dp^j\partial_x^\alpha U)=\textbf{R}_{\alpha,j}
\end{align*}
where for every $t\in[0,T]$, $( R_{\alpha,j}(t,\cdot),\textbf{R}_{\alpha,j}(t,\cdot))\in l^2(L^2(\R))\times l^2(L^2(\R))$ and 
\begin{align*}
\n R_{\alpha,j}(t,\cdot)\n_{l^2(L^2_x)}+ \n \textbf{R}_{\alpha,j}(t,\cdot)\n_{l^2(L^2_x)}\leq C M(1+\kappa\n \partial_xH\n_{H^{s,2}}).
\end{align*}

\item For any $\alpha\in \N$, $j\in\{0,1,2\}$, $0\leq\alpha\leq s-j$ , it holds
\begin{align*}
&\partial_t \Dp^j\partial_x^\alpha H+\left({\sf M}^j(\underline{U}+U)\right)\partial_x( \Dp^j\partial_x^\alpha H)=\kappa\partial_x^2(  \Dp^j\partial_x^\alpha H)+r_{\alpha,j}+\partial_x \textbf{r}_{\alpha,j}
\end{align*}

where for every $t\in[0,T]$, $( r_{\alpha,j}(t,\cdot),\textbf{r}_{\alpha,j}(t,\cdot))\in l^2(L^2(\R))\times l^2(L^2(\R))$ and 
\begin{align*}
\kappa^\frac{1}{2}\n r_{\alpha,j}(t,\cdot)\n_{l^2(L^2_x)}+ \n \textbf{r}_{\alpha,j}(t,\cdot)\n_{l^2(L^2_x)}\leq C M.
\end{align*}  

\end{itemize}

\end{lemma}
\begin{proof} ~\\

\noindent Estimates for $\mathcal{R}_{\alpha,0}=-[\partial_x^\alpha,\underline{U}+U]\partial_xH-\partial_x^\alpha\left((\underline{H}+H)\partial_xU\right)$, with $0\leq\alpha\leq s-1$.\\

\noindent Since $s-1>\frac{1}{2}+\frac{1}{2}$, and using  Lemma \ref{comnatu} ($(2)$ and $(3)$) and Lemma \ref{A1} we have
\begin{align*}
\n\partial_x^\alpha\left(H\,\partial_xU\right)\n_{l^2(L^2_x)}\lesssim\n H\n_{H^{s-1,0}} \n U\n_{H^{s,1}}+ \n H\n_{H^{s-1,1}} \n U\n_{H^{s,0}},
\end{align*}
and
\begin{align*}
\n[\partial_x^\alpha,U]\partial_xH\n_{L^2(L^2_x)}
&\lesssim\n U\n_{H^{s,1}}\n H\n_{H^{s-1,0}},
\end{align*}
moreover 
\begin{align*}
\n\partial_x^\alpha\left(\underline{H}\,\partial_xU\right)\n_{l^2(L^2_x)}\lesssim \br\underline{H}\br_{l^\infty}\n U\n_{H^{s,0}}.
\end{align*}
Hence \begin{align}\label{Rest1}
\n\mathcal{R}_{\alpha,0}\n_{l^2(L^2_x)}\lesssim\left(\br\underline{H}\br_{l^\infty}+\n H\n_{H^{s-1,1}}\right)\n U\n_{H^{s,1}}.
\end{align}

\noindent Estimates for $\mathcal{R}_{\alpha,1}=-\llbracket\partial_x^\alpha\Dp,\underline{U}+U\rrbracket\partial_xH -\partial_x^\alpha\Dp((\underline{H}+H)\partial_xU)$, with $0\leq\alpha\leq s-2$.\\

\noindent By Lemma \ref{IL} we have
\begin{align*}
\partial_x^\alpha\Dp((\underline{H}+H)\partial_xU)=&\partial_x^\alpha\left((\Dp\underline{H})\,({\sf M}\partial_xU)\right)+\partial_x^\alpha\left((\Dp H)\,({\sf M}\partial_xU)\right)+\partial_x^\alpha\left(({\sf M}\underline{H})\,(\Dp\partial_xU)\right)
\\&
+\partial_x^\alpha\left(({\sf M}H)\,(\Dp\partial_xU)\right).
\end{align*}
Using  Lemma \ref{comnatu} $(1)$ and since $s-2>\frac{1}{2}$ we have
\begin{align*}
\n\partial_x^\alpha\left((\Dp H)\,({\sf M}\partial_xU)\right)\n_{l^2(L^2_x)}&\lesssim \n\Dp H\n_{H^{s-2,0}} \n{\sf M}\partial_xU\n_{H^{s-\frac{3}{2},1}}\\
& \lesssim \n H\n_{H^{s-1,1}}\n U\n_{H^{s,1}},
\end{align*}
and
\begin{align*}
\n\partial_x^\alpha\left(({\sf M}H)\,(\Dp\partial_xU)\right)\n_{l^2(L^2_x)}&\lesssim \n{\sf M}H\n_{H^{s-\frac{3}{2},1}} \n\Dp\partial_x U\n_{H^{s-2,0}}\\
& \lesssim \n H\n_{H^{s-1,1}}\n U\n_{H^{s,1}}.
\end{align*}
Using Lemma \ref{comutgen} $(2)$  and since $s-2>\frac{1}{2}$  we have
\begin{align*}
\n\llbracket\partial_x^\alpha\Dp,U\rrbracket\partial_xH \n_{l^2(L^2_x)}
& \lesssim\n U\n_{H^{s,2}} \n H\n_{H^{s-1,1}},
\end{align*}
once again with Lemma \ref{IL}
\begin{align*}
\n\llbracket\partial_x^\alpha\Dp,\underline{U}\rrbracket\partial_xH \n_{l^2(L^2_x)}&=\n\partial_x^\alpha((\Dp\underline{U})({\sf M}\partial_x H)\n_{l^2(L^2_x)}\leq \br\Dp\underline{U}\br_{l^\infty} \n H\n_{H^{s-1,0}}.
\end{align*}
We immediately have 
\begin{align*}
\n\partial_x^\alpha\left((\Dp\underline{H})\,({\sf M}\partial_xU)\right)\n_{l^2(L^2_x)}&\leq \br\Dp\underline{H}\br_{l^\infty}\n U\n_{H^{s-1,0}},
\end{align*}
and
\begin{align*}
\n\partial_x^\alpha\left(({\sf M}\underline{H})\,(\Dp\partial_xU)\right)\n_{l^2(L^2_x)}&\leq \br\underline{H}\br_{l^\infty}\n U\n_{H^{s,1}}.
\end{align*}
Hence \begin{align}\label{Rest2}
\n\mathcal{R}_{\alpha,1}\n_{l^2(L^2_x)}\lesssim \left(\br\underline{H}\br_{w^{1,\infty}}+\n H\n_{H^{s-1,1}}\right)\n U\n_{H^{s,2}}+\br\Dp\underline{U}\br_{l^\infty} \n H\n_{H^{s-1,0}}.
\end{align}

We have  
$R_{\alpha,0}=-[\partial_x^\alpha, U]\partial_xH-[\partial_x^\alpha,H]\partial_x U$ and $\textbf{R}_{\alpha,0}= -\left[\partial_x^\alpha,\underline{U}+U-\kappa\frac{\partial_xH}{\underline{H}+H}\right]\partial_xU$. We notice that 
\begin{align*}
{\sf S}R_{\alpha,0}=-{\sf S}\big([\partial_x^\alpha;U,\partial_xH]+(\partial_x^\alpha U)(\partial_xH) \big)-{\sf S}\big([\partial_x^\alpha;H,\partial_x U]+(\partial_x^\alpha H)(\partial_x U)\big).
\end{align*}
On the other hand using Lemma \ref{IL} and since ${\sf S}_0,\,{\sf S},\,\ru,\,\Dp$ commute with $\partial_x$ we have
\begin{align*}
-{\sf S}[\partial_x^\alpha;U,&\partial_xH]=-\bigg({\sf S}\partial_x^\alpha\left( U\partial_xH\right)-{\sf S}(U\partial_x\partial_x^\alpha H)-{\sf S}\left((\partial_x^\alpha U)(\partial_xH)\right)\bigg)\\
&\hspace{0,8cm}=-\bigg(\partial_x^\alpha \left(U\partial_x{\sf S}H\right)-\partial_x^\alpha\left( {\sf S}_0((\Dp(U))(\ru\partial_x{\sf S}H))\right)-U\partial_x{\sf S}\partial_x^\alpha H\\
&\hspace{2cm}+{\sf S}_0((\Dp(U)(\ru\partial_x{\sf S}\partial_x^\alpha H))-(\partial_x^\alpha U)(\partial_x{\sf S}H)
+{\sf S}_0((\Dp(\partial_x^\alpha U))(\ru\partial_x{\sf S}H))\bigg)
\\&\hspace{0,8cm}=-[\partial_x^\alpha;U,\partial_x{\sf S}H]+{\sf S}_0[\partial_x^\alpha;\Dp(U),\ru\partial_x{\sf S}H]
\end{align*}
and
\begin{align*}
-{\sf S}((\partial_x^\alpha H)(\partial_x U))&=-\bigg((\partial_x U)({\sf S}\partial_x^\alpha H)-{\sf S}_0\big((\Dp(\partial_x U))(\ru({\sf S}\partial_x^\alpha H)\big)\bigg),
\end{align*}
and finally we find 
\begin{align*}
{\sf S}&R_{\alpha,0} =-[\partial_x^\alpha;U,\partial_x{\sf S}H]-
(\partial_x U)({\sf S}\partial_x^\alpha H)+{\sf S}_0\bigg([\partial_x^\alpha;\Dp(U),\ru\partial_x{\sf S}H]+(\Dp(\partial_x U))(\ru{\sf S}\partial_x^\alpha H)\bigg)\\
&\hspace{1cm}-{\sf S}\bigg((\partial_x^\alpha U)(\partial_xH)
+[\partial_x^\alpha;H,\partial_x U]
\bigg).
\end{align*}

Since $\sqrt{N}{\sf P}[\partial_x^\alpha;U,\partial_x{\sf S}H]=[\partial_x^\alpha;U,\partial_x{\sf T}{\sf S}H]$, we immediately recover a similar equation for the trace case
\begin{align*}
{\sf T}&{\sf S}R_{\alpha,0} =-[\partial_x^\alpha;U,\partial_x{\sf T}{\sf S}H]-
(\partial_x U)({\sf T}{\sf S}\partial_x^\alpha H)+{\sf T}{\sf S}_0\bigg([\partial_x^\alpha;\Dp(U),\ru\partial_x{\sf S}H]+(\Dp(\partial_x U))(\ru{\sf S}\partial_x^\alpha H)\bigg)\\
&\hspace{1,5cm}-{\sf T}{\sf S}\bigg((\partial_x^\alpha U)(\partial_xH)
+[\partial_x^\alpha;H,\partial_x U]
\bigg).\end{align*}
Estimates for ${\sf S}R_{\alpha,0}$, ${\sf T}{\sf S}R_{\alpha,0}$  with $0\leq\alpha\leq s$ .\\

\noindent Using Lemma \ref{A1} and since $s>\frac{1}{2}+\frac{3}{2}$, we have
\begin{align*}
\n (\partial_xU)({\sf S}\partial_x^\alpha H)\n_{l^2(L^2_x)}\leq\n\partial_xU\n_{l^\infty(L^\infty_x)}\n {\sf S}\Lambda^s H\n_{l^2(L^2_x)}\lesssim \n {\sf S}H
\n_{H^{s,0}} \n U
\n_{H^{s,1}},
\end{align*}
and  
\begin{align*}
\n (\partial_xU)({\sf T}{\sf S}\partial_x^\alpha H)\n_{l^2(L^2_x)}\leq\n\partial_xU\n_{l^\infty(L^\infty_x)}\n {\sf T}{\sf S}\Lambda^s H\n_{l^2(L^2_x)} \lesssim  \n {\sf T}{\sf S}H
\n_{H^{s,0}}\n U
\n_{H^{s,1}}.
\end{align*}
Using Lemma \ref{comnatu} $(6)$ and since $s-2>\frac{1}{2}$ we obtain 
\begin{align*}
\n[\partial_x^\alpha;U,\partial_x{\sf S}H]\n_{l^2(L^2_x)}&\lesssim \n U\n_{H^{\max(s-2+\frac{3}{2},s-\frac{1}{2}),1}}\n\partial_x{\sf S}H\n_{H^{\max(s-1,s-2+1),0}}\\
&\lesssim\n U
\n_{H^{s,1}}\n {\sf S}H
\n_{H^{s,0}},
\end{align*}
and 
\begin{align*}
\n[\partial_x^\alpha;U,\partial_x{\sf T}{\sf S}H]\n_{l^2(L^2_x)}&\lesssim\n U\n_{H^{\max(s-2+\frac{3}{2},s-\frac{1}{2}),1}}\n\partial_x{\sf T}{\sf S}H\n_{H^{\max(s-1,s-2+1),0}}\\
&\lesssim\n U
\n_{H^{s,1}}\n {\sf T}{\sf S}H
\n_{H^{s,0}}.
\end{align*}
Moreover using Lemma \ref{contop}, and estimate \eqref{symopcommu} and since $s-1>\frac{1}{2}+1$ we have 

\begin{align*}
\n\br[\partial_x^\alpha;\Dp(U),\ru\partial_x{\sf S}H]\br_{l^1}\n_{L^2_x}
&\lesssim \n\Lambda^{s-1}(\Dp(U))\n_{l^2(L^2_x)}\n\Lambda^{s-2+1}(\partial_x{\sf S}H)\n_{l^2(L^2_x)}\\
&\hspace{2,5cm}+\n\Lambda^{s-2+1}(\Dp(U))\n_{l^2(L^2_x)}\n\Lambda^{s-1}(\partial_x{\sf S}H)\n_{l^2(L^2_x)}\\
&\lesssim\n U\n_{H^{s,1}}\n {\sf S}H\n_{H^{s,0}},
\end{align*}
and  
\begin{align*}
\n \br[\partial_x^\alpha;H,\partial_xU]\br_{l^1}\n_{L^2_x}
&\lesssim  \n \Lambda^{s-1}H\n_{l^2(L^2_x)}\n \Lambda^{s-2+1}\left(\partial_xU\right)\n_{l^2(L^2_x)}\\
&\hspace{5cm}+\n\Lambda^{s-1}\left(\partial_xU\right)\n_{l^2(L^2_x)}\n
\Lambda^{s-2+1}H\n_{l^2(L^2_x)}
\\
&\lesssim\n H\n_{H^{s-1,0}}\n U
\n_{H^{s,0}}.\end{align*}
Using Lemma \ref{contop} and classical product estimates in Sobolev spaces and the fact that and since $s-1>\frac{1}{2}+1$ we have  
\begin{align*}
\n  \br (\partial_x^\alpha U)(\partial_xH)\br_{l^1}\n_{L^2_x}
&\lesssim \n\Lambda^s\left( U\right)\n_{l^2(L^2_x)} \n\Lambda^{s-2}\left(\partial_xH\right)\n_{l^2(L^2_x)}
\\
& \lesssim\n U
\n_{H^{s,0}}\n H
\n_{H^{s-1,0}}, \end{align*}
and 

\begin{align*}
\n \br(\Dp(\partial_x U))(\ru{\sf S}\partial_x^\alpha H)\br_{l^1}\n_{L^2_x}&\lesssim \n\Lambda^{s-2+1}(\Dp U)\n_{l^2(L^2_x)}
\n \Lambda^s {\sf S}H\n_{l^2(L^2_x)}\\
&\lesssim \n U\n_{H^{s,1}}\n {\sf S}H\n_{H^{s,0}}.
\end{align*}

Using the fact that $\n{\sf S}_0\n_{l^1\to l^2}\leq 1$, $\n{\sf T S}_0\n_{l^1\to l^2}\leq 1$, $\n{\sf S}\n_{l^1\to l^2}\leq 1$ (Lemma \ref{contop}) we can deduce  estimates of the remaining terms of ${\sf T}{\sf S}R_{\alpha,0}$ follow directly from the previous estimates and moreover  we have \begin{equation}\label{Rest3}
\n {\sf S}R_{\alpha,0}\n_{l^2(L^2_x)}+\n {\sf T}{\sf S}R_{\alpha,0}\n_{l^2(L^2_x)}\leq \left( \n H\n_{H^{s-1,0}}+ \n {\sf S}H\n_{H^{s,0}}+\n {\sf T}{\sf S}H\n_{H^{s,0}}\right)\n U\n_{H^{s,1}}.
\end{equation}
Estimate for $\textbf{R}_{\alpha,0}$ with $0\leq\alpha\leq s$.
\begin{align*}
\textbf{R}_{\alpha,0}=-[\partial_x^\alpha,U]\partial_xU+\kappa\frac{1}{\underline{H}+H}[\partial_x^\alpha,\partial_xH]\partial_xU+\kappa\left[\partial_x^\alpha,\frac{1}{\underline{H}+H}\right](\partial_xH)(\partial_x U).
\end{align*}
Using Lemma \ref{comnatu} $(4)$ and since $s-\frac{3}{2}>\frac{1}{2}$, we have
\begin{align*}
\n[\partial_x^\alpha,U]\partial_xU\n_{l^2(L^2_x)}&\lesssim \n\partial_xU\n_{H^{s-1,1}}\n\partial_xU\n_{H^{s-1,1}}\\
& \lesssim \n U\n_{H^{s,1}}^2.
\end{align*}
Using Lemma \ref{comnatu} $(4)$ and since $s-\frac{3}{2}>\frac{1}{2}$, we have
\begin{align*}
\left\|\kappa\left[\partial_x^\alpha,\frac{1}{\underline{H}+H}\right](\partial_xH)(\partial_x U)\right\|_{l^2(L^2_x)}
&\lesssim\kappa \bigg(\left\|\Lambda^{s-\frac{3}{2}}\left(\partial_x\frac{1}{\underline{H}+H}\right)\right\|_{l^\infty(L^2_x)}\n\Lambda^{s-1}((\partial_xH)(\partial_xU))\n_{l^2(L^2_x)}\\
&\hspace{0,5cm}+\n\Lambda^{s-\frac{3}{2}}\left((\partial_xH)(\partial_xU)\right)\n_{l^\infty(L^2_x)}
\left\|\Lambda^{s-1}\left(\partial_x\frac{1}{\underline{H}+H}\right)\right\|_{l^2(L^2_x)}
\bigg).
\end{align*}
Moreover one has for any $s_0>\frac{1}{2}$, $\sigma\geq0$ and any $i\in\{1,\cdots,N\}$, using classical product estimates and composition estimates in Sobolev spaces we have
\begin{align*}
\left|\partial_x\frac{1}{\underline{H}_i+H_i}\right|_{H^{\sigma}_x}&=\left|\frac{\partial_xH_i}{(\underline{H}_i+H_i)^2}\right|_{H^{\sigma}_x}\leq \left|\frac{\partial_xH_i}{\underline{H}_i^2}\right|_{H^{\sigma}_x}+\left|\left(\frac{1}{\underline{H}_i^2}-\frac{1}{(\underline{H}_i+H_i)^2}\right)\partial_x H_i\right|_{H^{\sigma}_x}\\
&\lesssim(h_*^{-1})^2\left|\partial_x H_i\right|_{H^{\sigma}_x}+\left|\left(\frac{1}{\underline{H}_i^2}-\frac{1}{(\underline{H}_i+H_i)^2}\right)\right|_{H^{s_0}_x}\left|\partial_x H_i\right|_{H^{\sigma}_x}\\&\hspace{4cm}+\left<\left|\left(\frac{1}{\underline{H}_i^2}-\frac{1}{(\underline{H}_i+H_i)^2}\right)\right|_{H^{\sigma}_x}\left|\partial_x H_i\right|_{H^{s_0}_x}\right>_{\sigma>s_0}\\
&\leq C\left(h_*^{-1},\n H_i\n_{H^{s_0}_x}\right)\n\partial_xH_i\n_{H^{\sigma}_x},
\end{align*}
and therefore   
\begin{align*}
\left\|\Lambda^{s-\frac{3}{2}}\left(\partial_x\frac{1}{\underline{H}+H}\right)\right\|_{l^\infty(L^2_x)}&\lesssim C\left(h_*^{-1},\n\Lambda^{s-\frac{3}{2}}H\n_{l^\infty(L^2_x)}\right)\n\Lambda^{s-\frac{3}{2}}\partial_xH\n_{l^\infty(L^2_x)},
\end{align*}
and
\begin{align*}
\left\|\Lambda^{s-1}\left(\partial_x\frac{1}{\underline{H}+H}\right)\right\|_{l^2(L^2_x)}&\lesssim C\left(h_*^{-1},\n\Lambda^{s-\frac{3}{2}}H\n_{l^\infty(L^2_x)}\right)\n\Lambda^{s-1}\partial_xH\n_{l^2(L^2_x)}.
\end{align*}
Using classical product estimates in Sobolev spaces we have
\begin{align*}
\n\Lambda^{s-\frac{3}{2}}\left((\partial_xH)(\partial_xU)\right)\n_{l^\infty(L^2_x)}&\lesssim \n\partial_x\Lambda^{s-\frac{3}{2}}H\n_{l^\infty(L^2_x)}\n\Lambda^{s-\frac{3}{2}} \partial_xU\n_{l^\infty(L^2_x)}\lesssim \n H\n_{H^{s,1}}\n U\n_{H^{s,1}}.
\end{align*}
By Lemma \ref{comnatu} ($(2)$ and $(4)$) and since $s-1>\frac{1}{2}+\frac{1}{2}$
\begin{align*}
\n\Lambda^{s-1}((\partial_xH)(\partial_x U))&\n_{l^2(L^2_x)} \lesssim \n H\n_{H^{s,1}}\n U\n_{H^{s,1}}.
\end{align*}
and 
\begin{align*}
&\left\|\kappa\frac{1}{\underline{H}+H}[\partial_x^\alpha,\partial_xH]\partial_xU\right\|_{l^2(L^2_x)}\lesssim \kappa h_*^{-1}\left(\n \partial_x^2 H\n_{H^{s-1,1}}\n \partial_xU\n_{H^{s-1,1}}\right)\\
&\hspace{5cm}\lesssim\kappa h_*^{-1}\left(\n \partial_x H\n_{H^{s,1}}\n U\n_{H^{s,1}}\right).
\end{align*}
Consequently, using Lemma \ref{A1},
\begin{align}\label{Rest4}
\n \textbf{R}_{\alpha,0}\n_{l^2(L^2_x)}\lesssim \n U\n_{H^{s,1}}^2+\kappa C(h_*^{-1},\n H\n_{H^{s-1,1}})\left( \n H\n^2_{H^{s,1}}+\n \partial_xH\n_{H^{s,1}}\right)\n U\n_{H^{s,1}}.
\end{align}
Estimates for $R_{\alpha,j}$ for $ j\in\left\{1,2\right\} $, with $0\leq\alpha\leq s-j$ .\\

\noindent For $j=1$ using the fact that $\Dp {\sf S}=\rd$, we have
$$R_{\alpha,1}=\rd\left(-[\partial_x^{\alpha+1},U] H-\partial_x^\alpha\left(( \underline{H})(\partial_xU)\right)\right).$$
Using Lemma \ref{comnatu} $(4)$ and since $s-1>\frac{1}{2}+1$ we have 
\begin{align*}
\n[\partial_x^{\alpha+1}, U] H\n_{l^2(L^2_x)}&\lesssim \n\partial_x U\n_{H^{s-1,1}}\n H\n_{H^{s-1,1}}\\
&\lesssim \n U\n_{H^{s,1}}\n H\n_{H^{s-1,1}}.
\end{align*}
and
\begin{align*}
\n\partial_x^\alpha\left((\underline{H})(\partial_xU)\right) \n_{l^2(L^2_x)}\leq \br\underline{H}\br_{l^\infty}\n U\n_{H^{s,0}}.
\end{align*}

\noindent For $j=2$, using Lemma \ref{IL}  and the fact that $\Ds {\sf S}=\rd\Dp $, we have
\begin{align*}
R_{\alpha,2}=&\rd\bigg(-\llbracket\partial_x^\alpha\Dp,\underline{U}+U\rrbracket\partial_xH-\partial_x^\alpha\Dp\left((\underline{H}+H)\partial_x U)\right)\bigg).
\end{align*}
Using Lemma \ref{comutgen} $(2)$, Lemma \ref{AlgebreBanach} (1) and since $s-2>\frac{1}{2}$, we have
\begin{align*}
\n\llbracket\partial_x^\alpha\Dp,U\rrbracket\partial_xH\n_{l^2(L^2_x)}&
\lesssim\n U\n_{H^{s,2}}\n H\n_{H^{s-1,1}},
\end{align*}
and
\begin{align*}
\n\partial_x^\alpha\Dp\left((H)(\partial_x U)\right)\n_{l^2(L^2_x)}&\lesssim\n  H\n_{H^{s-1,1}} \n  U\n_{H^{s,1}}.
\end{align*}
Moreover we have
\begin{align*}
\n\llbracket\partial_x^\alpha\Dp,\underline{U}\rrbracket\partial_xH\n_{l^2(L^2_x)}&\leq\br\Dp \underline{U}\br_{l^\infty}\n H\n_{H^{s-1,0}}, 
\end{align*}
and
\begin{align*}
\n\partial_x^\alpha\Dp\left((\underline{H})(\partial_x U)\right)\n_{l^2(L^2_x)}&\lesssim\br \underline{H}\br_{w^{1,\infty}}\n U\n_{H^{s,1}}.
\end{align*}
Hence 
\begin{align}\label{Rest5}
\n R_{\alpha,j}\n _{l^2(L^2_x)}\lesssim (\n H\n_{H^{s-1,1}}+\br\underline{H}\br_{w^{1,\infty}}+\br\Dp\underline{U}\br_{l^\infty})(\n U\n_{H^{s,2}}+\n H\n_{H^{s-1,0}}).
\end{align}
Estimates for $\textbf{R}_{\alpha,j}=-\left\llbracket\partial_x^\alpha \mathrm{D}_{\rho}^j,\underline{U}+U-\kappa\frac{\partial_x H}{\underline{H}+H}\right\rrbracket\partial_x U-\partial_x^\alpha\left( \mathrm{D}_{\rho}^j\left({\sf \Gamma}\partial_xH\right)\right)$, with $j\in\left\{1,2\right\}$ and $0\leq\alpha\leq s-j$. \\

\noindent By means of Lemma \ref{IL} and Lemma \ref{propD}, we easily infer  \begin{align*}
\left\|\frac{\partial
_xH}{\underline{H}}\right\|_{H^{s,2}}\leq C(h_*^{-1},\br\Dp\underline{H}\br_{l^\infty},\br\Ds\underline{H}\br_{l^\infty})\n\partial_xH\n_{H^{s,2}}.
\end{align*}
Since $s>\frac{3}{2}+\frac{1}{2}$, by Lemma \ref{comutgen} we have 
\begin{align*}
\n\llbracket\partial_x^\alpha\Dp^j,U\rrbracket\partial_x U\n_{l^2(L^2_x)}\lesssim \n U\n_{H^{s,2}}\n \partial_xU\n_{H^{s-1,2}}.
\end{align*}
By Lemma \ref{IL}, we have 
\begin{align*}
\n\llbracket\partial_x^\alpha\Dp^j,\underline{U}\rrbracket\partial_x U\n_{l^2(L^2_x)}
\lesssim\left(\br \Dp\underline{U}\br_{l^\infty}+\br \Ds\underline{U}\br_{l^\infty}\right)\n U\n_{H^{s,1}}.
\end{align*}
Moreover  using Lemma \ref{comutgen}, Lemma \ref{AlgebreBanach} $(2)$ and Lemma \ref{composHsk} it follows that

\begin{align*}
\kappa\bigg\|\left\llbracket\partial_x^\alpha\Dp^j,\frac{\partial_xH}{\underline{H}+H}\right\rrbracket\partial_x U\bigg\|_{l^2(L^2_x)}&\lesssim\kappa\left\|\frac{\partial
_xH}{\underline{H}+H}\right\|_{H^{s,2}}\n\partial_xU\n_{H^{s-1,2}}\\
&\lesssim\kappa \left\|\frac{\partial
_xH}{\underline{H}}\right\|_{H^{s,2}}\n U\n_{H^{s,2}}+\kappa\n U\n_{H^{s,2}}\n\partial_xH\n_{H^{s,2}}\left\|\frac{1}{\underline{H}+H}-\frac{1}{\underline{H}}\right\|_{H^{s-1,1}}\\
&\quad+\kappa\n U\n_{H^{s,2}}\n\partial_xH\n_{H^{s-1,1}}\left\|\frac{1}{\underline{H}+H}-\frac{1}{\underline{H}}\right\|_{H^{s,2}}\\
&\lesssim\kappa C(h_*^{-1},\br\Dp\underline{H}\br_{w^{1,\infty}})\n\partial_xH\n_{H^{s,2}}\n U\n_{H^{s,2}}\\
&\quad+\kappa\n U\n_{H^{s,2}}\n\partial_xH\n_{H^{s,2}}C(h^{-1}_*,\underline{M},M)\n H\n_{H^{s-1,1}}\\
&\quad+\kappa\n U\n_{H^{s,2}}\n H\n_{H^{s,2}}C(h^{-1}_*,\underline{M},M)\n H\n_{H^{s,2}}.
\end{align*}
Using \eqref{rhogammadecomp} we have
\begin{equation*}
\partial_x^\alpha\left(\Dp^j\left({\sf \Gamma}\partial_xH\right)\right)= \partial_x^\alpha\left(\Dp^j\left(\boldsymbol{\rho}_1\boldsymbol{\rho}^{-1}({\sf TS})^t\partial_x{\sf T}{\sf S}H\right)\right)+ \partial_x^\alpha\left(\Dp^j\left( \boldsymbol{\rho}^{-1}{\sf S}^t{\sf C}\partial_x{\sf S}H\right)\right),
\end{equation*}
consequently using Lemma \ref{IL} and the fact that $\Dp\boldsymbol{\rho}=-(1\cdots,1)^t$, $\Ds\boldsymbol{\rho}=(0,\cdots,0)^t$,\\ $\Dp({\sf T}{\sf S})^t({\sf T}{\sf S})=0$ and $\Dp{\sf S}=\rd$, $\Dp{\sf S}^t=-\ru$, one has 
\begin{align*}
\left\| \partial_x^\alpha\left(\Dp^j\left(\boldsymbol{\boldsymbol{\rho}}_1\boldsymbol{\rho}^{-1}({\sf TS})^t\partial_x{\sf T}{\sf S}H\right)\right)\right\|_{l^2(L^2_x)} &\lesssim  \boldsymbol{\rho}_1  C\left( \left|\boldsymbol{\rho}^{-1}\right|_{l^\infty}\right)\n\partial_x^{\alpha+1}{\sf T}{\sf S}H\n_{l^2(L^2_x)}\\
&\lesssim \boldsymbol{\rho}_1 C\left(\left|\boldsymbol{\rho}^{-1}\right|_{l^\infty}\right) \n {\sf T}{\sf S}H\n_{H^{s,0}}
\end{align*}
and
\begin{align*}
\left\| \partial_x^\alpha\left(\Dp^j\left( \boldsymbol{\rho}^{-1}{\sf S}^t{\sf C}\partial_x{\sf S}H\right)\right)\right\|_{l^2(L^2_x)} &\lesssim    C\left(  \left|\boldsymbol{\rho}^{-1}\right|_{l^\infty}\right)\n\partial_x^{\alpha+1}{\sf S}H\n_{l^2(L^2_x)}\\
&\lesssim   C\left( \left|\boldsymbol{\rho}^{-1}\right|_{l^\infty}\right) \n {\sf S}H\n_{H^{s,1}}.\\
\end{align*}
Hence 
\begin{align}\label{Rest6}
\n\mathbf{R}_{\alpha,j}\n_{l^2(L^2_x)}&\lesssim  C\left( \boldsymbol{\rho}_1 , \left|\boldsymbol{\rho}^{-1}\right|_{l^\infty}\right) (\n {\sf T}{\sf S}H\n_{H^{s,0}}+\n {\sf S}H\n_{H^{s,0}})+(\br\Dp\underline{U}\br_{w^{1,\infty}}+\n U\n_{H^{s,2}})\n U\n_{H^{s,2}}\nonumber\\
&\hspace{0,5cm}+\kappa C(h_*^{-1},\br\underline{H}\br_{w^{2,\infty}},\n H\n_{H^{s-1,1}})\left( \n H\n^2_{H^{s,2}}+\n \partial_xH\n_{H^{s,2}}\right)\n U\n_{H^{s,2}}.
\end{align}
Estimates for $r
_{\alpha,0}=-[\partial_x^{\alpha+1};U,H]+\left(\partial_x^\alpha U\right)(\partial_xH)$, $\textbf{r}_{\alpha,0}=-(\partial_x^\alpha  U)(H)-\partial_x^\alpha \left(\underline{H} U\right)$, with  $0\leq\alpha\leq s$.\\

Using Lemma \ref{comnatu} with Lemma \ref{A1} and since $s-\frac{3}{2}>\frac{1}{2}$ we have
\begin{align*}
\n[\partial_x^{\alpha+1};U,H]\n_{l^2(L^2_x)}\lesssim \n H\n_{H^{s,1}}\n U\n_{H^{s,1}},
\end{align*}
and
\begin{align*}
\n\left(\partial_x^\alpha U\right)(\partial_xH)\n_{l^2(L^2_x)}&\lesssim \n H\n_{H^{s,1}}\n U\n_{H^{s,0}},
\end{align*}
and
\begin{align*}
\n(\partial_x^\alpha  U)(H)\n_{l^2(L^2_x)}&\lesssim \n H\n_{H^{s-1,1}}\n U\n_{H^{s,0}}.
\end{align*}
Finally,
\begin{align*}
\n\partial_x^\alpha \left(\underline{H} U\right)\n_{l^2(L^2_x)}\lesssim \br\underline{H}\br_{l^\infty}\n U\n_{H^{s,0}}.
\end{align*}
Estimates for  $r_{\alpha,j},\, \textbf{r}_{\alpha,j}$ for $j\in\left\{1,2\right\}$ such that $0\leq\alpha\leq s-j$.\\

\noindent We have
\begin{align*}
&r_{\alpha,j}=-\llbracket\partial_x^{\alpha+1} {\sf D}_{\rho}^j;U,H\rrbracket+(\partial_x^\alpha {\sf D}_{\rho}^j U)({\sf M}^j\partial_xH),\\
& \textbf{r}_{\alpha,j}=-\llbracket{\sf D}_{\rho}^j,\underline{U}\rrbracket\partial_x^\alpha H -\partial_x^\alpha \left( {\sf D}_{\rho}^j(\underline{H} U)\right)-(\partial_x^\alpha  {\sf D}_{\rho}^j U)({\sf M}^jH).
\end{align*}
Using Lemma \ref{IL}, we have
\begin{align*}
\n\llbracket\Dp^j,\underline{U}\rrbracket\partial_x^\alpha H\n_{l^2(L^2_x)}&\lesssim \br\Dp\underline{U}\br_{l^\infty}\n H\n_{H^{s-1,1}}+\br\Ds\underline{U}\br_{l^\infty}\n H\n_{H^{s-2,0}},
\end{align*}
and
\begin{align*}
\n\partial_x^\alpha \left(\Dp^j(\underline{H} U)\right)\n_{l^2(L^2_x)}
&\lesssim  \br\underline{H}\br_{l^\infty}\n U\n_{H^{s,2}}+\br\Dp\underline{H}\br_{l^\infty}\n U\n_{H^{s-1,1}}+\br\Ds\underline{H}\br_{l^\infty}\n U\n_{H^{s-2,0}}.
\end{align*}
Moreover using Lemma \ref{A1} and since $s>\frac{3}{2}+\frac{1}{2}$ we have
\begin{align*}
\n(\partial_x^\alpha \Dp^j U)({\sf M}^jH)\n_{l^2(L^2_x)}&\lesssim
 \n H\n_{H^{s-1,1}}\n U\n_{H^{s,2}},
\end{align*}
and 
\begin{align*}
\n(\partial_x^\alpha\Dp^j U)({\sf M}^j\partial_xH)\n_{l^2(L^2_x)}&\lesssim \n
H\n_{H^{s,1}}\n U\n_{H^{s,2}}.
\end{align*}
Since $\alpha+1\leq s+1-j$ and $s+1\geq 2$, $  s+1> \frac{5}{2}+\frac{1}{2} $ then by Lemma \ref{comutgen} $(3)$.
\begin{align*}
\n\llbracket\partial_x^{\alpha+1}\Dp^j;U,H\rrbracket\n_{l^2(L^2_x)}\lesssim\n U\n_{H^{s,2}}\n H\n_{H^{s,2}}.
\end{align*}
Hence for $j\in\{0,1,2\}$ such that $0\leq\alpha\leq s-j$
\begin{align}\label{Rest7}
&\n r_{\alpha,j}\n_{l^2(L^2_x)}\lesssim \n H\n_{H^{s,2}}\n U\n_{H^{s,2}},\\
&\n \textbf{r}_{\alpha,j}\n_{l^2(L^2_x)}\lesssim (\br\underline{H}\br_{w^{2,\infty}}+\br\Dp\underline{U}\br_{w^{1,\infty}}+\n H\n_{H^{s-1,1}})(\n U\n_{H^{s,2}}+\n H\n_{H^{s-1,1}})\nonumber.
\end{align}
Collecting estimates $(\ref{Rest1})-(\ref{Rest7})$ the lemma is proved.
\end{proof}

\subsection{Energy estimates} \label{section4.2}
In this section we give some energy estimates of linear equations arising in Lemma \ref{restes}. We first recall some energy estimates of transport and transport diffusion equations:

\begin{equation}\label{transdiff}
\partial_t\dot{H}+\dot{U}\,\partial_x\dot{H}=\kappa\partial_x^2\dot{H}+R+\partial_x\textbf{R}.
\end{equation}

\begin{equation}\label{trans}
\partial_t\dot{H}+\dot{U}\,\partial_x\dot{H}=R.
\end{equation}

\begin{lemma}\label{esttransdiff,trans} ~\\

\begin{enumerate}
\item 

There exists a universal constant $C_0>0$ such that for any $N\in\mathbb{N}^*$, $\kappa>0$ and $T>0$, for any $\dot{U}\in L^\infty(0,T;l^\infty(L^\infty_x(\R)))$ with $\partial_x \dot{U}\in L^1(0,T;l^\infty(L^\infty_x(\R)))$, for any $(R,\textbf{R})\in L^2(0,T;l^2(L^2_x(\R)))$ and for any $\dot{H}\in C^0([0,T);l^2(L^2_x(\R)))$ with $\partial_x\dot{H}\in L^2(0,T;l^2(L^2_x(\R)))$ such that $(\ref{transdiff})$ holds in $L^2(0,T;(H^{1,0}(\R))')$ we have 

\begin{multline}\label{tdifest}
\n\dot{H}\n_{L^\infty(0,T;l^2(L^2_x))}+\kappa^\frac{1}{2} \n\partial_x\dot{H}\n_{L^2(0,T;l^2(L^2_x))}\\
\leq C_0\left( \n\dot{H}|_{t=0}\n_{l^2(L^2_x)}+\n R\n_{L^1(0,T;l^2(L^2_x))}+\kappa^{-\frac{1}{2}}\n\textbf{R}\n_{L^2(0,T;l^2(L^2_x))}\right)\\
\hspace{8cm}\exp\left(C_0\n\partial_x \dot{U}\n_{L^1(0,T;l^\infty(L^\infty_x))}\right).
\end{multline}

\item There exists a universal constant $C_0>0$ such that for any $N\in\N^*$, $T>0$, and any $\dot{H}\in L^\infty(0,T;l^2(L^2_x(\R)))$ solution of $(\ref{trans})$ with initial data $\dot{H}|_{t=0}\in l^2(L^2_x)$, with\\ $R\in L^1(0,T;l^2(L^2_x(\R)))$, we have 
\begin{multline}\label{trasest}
\n\dot{H}\n_{L^\infty(0,T;l^2(L^2_x))}\leq \left( \n\dot{H}|_{t=0}\n_{l^2(L^2_x)}+\n R\n_{L^1(0,T;l^2(L^2_x))}\right)\exp\left(C_0\n\partial_x \dot{U}\n_{L^1(0,T;l^\infty(L^\infty_x))}\right).
\end{multline}
\end{enumerate}

\end{lemma}
\begin{proof}

Let $i\in\{1,\cdots,N\}$ fixed, then $\dot{H}_i$ satisfies the following equation 

\begin{equation*}
\partial_t\dot{H}_i+\dot{U}_i\;\partial_x\dot{H}_i=\kappa\partial_x^2\dot{H}_i+R_i+\partial_x\textbf{R}_i.
\end{equation*}

It is standard (see Chapter $3$ in \cite{BCD}) that we have the estimate
\begin{multline*}
\n\dot{H}_i\n_{L^{\infty}(0,T;L^2_x)}+\kappa^\frac{1}{2} \n\partial_x\dot{H}_i\n_{L^2(0,T;L^2_x)}\\
\leq C_0\left( \n\dot{H}_i|_{t=0}\n_{L^2_x}+\n R_i\n_{L^1(0,T;L^2_x)}+\kappa^{-\frac{1}{2}}\n\textbf{R}_i\n_{L^2(0,T;L^2_x)}\right)\\
\hspace{8cm}\exp\left(C_0\n\partial_x \dot{U}\n_{L^1(0,T:l^\infty(L^\infty_x))}\right).
\end{multline*}

Taking the $l^2$-norm of this estimate and using triangular inequality  we infer \eqref{tdifest}.

In the same way \eqref{trasest}) is obtained  using standard energy estimates for the transport equation (see Th.~3.14 in \cite{BCD}).

\end{proof}

We now consider system 
\begin{align}\label{quasipt}
&\partial_t\dot{H}+(\underline{U}+U)\partial_x\dot{H}
+(\underline{H}+H)\partial_x\dot{U}=\kappa\partial_x^2\dot{H}+R,\\&
\partial_t \dot{U}+{\sf \Gamma}\partial_x\dot{H}+\left(\underline{U}+U-\kappa\frac{\partial_xH}{\underline{H}+H}\right)\partial_x\dot{U}=\textbf{R},\nonumber.
\end{align}
where we recall that ${\sf \Gamma}$ satisfies the decomposition \eqref{rhogammadecomp}.
\begin{lemma}\label{energy}
Let $h_*,M,\boldsymbol{\rho}_1>0$ .  There exists $C=C(M,\boldsymbol{\rho}_1,h_*)>0$ $($independent of $N$$)$ such that for all $\kappa>0$, $N\in\N^*$ and any 
$\underline{H},\underline{U}\in \R^N$  and any $T>0$ and $(H,U)$ solution to the system $(\ref{multicouches})$ on $[0,T]$ with $(H, U)\in C([0, T];H^{s,2}(\R)^{2})$, $H\in L^2(0,T;H^{s+1,2}(\R)) $, satisfying for almost every $t\in[0,T]$ the upper bound 
\begin{equation*}
\n\partial_xH(t,\cdot)\n_{L^\infty_x(l^2)}+\n\partial_xU(t,\cdot)\n_{l^\infty(L^\infty_x)}
\leq M,
\end{equation*}
and the lower bound
\begin{equation}\label{noncavitation}
\underset{(i,x)\in\{1,\cdots,N\}\times\R}{\inf}\underline{H}_i+H_i(t,x)\geq h_*,
\end{equation}
and for any $(\dot{H},\dot{U})\in C([0, T];l^2(L^2(\R)^{2})$ with $\dot{H}\in L^2(0,T;l^2(H^1(\R)))$ satisfying \eqref{quasipt}   with $R, \textbf{R}\in L^2(0,T;l^2(L^2(\R)))$, the following estimate holds.
\begin{align*}
\mathcal{E}^\frac{1}{2}(\dot{H},\dot{U})(t)+\frac{\kappa^\frac{1}{2}}{2}&\n\partial_x{\sf C}{\sf S}\dot{H}\n_{L^2(0,t;l^2(L^2_x))}+\left(\boldsymbol{\rho}_1\kappa\right)^\frac{1}{2}\n\partial_x{\sf T}{\sf S}\dot{H}\n_{L^2(0,t;l^2(L^2_x))}\\&\leq \left(\mathcal{E}^\frac{1}{2}(\dot{H},\dot{U})(0)+C\int\limits_{0}^t\mathcal{E}^\frac{1}{2}(R,\textbf{R})(\tau)d\tau\right)\\
&\hspace{2cm}\times \exp\left(C\int\limits_0^t\left(1+\kappa^{-1}\n \Dp( \underline{U}+U)(\tau,\cdot)\n_{L^\infty_x(l^2)}^2\right)d\tau\right),
\end{align*}
where we denote 
\begin{equation*}
\mathcal{E}(\dot{H},\dot{U})(t):=\frac{1}{2}\n {\sf C}{\sf S}\dot{H}\n_{l^2(L^2_x)}^2+ \frac{\boldsymbol{\rho}_1}{2}\n {\sf T}{\sf S}\dot{H}\n_{l^2(L^2_x)}^2+\frac{1}{2}\int_{\R}\left<\dot{U},\boldsymbol{\rho} (\underline{H}+H)\dot{U}\right>_{l^2}.
\end{equation*}

\end{lemma}
\begin{proof}
In this proof we consider $(\dot{H},\dot{U})$ to be a sufficiently regular solution to the system \eqref{quasipt} so that the following computations and integration by parts hold true. The case of $(\dot{H},\dot{U})$ with the regularity mentioned in the Lemma is then done by the classical argument of regularization and passing to the limit.\\

By applying ${\sf C}{\sf S}$ to $(\ref{quasipt})_1$ and then testing it against ${\sf C}{\sf S}\dot{H}$, and using Abel's summation of Lemma \ref{IL} we obtain
\begin{align*}
\frac{1}{2}\frac{d}{dt}\n {\sf C}{\sf S}\dot{H}\n_{l^2(L^2_x)}^2&+\big<(\underline{U}+U)\partial_x{\sf C}{\sf S}\dot{H}, {\sf C}{\sf S}\dot{H}\big>_{L^2_xl^2}-\big<{\sf C}{\sf S}_0\left((\Dp(\underline{U}+U))(\ru(\partial_x{\sf S}\dot{H})\right), {\sf C}{\sf S}\dot{H}\big>_{L^2_xl^2}\\
&+\big<{\sf C}{\sf S}\left((\underline{H}+H)\partial_x\dot{U}\right), {\sf C}{\sf S}\dot{H}\big>_{L^2_xl^2}=-\kappa\n\partial_x{\sf C}{\sf S}\dot{H}\n_{l^2(L^2_x)}^2+\big<{\sf C}{\sf S} R, {\sf C}{\sf S}\dot{H}\big>_{L^2_xl^2}.
\end{align*}

In the same way by applying ${\sf T}{\sf S}$ to $(\ref{quasipt})_1$ and then testing it against $\boldsymbol{\rho}_1{\sf T}{\sf S}\dot{H}$  and using Abel's summation of Lemma \ref{IL} we obtain
\begin{align*}
\frac{\boldsymbol{\rho}_1}{2}\frac{d}{dt}\n {\sf T}{\sf S}\dot{H}&\n_{l^2(L^2_x)}^2+\boldsymbol{\rho}_1\kappa\n\partial_x{\sf T}{\sf S}\dot{H}\n_{l^2(L^2_x)}^2=-\boldsymbol{\rho}_1\big<{\sf T}\left((\underline{U}+U)\partial_x{\sf S}\dot{H}\right), {\sf T}{\sf S}\dot{H}\big>_{L^2_xl^2}+\boldsymbol{\rho}_1\big<{\sf T}{\sf S} R, {\sf T}{\sf S}\dot{H}\big>_{L^2_xl^2}\\
&+\boldsymbol{\rho}_1\big<{\sf T}{\sf S}_0\left((\Dp(\underline{U}+U))(\ru\partial_x{\sf S}\dot{H})\right), {\sf T}{\sf S}\dot{H}\big>_{L^2_xl^2}-\boldsymbol{\rho}_1\big<{\sf T}{\sf S}\left((\underline{H}+H)\partial_x\dot{U}\right), {\sf T}{\sf S}\dot{H}\big>_{L^2_xl^2}.
\end{align*}

By testing $(\ref{quasipt})_2$  against $\boldsymbol{\rho} (\underline{H}+H)\dot{U}$ and using the identity \eqref{rhogammadecomp} and the fact that $(H, U)$ satisfies \eqref{multicouches} we have
\begin{align*}
&\frac{1}{2}\frac{d}{dt}\left(\big<\dot{U},\boldsymbol{\rho} (\underline{H}+H)\dot{U}\big>_{L^2_xl^2}\right)+\boldsymbol{\rho}_1\big<{\sf T}{\sf S} \partial_x\dot{H}, {\sf T}{\sf S}\left((\underline{H}+H)\dot{U}\right)\big>_{L^2_xl^2}+\big<{\sf C}{\sf S} \partial_x\dot{H}, {\sf C}{\sf S}\left((\underline{H}+H)\dot{U}\right)\big>_{L^2_xl^2}\\
&\hspace{7cm}=\big<\textbf{R},\boldsymbol{\rho}(\underline{H}+H)\dot{U}\big>_{L^2_xl^2}.
\end{align*}
Collecting the above, $\mathcal{E}(\dot{H},\dot{U})$ satisfies the following  differential equation:
\begin{align}
&\frac{d}{dt}\mathcal{E}(\dot{H},\dot{U})(t)+\kappa\n\partial_x{\sf C}{\sf S}\dot{H}\n_{l^2(L^2_x)}^2+\boldsymbol{\rho}_1\kappa\n\partial_x{\sf T}{\sf S}\dot{H}\n_{l^2(L^2_x)}^2 \nonumber\\
&=-\big<(\underline{U}+U)\partial_x{\sf C}{\sf S}\dot{H}, {\sf C}{\sf S}\dot{H}\big>_{L^2_xl^2}-\boldsymbol{\rho}_1\big<{\sf T}\left((\underline{U}+U)\partial_x{\sf S}\dot{H}\right), {\sf T}{\sf S}\dot{H}\big>_{L^2_xl^2}\label{1}\\
&+\big<{\sf C}{\sf S}_0\left((\Dp(\underline{U}+U))(\ru(\partial_x{\sf S}\dot{H})\right), {\sf C}{\sf S}\dot{H}\big>_{L^2_xl^2}+\boldsymbol{\rho}_1\big<{\sf T}{\sf S}_0\left((\Dp(\underline{U}+U))(\ru\partial_x{\sf S}\dot{H})\right), {\sf T}{\sf S}\dot{H}\big>_{L^2_xl^2}\label{2}\\
&-\boldsymbol{\rho}_1\big<{\sf T}{\sf S} \partial_x\dot{H}, {\sf T}{\sf S}\left((\underline{H}+H)\dot{U}\right)\big>_{L^2_xl^2}-\big<{\sf C}{\sf S} \partial_x\dot{H}, {\sf C}{\sf S}\left((\underline{H}+H)\dot{U}\right)\big>_{L^2_xl^2}\label{3}\\
&-\boldsymbol{\rho}_1\big<{\sf T}{\sf S}\left((\underline{H}+H)\partial_x\dot{U}\right), {\sf T}{\sf S}\dot{H}\big>_{L^2_xl^2}-\big<{\sf C}{\sf S}\left((\underline{H}+H)\partial_x\dot{U}\right), {\sf C}{\sf S}\dot{H}\big>_{L^2_xl^2}\label{4}\\
&+\big<{\sf C}{\sf S} R, {\sf C}{\sf S}\dot{H}\big>_{L^2xl^2}+\big<\textbf{R},\boldsymbol{\rho}(\underline{H}+H)\dot{U}\big>_{L^2_xl^2}+\boldsymbol{\rho}_1\big<{\sf T}{\sf S} R, {\sf T}{\sf S}\dot{H}\big>_{L^2_xl^2}.\label{5}
\end{align}
Before estimating all the terms in the previous equality we notice that 
\begin{equation*}
\frac{1}{2}\n{\sf C} {\sf S}\dot{H}\n_{l^2(L^2_x)}^2+\frac{\boldsymbol{\rho}_1}{2}\n {\sf T}{\sf S}\dot{H}\n_{l^2(L^2_x)}^2+\frac{\boldsymbol{\rho}_1h_*}{2}\n\dot{U}\n_{l^2(L^2_x)}^2\leq \mathcal{E}(\dot{H},\dot{U})(t).
\end{equation*}
Hence we can now begin the estimating process to establish a differential inequality.\\

We estimate the terms of $(\ref{2})$ using Lemma \ref{contop} and Cauchy Schwarz inequality.
\begin{align*}
\bigg|\big<{\sf C}{\sf S}_0\left((\Dp(\underline{U}+U))(\ru(\partial_x{\sf S}\dot{H})\right)&, {\sf C}{\sf S}\dot{H}\big>_{L^2_xl^2}\bigg|\leq\int_{\R}\br (\Dp(\underline{U}+U))(\ru(\partial_x{\sf C}{\sf S}\dot{H})\br_{l^1}\br {\sf C}{\sf S}\dot{H}\br_{l^2}\\
&\leq \n \Dp(\underline{U}+U)\n_{L^\infty_x(l^2)}\n\partial_x{\sf C}{\sf S}\dot{H}\n_{l^2(L^2_x)}\n {\sf C}{\sf S}\dot{H}\n_{l^2(L^2_x)}\\
&\leq\frac{\kappa}{2}\n\partial_x{\sf C}{\sf S}\dot{H}\n_{l^2(L^2_x)}^2+\kappa^{-1} \n \Dp(\underline{U}+U)\n_{L^\infty_x(l^2)}^2\mathcal{E}(\dot{H},\dot{U})(t).
\end{align*} 
In the same way we have 
\begin{align*}
\bigg|\boldsymbol{\rho}_1\big<{\sf T}{\sf S}_0\left((\Dp(\underline{U}+U))(\ru\partial_x{\sf S}\dot{H})\right), &{\sf T}{\sf S}\dot{H}\big>_{L^2_xl^2}\bigg|
\leq\frac{\kappa}{4}\n\partial_x{\sf C}{\sf S}\dot{H}\n_{l^2(L^2_x)}^2\\
&\hspace{2,5cm}+2\kappa^{-1}\boldsymbol{\rho}_1 \n \Dp(\underline{U}+U)\n_{L^\infty_x(l^2)}^2\mathcal{E}(\dot{H},\dot{U})(t).
\end{align*}
We estimate now the terms in \eqref{1}. Using integration by parts, we have
\begin{align*}
-\big<(\underline{U}+U)\partial_x{\sf C}{\sf S}\dot{H}, {\sf C}{\sf S}\dot{H}\big>_{L^2_xl^2}=\frac{1}{2} \big<(\partial_xU)({\sf C}{\sf S}\dot{H}), {\sf C}{\sf S}\dot{H}\big>_{L^2_xl^2},
\end{align*}
 and consequently
\begin{align*}
\bigg| -\big<(\underline{U}+U)\partial_x{\sf C}{\sf S}\dot{H}, {\sf C}{\sf S}\dot{H}\big>_{L^2_xl^2} \bigg|&\leq\frac{1}{2}\n\partial_x U\n_{l^\infty(L^\infty_x)}\n{\sf C} {\sf S}\dot{H}\n_{l^2(L^2_x)}^2\leq\n\partial_x U\n_{l^\infty(L^\infty_x)} \mathcal{E}(\dot{H},\dot{U})(t).
\end{align*}
In the same way
\begin{align*}
\bigg|-\boldsymbol{\rho}_1\big<{\sf T}\left((\underline{U}+U)\partial_x{\sf S}\dot{H}\right), {\sf T}{\sf S}\dot{H}\big>_{L^2_xl^2}\bigg|\leq \frac{\boldsymbol{\rho}_1}{2} \n\partial_x U\n_{l^\infty(L^\infty_x)}\n {\sf T}{\sf S}\dot{H}\n_{l^2(L^2_x)}^2\leq\n\partial_x U\n_{l^\infty(L^\infty_x)} \mathcal{E}(\dot{H},\dot{U})(t) 
\end{align*}

Let us now consider the terms \eqref{3} and \eqref{4}. Integrating by parts it follows that 
\begin{align*}
-\big<{\sf C}{\sf S}\left((\underline{H}+H)\partial_x\dot{U}\right), {\sf C}{\sf S}\dot{H}\big>_{L^2_xl^2}&= \big<{\sf C}{\sf S}\left((\partial_xH)(\dot{U})\right), {\sf C}{\sf S}\dot{H}\big>_{L^2_xl^2}+\big<{\sf C}{\sf S}\left((\underline{H}+H)\dot{U}\right), \partial_x{\sf C}{\sf S}\dot{H}\big>_{L^2_xl^2}
\end{align*}
and
\begin{align*}
-\boldsymbol{\rho}_1\big<{\sf T}{\sf S}\left((\underline{H}+H)\partial_x\dot{U}\right), {\sf T}{\sf S}\dot{H}\big>_{L^2_xl^2}&=\boldsymbol{\rho}_1\big<{\sf T}{\sf S}\left((\partial_xH)(\dot{U})\right), {\sf T}{\sf S}\dot{H}\big>_{L^2_xl^2}\\
&\hspace{3cm}+\boldsymbol{\rho}_1\big<{\sf T}{\sf S}\left((\underline{H}+H)\dot{U}\right), \partial_x{\sf T}{\sf S}\dot{H}\big>_{L^2_xl^2}.
\end{align*}
Hence,
\begin{align*}
\eqref{3}+\eqref{4}= \big<{\sf C}{\sf S}\left((\partial_xH)(\dot{U})\right), {\sf C}{\sf S}\dot{H}\big>_{L^2_xl^2}+\boldsymbol{\rho}_1\big<{\sf T}{\sf S}\left((\partial_xH)(\dot{U})\right), {\sf T}{\sf S}\dot{H}\big>_{L^2_xl^2}.
\end{align*}
Using Lemma \ref{contop} it follows that
\begin{align*}
\bigg| \big<{\sf C}{\sf S}\left((\partial_xH)(\dot{U})\right), {\sf C}{\sf S}\dot{H}\big>_{L^2_xl^2} \bigg|&\leq\int_{\R} \br\partial_x H\br_{l^2}\br\dot{U}\br_{l^2}\br {\sf C}{\sf S}\dot{H}\br_{l^2}\\
& \leq\n\partial_x H\n_{L^\infty_x(l^2)}\n\dot{U}\n_{l^2(L^2_x)}\n {\sf C}{\sf S} \dot{H}\n_{l^2(L^2_x)}\\
&\leq\frac{1}{\sqrt{h_*\boldsymbol{\rho}_1}} \n\partial_x H\n_{L^\infty_x(l^2)}\mathcal{E}(\dot{H},\dot{U})(t),
\end{align*}
and similarly,
\begin{align*}
\bigg|\rho_1\big<{\sf T}{\sf S}\left((\partial_xH)(\dot{U})\right), {\sf T}{\sf S}\dot{H}\big>_{L^2_xl^2}\bigg|&\leq\boldsymbol{\rho}_1\int_{\R}  \br\partial_x H\br_{l^2}\br\dot{U}\br_{l^2}\br {\sf T}{\sf S}\dot{H}\br_{l^2}\\
&\leq\frac{1}{\sqrt{h_*}}\n\partial_x H\n_{L^\infty_x(l^2)}\mathcal{E}(\dot{H},\dot{U})(t).
\end{align*}
The  terms of $(\ref{5})$ are easily controlled using Cauchy-Schwarz inequality.
\begin{align*}
\big<{\sf C}{\sf S} R, {\sf C}{\sf S}\dot{H}\big>_{L^2_xl^2}+&\big<\textbf{R},\boldsymbol{\rho}(\underline{H}+H)\dot{U}\big>_{L^2_xl^2}+\boldsymbol{\rho}_1\big<{\sf T}{\sf S} R, {\sf T}{\sf S}\dot{H}\big>_{L^2_xl^2}\\
&\leq  \mathcal{E}(\dot{H},\dot{U})^\frac{1}{2}(t) \mathcal{E}(R,\textbf{R})^\frac{1}{2}(t).
\end{align*}
Gathering all the above estimates we have
\begin{align*}
&\frac{d}{dt}\mathcal{E}(\dot{H},\dot{U})(t)+\frac{\kappa}{2}\n\partial_x{\sf C}{\sf S}\dot{H}\n_{l^2(L^2_x)}^2+\boldsymbol{\rho}_1\kappa\n\partial_x{\sf T}{\sf S}\dot{H}\n_{l^2(L^2_x)}^2\\
&\leq\bigg(2(1+\boldsymbol{\rho}_1)\kappa^{-1} \n \Dp(\underline{U}+U)\n_{L^\infty_x(l^2)}^2+2\n\partial_xU\n_{l^\infty(L^\infty_x)}+(\frac{1}{\sqrt{h_*}}+\frac{1}{h_*\boldsymbol{\rho}_1}) \n\partial_x H\n_{L^\infty_x(l^2)}\bigg)\mathcal{E}(\dot{H},\dot{U})
(t)\\&\hspace{5cm}+\mathcal{E}(\dot{H},\dot{U})^\frac{1}{2}(t) \mathcal{E}(R,\textbf{R})^\frac{1}{2}(t),
\end{align*}
and hence there exists $C=C(h_*^{-1},\boldsymbol{\rho}_1,M)$ such that
\begin{align*}
\frac{d}{dt}\mathcal{E}(\dot{H},\dot{U})(t)+&\frac{\kappa}{2}\n\partial_x{\sf C}{\sf S}\dot{H}\n_{l^2(L^2_x)}^2+\boldsymbol{\rho}_1\kappa\n\partial_x{\sf T}{\sf S}\dot{H}\n_{l^2(L^2_x)}^2 \\&\leq C\left(1+\kappa^{-1} \n \Dp(\underline{U}+U)\n_{L^\infty_x(l^2)}^2 \right)\mathcal{E}(\dot{H},\dot{U})
(t)+C\mathcal{E}(\dot{H},\dot{U})^\frac{1}{2}(t) \mathcal{E}(R,\textbf{R})^\frac{1}{2}(t).
\end{align*}
We conclude by using Gronwall's inequality.
\end{proof}

\subsection{Large time well-posedness of the multi-layer system}\label{section4.3}

In this subsection we state and prove our first main result, concerning the well-posedness of system \eqref{multicouches} with a time of existence uniform with respect to the number of layers $N$.

\begin{theorem}\label{timeindepof}
Let $s\in\N$ be such that $s>2+\frac{1}{2}$,  and $\underline{M},M^*,h_*,h^*>0$. Then, there exists $C>0$ such that for any $N\in\N^*$ and any $\kappa>0$,

\begin{itemize}
\item for any $(\boldsymbol{\rho},\underline{H},\underline{U})\in \R^{3N}$ such that
\begin{equation*}
\br\boldsymbol{\rho}\br_{l^\infty}  +\br\boldsymbol{\rho}^{-1}\br_{l^\infty} + \br\underline{H}\br_{w^{2,\infty}}+\br\underline{U}\br_{w^{2,\infty}}\leq \underline{M},
\end{equation*}

\item for any initial data $(H_0, U_0)\in H^{s,2}$ with 

\begin{equation*}
M_0:= \n H_0\n_{H^{s-1,1}}+ \n {\sf S}H_0\n_{H^{s,2}}+ \n {\sf T}{\sf S}H_0\n_{H^{s,0}}+ \n U_0\n_{H^{s,2}}+\kappa^\frac{1}{2} \n H_0\n_{H^{s,2}}\leq M^*
\end{equation*}
and 
\begin{equation*}
\underset{(x,i)\in\R\times\{1,\cdots,N\}}{\inf}\underline{H}_i+(H_0)_i(x)\geq h_*,\quad \underset{(x,i)\R\times\{1,\cdots,N\}}{\sup}\underline{H}_i+(H_0)_i(x)\leq h^*,
\end{equation*}
\end{itemize}

the following holds. Denoting 

\begin{equation}
T^{-1}=C\left(1+\kappa^{-1}\left(\n\Dp\underline{U}\n_{l^2}^2+M_0^2\right)\right),
\end{equation}
there exists a unique strong solution $(H, U)\in C([0,T];H^{s,2}(\R)^{2})$ to \eqref{multicouches} and initial data $(H, U)|_{t=0}=(H_0,U_0)$. Moreover, $H\in L^2(0,T;H^{s+1,2}(\R))$ and one has, for any $t\in[0,T]$, the lower and upper bounds 
\begin{align*}
&\underset{(i,x)\in\{1,\cdots,N\}\times \mathbb{R}^d}{\inf}\underline{H}_i+H_i(t,x)\geq \frac{h_*}{2},\quad\underset{(i,x)\in\{1,\cdots,N\}\times \mathbb{R}^d}{\sup}\underline{H}_i+H_i(t,x)\leq 2h^*,
\end{align*}
and the estimate
\[\Norm{(H,U)}_{s}(t)\leq C M_0,\]
where we define
\begin{align*}
\Norm{(H,U)}_{s}(t)&:=\n H(t,\cdot)\n_{H^{s-1,1}}+ \n {\sf S}H(t,\cdot)\n_{H^{s,2}}+ \n {\sf T}{\sf S}H(t,\cdot)\n_{H^{s,0}}+ \n U(t,\cdot)\n_{H^{s,2}}+\kappa^\frac{1}{2} \n H(t,\cdot)\n_{H^{s,2}}\\
&\quad+ \kappa^\frac{1}{2}\n\partial_xH\n_{L^2(0,t;H^{s-1,1})}+\kappa^\frac{1}{2} \n \partial_x {\sf S}H\n_{L^2(0,t;H^{s,2})}+\kappa^\frac{1}{2} \n \partial_x {\sf T}{\sf S}H\n_{L^2(0,t;H^{s,0})}+\kappa \n \partial_x H\n_{L^2(0,t;H^{s,2})}.
\end{align*}

\end{theorem}
\begin{proof}

Let us denote $T^*\in(0,+\infty]$ the maximal time of existence and uniqueness of $(H, U)\in C([0, T^*);H^{s,2}(\R)^{2})$, $H\in L^2(0,T^*;H^{s+1,2}(\R)) $ as provided by Proposition  \ref{WPML}, and

\begin{equation*}
T_*=\sup\left\{0<T<T^*\; :\,\forall t\in(0,T], \quad \frac{h_*}{2}\leq \underline{H}_i+H_i(t,x)\leq2h^*\;\text{and}\; \Norm{(H,U)}_s(t)\leq c_0M_0\right\},
\end{equation*}
where $c_0>1$ will be determined later on. By continuity in time of the solution we deduce that $T_*>0$. Let $t\in (0, T_*)$.

By Lemma \ref{restes} and Lemma \ref{energy} and using the fact that $\br{\sf S}\br_{l^2\to l^2}\leq \frac{1}{\sqrt{N}}\br{\sf T}{\sf S}\br_{l^2\to l^2}+\br{\sf C}{\sf S}\br_{l^2\to l^2}$, we find that there exists $c>1$ depending on $\boldsymbol{\rho}_1, \boldsymbol{\rho}_{N},h_*, h^*$ and $C>0$ depending additionally on $\underline{M}, c_0 M_0$ such that
\begin{align}\label{esti1}
&\n \Lambda^s{\sf S}H(t,\cdot)\n_{l^2(L^2_x)}+ \n \Lambda^s{\sf T}{\sf S}H(t,\cdot)\n_{l^2(L^2_x)}+ \n \Lambda^sU(t,\cdot)\n_{l^2(L^2_x)}\nonumber\\&\quad+\kappa^\frac{1}{2}\n\partial_x{\sf S}\Lambda^s H\n_{L^2(0,t;l^2(L^2_x))}+\kappa^\frac{1}{2}\n\partial_x\Lambda^s{\sf T}{\sf S}H\n_{L^2(0,t;l^2(L^2_x))}\nonumber\\& \quad\quad \leq c\left(\n \Lambda^s{\sf S}H_0\n_{l^2(L^2_x)}+ \n \Lambda^sU_0\n_{l^2(L^2_x)}+ \n \Lambda^s{\sf T}{\sf S}H_0\n_{l^2(L^2_x)}+Cc_0M_0(t+\sqrt{t})\right)\nonumber\\
&\hspace{6cm}\times \exp\left(C\int\limits_0^t\left(1+\kappa^{-1}\n \Dp (\underline{U}+U)(\tau,\cdot)\n_{L^\infty_x(l^2)}^2\right)d\tau\right).
\end{align}
For all $j\in\{1,2\}$ and applying Lemma \ref{restes} and Lemma \ref{esttransdiff,trans} $(1)$ it follows that 
\begin{align}\label{esti2}
\n \Dp^j{\sf S}H(t&,\cdot)\n_{H^{s-j,0}}+\kappa^\frac{1}{2} \n\partial_x \Dp^j{\sf S}H\n_{L^2(0,t;H^{s-j,0})}\nonumber\\
&
\leq C_0\left( \n \Dp^j{\sf S}H_0\n_{H^{s-j,0}}+Cc_0M_0(t+\sqrt{t})\right)\exp\left(C_0\displaystyle\int\limits_0^t\n\partial_x U(\tau,\cdot)\n_{l^\infty(L^\infty_x)}d\tau\right),
\end{align}
and moreover using Lemma \ref{esttransdiff,trans} $(2)$ we have
\begin{multline}\label{esti3}
\n \Dp^jU(t,\cdot)\n_{H^{s-j,0}}
\leq \left( \n \Dp^jU_0\n_{H^{s-j,0}}+Cc_0M_0(t+\sqrt{t})\right)\\
\hspace{4cm}\times\exp\left(C_0\displaystyle\int\limits_0^t\left\|\partial_x \left(\underline{U}+U-\kappa\frac{\partial_xH}{\underline{H}+H}\right)(\tau,\cdot)\right\|_{l^\infty(L^\infty_x)}d\tau\right).
\end{multline}
For $j\in\{0,1\}$, and using Lemma \ref{restes} and Lemma \ref{esttransdiff,trans} $(1)$ we have 
\begin{align*}
\n\Dp^jH(t,\cdot)\n_{H^{s-1-j,0}}+&\kappa^\frac{1}{2} \n\partial_x\Dp
^jH\n_{L^2(0,t;H^{s-1-j,0})}\\
&\leq C_0\left( \n\Dp^jH_0\n_{H^{s-1-j,0}}+Cc_0M_0t\right)\exp\left(C_0\displaystyle\int\limits_0^t\n\partial_x U(\tau,\cdot)\n_{l^\infty(L^\infty_x)}d\tau\right).
\end{align*}
Finally, for any $j\in\{0,1,2\}$ using  Lemma \ref{restes} and Lemma \ref{esttransdiff,trans} $(1)$ we have 

\begin{align}\label{esti4}
\kappa^\frac{1}{2}\n \Dp^jH&(t,\cdot)\n_{H^{s-j,0}}+\kappa \n\partial_x \Dp^jH\n_{L^2(0,t;H^{s-j,0})}\nonumber\\
&\leq C_0\left(\kappa^\frac{1}{2} \n \Dp^jH_0\n_{H^{s-j,0}}+Cc_0M_0(t+\sqrt{t})\right)\exp\left(C_0\displaystyle\int\limits_0^t\n\partial_x U(\tau,\cdot)\n_{l^\infty(L^\infty_x)}d\tau\right).
\end{align}

Since $s>2+\frac{1}{2}$, by using Lemma \ref{A1}, and the continuous injection $\n\cdot\n_{l^\infty(L^\infty_x)}\leq\n\cdot\n_{H^{s-\frac{3}{2},1}}$ and the fact that 
\begin{align*}
&\n\partial_xU(\tau,\cdot)\n_{l^\infty(L^\infty_x)}\leq \n U(\tau,\cdot)\n_{H^{s-\frac{1}{2},1}},
\end{align*}
and
\begin{align*}
\left\|\partial_x\left(\frac{\partial_xH}{\underline{H}+H}\right)(\tau,\cdot) \right\|_{l^\infty(L^\infty_x)}&\leq  \left\|\left(\frac{\partial^2_xH}{\underline{H}+H}\right)(\tau,\cdot)\right\|_{l^\infty(L^\infty_x)}+ \left\|\left(\frac{\partial_xH}{\underline{H}+H}\right)^2(\tau,\cdot) \right\|_{l^\infty(L^\infty_x)}\\
&\leq C(h_*^{-1})(\n\partial_x H(\tau,\cdot)\n_{H^{s-\frac{1}{2},1}}+\n\partial_x H(\tau,\cdot)\n_{H^{s-\frac{3}{2},1}}^2),
\end{align*}
we have then 
\begin{equation*}
\left\|\partial_x \left(\underline{U}+U-\kappa\frac{\partial_xH}{\underline{H}+H}(\tau,\cdot)\right)\right\|_{l^\infty(L^\infty_x)}\leq C(h_*^{-1}) (\n U(\tau,\cdot)\n_{H^{s,2}}+\kappa\n \partial_xH(\tau,\cdot)\n_{H^{s,2}}+\kappa\n H(\tau,\cdot)\n_{H^{s,2}}^2).
\end{equation*}
and 
\begin{equation*}
\n \Dp(\underline{U}+U)\n_{L^\infty_x(l^2)}^2\lesssim \n\Dp\underline{U}\n_{l^2}^2+\n U\n_{H^{s,2}}^2\leq  \n\Dp\underline{U}\n_{l^2}^2+(c_0M_0)^2.
\end{equation*}
Hence gathering estimates $\eqref{esti1}-\eqref{esti4}$ we find that 
\begin{equation*}
\Norm{(H,U)}_{s}(t)\leq c \bigg(M_0+Cc_0M_0(t+\sqrt{t})\bigg)\exp\left(C\left(t+\sqrt{t}+\int\limits_0^t\left(1+\kappa^{-1}(\n\Dp\underline{U}\n_{l^2}^2+(c_0M_0)^2)\right)d\tau\right)\right),
\end{equation*}
where we recall that $c>1$ depends on $h_*$ and $h^*$, and $C>0$ depends on $\underline{M}, h_*,h^*,c_0 M_0$. Hence choosing $c_0=2c$, we find that there exists $C_1\geq1$ depending only on $\underline{M},M^*, h_*,h^*$ such that 
\begin{equation*}
t\left(1+\kappa^{-1}\left(\n\Dp\underline{U}\n_{l^2}^2+M_0^2\right)\right)\leq C_1^{-1}\implies \Norm{(H,U)}_s(t)\leq \frac{3}{4}c_0 M_0.
\end{equation*}
Moreover we notice that since for all $i\in\{1,\cdots,N\}$
\begin{equation*}
\partial_tH_i=\kappa\partial_x^2H_i+g_i,\quad \text{with}\quad g_i=\partial_x\left(H_i U_i+\underline{H}_i U_i+ \underline{U}_iH_i\right),
\end{equation*}
by the positivity of the heat kernel we have 
\begin{align*}
&\underset{(i,x)\in\{1,\cdots,N\}\times \mathbb{R}^d}{\inf}H_i(t,x)\geq \underset{(i,x)\in\{1,\cdots,N\}\times \mathbb{R}^d}{\inf}H_i(0,x)-\n g\n_{L^1(0,t;l^\infty(L^\infty))},\\
&
\underset{(i,x)\in\{1,\cdots,N\}\times \mathbb{R}^d}{\sup}H_i(t,x)\leq \underset{(i,x)\in\{1,\cdots,N\}\times \mathbb{R}^d}{\sup}H_i(0,x)+\n g\n_{L^1(0,t;l^\infty(L^\infty))}.
\end{align*}
By Lemma \ref{A1}  and by Lemma \ref{AlgebreBanach}
\begin{align*}
\n g\n_{l^\infty(L^\infty_x)}&\lesssim \br\underline{H}\br_{w^{1,\infty}}\n U\n_{H^{s,1}}+\br\underline{U}\br_{w^{1,\infty}}\n H\n_{H^{s,1}}+\n U\n_{H^{s,1}}\n H\n_{H^{s,1}}\\
&\leq C(\underline{M})(1+\kappa^{-1}M_0^2).
\end{align*}
Hence augmenting $C_1$ if necessary we find that 
\begin{equation*}
t\left(1+\kappa^{-1}M_0^2\right)\leq C_1^{-1}\implies \forall (i,x)\in\{1,\cdots,N\}\times\mathbb{R}^d,\quad \frac{2}{3}h_*\leq\underline{H}_i +H_i(t,x)\leq \frac{3}{2}h^*.
\end{equation*}
Hence by continuity in time of the solution we find that for all $t\in(0,T_*)$ satisfying that $ t\left(1+\kappa^{-1}\left(\n\Dp\underline{U}\n_{l^2}^2+M_0^2\right)\right)\leq C_1^{-1}$, there exists $\delta>0$ such that $[t-\delta,t+\delta]\subset(0,T_*)$.
By a continuity argument we deduce that $T_*>\left(C_1\left(1+\kappa^{-1}\left(\n\Dp\underline{U}\n_{l^2}^2+M_0^2\right)\right)\right)^{-1}$, which completes the proof.
\end{proof}

\section{Convergence estimate}\label{section5}

This section is dedicated to the proof of our second main result, Theorem \ref{mainth}. We prove that considering a sufficiently regular solution of the hydrostatic continuously stratified system \eqref{contsystem-intro} satisfying the non-cavitation assumption and appropriate bounds, the solutions to the multi-layer systems  \eqref{mltc-intro} with suitably chosen densities $\boldsymbol{\rho}_{i}$, reference depths $\underline{H}_i$, background velocities $\underline{U}_{i}$ and initial data $(H_i,U_i)\vert_{t=0}$ are at a distance $\mathcal{O}(1/N^2)$ to the continuously stratified solution.

This convergence result is deduced from a consistency result obtained in Section \ref{section5.1}, and on stability estimates derived in Section \ref{section5.2}. The proof of Theorem \ref{mainth} is then completed in Section \ref{section5.3}.

\subsection{Consistency}\label{section5.1}

Let $\rho_{\bott}>\rho_{\surf}>0$ and $\kappa>0$. Let $(\underline{h},\underline{u})$ and $(h,u)$ be a sufficiently smooth solution to the continuously stratified system 
\begin{equation}\label{contsystem}
\left\{\begin{array}{l}
\partial_th+\partial_x((\underline{h}+h)(\underline{u}+u))=\kappa\partial_x^2h\\ 
\partial_tu+\left(\underline{u}+u-\kappa\frac{\partial_xh}{\underline{h}+h}\right)\partial_xu+\frac{1}{\varrho}\mathcal{M}\partial_xh=0\\
h_{|t=0}=h_0\\         
u_{|t=0}=u_0,
\end{array}     
\right.  \end{equation}
where $(\mathcal{M}\eta)(\cdot,\varrho)=\rho_{\surf}\int\limits_{\rho_{\surf}}^{\rho_{\bott}}\eta(\cdot,\varrho')d\varrho'+\int\limits_{\rho_{\surf}}^{\varrho}\int\limits_{\varrho'}^{\rho_{\bott}}\eta(\cdot,\varrho'')d\varrho''d\varrho'
$.

Let $N\in\N^*$ and recall (see Section \ref{section2}) the definition of the linear operator $P_N$
\[
P_N:	\begin{array}{ccc}
	\mathcal{C}([\rho_{\surf},\rho_{\bott}])& \rightarrow & \mathbb{R}^N\\
	f&\mapsto &\left(f(\boldsymbol{\rho}_i\right))_{1\leq i\leq N}
\end{array}
\]
where $\boldsymbol{\rho}_i=\rho_{\surf}+(i-\tfrac12) \frac{\rho_{\bott}-\rho_{\surf}}{N}$. 

Notice that $P_N(fg)=P_N(f)P_N(g)$.
Hence applying $P_N$ to \eqref{contsystem}, we see that $(P_N h,P_N u)$ satisfy 
\begin{equation}\label{smoothsoldiscretised}
\left\{\begin{array}{l}
\partial_tP_Nh+\partial_x(P_N(\underline{h}+h)P_N(\underline{u}+u))=\kappa\partial^2_xP_Nh\\ 
\partial_tP_Nu+P_N\left(\underline{u}+u-\kappa\frac{\partial_xh}{\underline{h}+h}\right)\partial_xP_Nu+{\sf \Gamma}\partial_x P_Nh=\boldsymbol{R}_N\\
P_Nh_{|t=0}=P_N(h_0)\\         
P_Nu_{|t=0}=P_N(u_0),
\end{array}     
\right.  
\end{equation}\\
with ${\sf \Gamma}_{i,j}=\frac1N\frac{\min(\boldsymbol{\rho}_i,\boldsymbol{\rho}_j)}{\boldsymbol{\rho}_i}$ and
\begin{equation}\label{R_N}
\boldsymbol{R}_N={\sf \Gamma}\partial_xP_Nh-P_N\left(\frac{1}{\varrho}\mathcal{M}\partial_xh\right).
\end{equation}
We now estimate $\boldsymbol{R}_N$.

\begin{lemma}\label{ResteN}
For any $\rho_{\bott}>\rho_{\surf}>0$, there exists $C>0$ such that for any $N\in\N^*$, and any sufficiently regular function $h:\R\times(\rho_{\surf},\rho_{\bott})\to \R$, $\boldsymbol{R}_N$ defined by \eqref{R_N} satisfies
\begin{align*}
\n&\boldsymbol{R}_N\n_{H^{s,2}}\leq\frac{C}{N^2}\left(\n\partial_\varrho\partial_x\Lambda^sh\n_{L^\infty_\varrho(L^2_x)}+\n\partial^2_\varrho\partial_x\Lambda^sh\n_{L^\infty_\varrho(L^2_x)}\right).
\end{align*}
\end{lemma}

\begin{proof}~\\
We recall that the densities $\boldsymbol{\rho}_{i}$ are defined for all $i\in\{1,2,\cdots,N\}$ as $\boldsymbol{\rho}_{i}=\rho_{\surf}+\frac{i-\frac{1}{2}}{N}(\rho_{\bott}-\rho_{\surf})$, and in this proof we extend this definition to $i\in\{\frac{1}{2},\frac{3}{2},\cdots,N+\frac{1}{2}\}$. In this proof, we denote by ${C(\rho_{\surf},\rho_{\bott})>0}$ a constant that depends only on $\rho_{\surf}$ and $\rho_{\bott}$ and which will increase if necessary throughout the proof. For $i\in\{1,\cdots,N\}$ fixed and for almost every $(t,x)\in [0,T]\times\mathbb{R}$ with $T$ the time of existence and uniqueness of the solution $(h,u)$, integration by parts yields 
\begin{align*}
(\mathcal{M}\partial_x h)(x,\boldsymbol{\rho}_i)&=\rho_{\surf}\int\limits_{\rho_{\surf}}^{\rho_{\bott}}\partial_xh(x,\varrho)d\varrho+\int\limits_{\rho_{\surf}}^{\boldsymbol{\rho}_i}\int\limits_{\varrho}^{\rho_{\bott}}\partial_xh(x,\varrho')d\varrho'd\varrho\\
&=\int\limits_{\rho_{\surf}}^{\boldsymbol{\rho}_i}\varrho \partial_xh(x,\varrho)d\varrho+\boldsymbol{\rho}_i\int\limits_{\boldsymbol{\rho}_i}^{\rho_{\bott}}\partial_xh(x,\varrho)d\varrho\\
&=\int\limits_{\rho_{\surf}}^{\rho_{\bott}}\min(\varrho,\boldsymbol{\rho}_i)\partial_xh(x,\varrho)d\varrho.
\end{align*}
This yields the following identites.

Firstly, one has
\begin{align*}
\big(\boldsymbol{\rho}\Lambda^s\boldsymbol{R}_N\big)_i
&=-\sum_{j=1}^N\int\limits_{\boldsymbol{\rho}_{j-\frac{1}{2}}}^{\boldsymbol{\rho}_{j+\frac{1}{2}}}\min(\varrho,\boldsymbol{\rho}_i)\partial_x\Lambda^sh(\cdot,\varrho)-\min(\boldsymbol{\rho}_i,\boldsymbol{\rho}_j)\partial_x \Lambda^sh(\cdot,\boldsymbol{\rho}_j)d\varrho\\
&=-\sum\limits_{j=1}^{i-1}\int\limits_{\boldsymbol{\rho}_{j-\frac{1}{2}}}^{\boldsymbol{\rho}_{j+\frac{1}{2}}}(\varrho\partial_x\Lambda^sh(\cdot,\varrho)-\boldsymbol{\rho}_j\partial_x\Lambda^sh(\cdot,\boldsymbol{\rho}_j))d\varrho-\boldsymbol{\rho}_i\sum\limits_{j=i+1}^{N}\int\limits_{\boldsymbol{\rho}_{j-\frac{1}{2}}}^{\boldsymbol{\rho}_{j+\frac{1}{2}}}(\partial_x\Lambda^sh(\cdot,\varrho)-\partial_x\Lambda^sh(\cdot,\boldsymbol{\rho}_j))d\varrho\\
&\hspace{0,5cm}-\int\limits_{\boldsymbol{\rho}_{i-\frac{1}{2}}}^{\boldsymbol{\rho}_{i}}(\varrho\partial_x\Lambda^sh(\cdot,\varrho)-\boldsymbol{\rho}_i\partial_x\Lambda^sh(\cdot,\boldsymbol{\rho}_i))d\varrho-\boldsymbol{\rho}_i\int\limits_{\boldsymbol{\rho}_{i}}^{\boldsymbol{\rho}_{i+\frac{1}{2}}}(\partial_x\Lambda^sh(\cdot,\varrho)-\partial_x\Lambda^sh(\cdot,\boldsymbol{\rho}_i))d\varrho.
\end{align*}
Hence by numerical integration midpoint rule (applied to the first two terms) and the rectangle rule (applied to the last two terms) we infer
\begin{align}\label{restd0}
\n\left(\Lambda^s\boldsymbol{R}_N\right)_i\n_{L^2(\mathbb{R}^d)}\leq&\frac{C(\rho_{\surf},\rho_{\bott})}{N^2}\left(\n\partial_\varrho\partial_x\Lambda^sh(t,\cdot,\cdot)\n_{L^\infty_\varrho(L^2_x)}+\n\partial^2_\varrho\partial_x\Lambda^sh(t,\cdot,\cdot)\n_{L^\infty_\varrho(L^2_x)}\right).
\end{align}

Secondly, one has
\begin{align*}
\bigg(\Dp(\boldsymbol{\rho}\Lambda^{s-1}\boldsymbol{R}_N)\bigg)_i&=N \left(\left(\boldsymbol{\rho}\Lambda^{s-1}\boldsymbol{R}_N\right)_i-\left(\boldsymbol{\rho}\Lambda^{s-1}\boldsymbol{R}_N\right)_{i+1}\right)\\&
=-N\sum_{j=1}^N\int\limits_{\boldsymbol{\rho}_{j-\frac{1}{2}}}^{\boldsymbol{\rho}_{j+\frac{1}{2}}}q_i(\varrho)\partial_x\Lambda^{s-1}h(\cdot,\varrho)- q_i(\boldsymbol{\rho}_j)\partial_x \Lambda^{s-1}h(\cdot,\boldsymbol{\rho}_j)d\varrho
\end{align*}
where $ q_i:\varrho\mapsto\min(\varrho,\boldsymbol{\rho}_i)-\min(\varrho,\boldsymbol{\rho}_{i+1})$. After simple computations we infer 
\begin{align*}
\bigg(\Dp(\boldsymbol{\rho}\Lambda^{s-1}\boldsymbol{R}_N)\bigg)_i=&-N\Bigg(\int\limits_{\boldsymbol{\rho}_{i}}^{\boldsymbol{\rho}_{i+1}}(\boldsymbol{\rho}_i-\varrho)\partial_x\Lambda^{s-1}h(\cdot,\varrho)-{2N}(\boldsymbol{\rho}_i-\boldsymbol{\rho}_{i+1})\partial_x\Lambda^{s-1}h(\cdot,\boldsymbol{\rho}_{i+1})d\varrho\Bigg)\\
&-N(\boldsymbol{\rho}_i-\boldsymbol{\rho}_{i+1})\Bigg(\int\limits_{\boldsymbol{\rho}_{i+1}}^{\boldsymbol{\rho}_{i+\frac{3}{2}}}\partial_x\Lambda^{s-1}h(\cdot,\varrho)-\partial_x\Lambda^{s-1}h(\cdot,\boldsymbol{\rho}_{i+1})d\varrho\Bigg)\\
&-\sum\limits_{j=i+2}^N\Bigg(\int\limits_{\boldsymbol{\rho}_{j-\frac{1}{2}}}^{\boldsymbol{\rho}_{j+\frac{1}{2}}}\partial_x\Lambda^{s-1}h(\cdot,\varrho)-\partial_x\Lambda^{s-1}h(\cdot,\boldsymbol{\rho}_{j})d\varrho\Bigg).
\end{align*}
Consequently, using the Taylor expansion of $\varrho\mapsto \partial_x\Lambda^{s-1}h(\cdot,\varrho)$ up to the first order in the neighborhood of $\boldsymbol{\rho}_{i+1}$ and up to the second order in the neighborhood of $\boldsymbol{\rho}_{j}$ for $j\in\{i+2,\cdots,N\}$,
we obtain
\begin{align*}
\n\left(\Dp\left(\boldsymbol{\rho}\Lambda^{s-1}\boldsymbol{R}_N\right)\right)_i\n_{L^2_x} \leq 
&\frac{C(\rho_{\bott})}{N^2}\left(\n\partial_\varrho\partial_x\Lambda^{s-1}h(t,\cdot,\cdot)\n_{L^\infty_\varrho(L^2_x)}+\n\partial^2_\varrho\partial_x\Lambda^{s-1}h(t,\cdot,\cdot)\n_{L^\infty_\varrho(L^2_x)}\right).
\end{align*}
Using Lemma \ref{IL} (2) and the fact that $\Dp \boldsymbol{\rho}=-(1,\cdots,1)^t\in \R^{N-1}$ in addition to the previous estimate \eqref{restd0}, we infer
\begin{align}\label{restd1}
\n\left(\Dp\Lambda^{s-1}\boldsymbol{R}_N\right)_i\n_{L^2_x} \leq 
&\frac{C(\rho_{\surf},\rho_{\bott})}{N^2}\left(\n\partial_\varrho\partial_x\Lambda^{s-1}h(t,\cdot,\cdot)\n_{L^\infty_\varrho(L^2_x)}+\n\partial^2_\varrho\partial_x\Lambda^{s-1}h(t,\cdot,\cdot)\n_{L^\infty_\varrho(L^2_x)}\right).
\end{align}

Lastly, one has
\begin{align*}
\bigg(\Ds(\boldsymbol{\rho}\Lambda^{s-2}\boldsymbol{R}_N)\bigg)_i&=N \left( \left(\Dp(\boldsymbol{\rho}\Lambda^{s-2}\boldsymbol{R}_N)\right)_i-\left(\Dp(\boldsymbol{\rho}\Lambda^{s-2}\boldsymbol{R}_N)\right)_{i+1}\right)\\
&=-N^2\sum_{j=1}^N\left(\int\limits_{\boldsymbol{\rho}_{j-\frac{1}{2}}}^{\boldsymbol{\rho}_{j+\frac{1}{2}}}k_i(\varrho)\partial_x\Lambda^{s-2}h(\cdot,\varrho)- k_i(\boldsymbol{\rho}_j)\partial_x \Lambda^{s-2}h(\cdot,\boldsymbol{\rho}_j)d\varrho\right),
\end{align*}
where $ k_i:\varrho \mapsto\min(\varrho,\boldsymbol{\rho}_i)-2\min(\varrho,\boldsymbol{\rho}_{i+1})+\min(\varrho,\boldsymbol{\rho}_{i+2})$.
After simple computations we infer
\begin{align*}
\bigg(\Ds(\boldsymbol{\rho}\Lambda^{s-2}\boldsymbol{R}_N)\bigg)_i=&-N^2\Bigg(\int\limits_{\boldsymbol{\rho}_{i}}^{\boldsymbol{\rho}_{i+1}}(\boldsymbol{\rho}_i-\varrho)\partial_x\Lambda^{s-2}h(\cdot,\varrho)d\varrho+\int\limits_{\boldsymbol{\rho}_{i+1}}^{\boldsymbol{\rho}_{i+2}}(\boldsymbol{\rho}_i-2\boldsymbol{\rho}_{i+1}+\varrho)\partial_x\Lambda^{s-2}h(\cdot,\varrho)d\varrho\\
&\hspace{1cm}-\frac{1}{N}(\boldsymbol{\rho}_i-\boldsymbol{\rho}_{i+1})\partial_x\Lambda^{s-2}h(\cdot,\boldsymbol{\rho}_{i+1})\Bigg).
\end{align*} 
Consequently, using the Taylor expansion of $\varrho\mapsto \partial_x\Lambda^{s-2}h(\cdot,\varrho)$ up to the second order in the neighborhood of $\boldsymbol{\rho}_{i+1}$, we obtain
\begin{equation*}
\|\Ds(\boldsymbol{\rho}\Lambda^{s-2}\boldsymbol{R}_N)\|_{L^2_x}\leq
\frac{C(\rho_{\bott})}{N^2}\left(\n\partial_\varrho\partial_x\Lambda^{s-2}h(t,\cdot,\cdot)\n_{L^\infty_\varrho(L^2_x)}+\n\partial^2_\varrho\partial_x\Lambda^{s-2}h(t,\cdot,\cdot)\n_{L^\infty_\varrho(L^2_x)}\right).
\end{equation*}
Using Lemma \ref{IL} (3) and the fact that $\Ds \boldsymbol{\rho}=(0,\cdots,0)^t\in \R^{N-2}$ in addition to the previous estimates \eqref{restd0}, \eqref{restd1}, we infer 
\begin{equation*}
\|\Ds(\Lambda^{s-2}\boldsymbol{R}_N)\|_{L^2_x}\leq
\frac{C(\rho_{\surf},\rho_{\bott})}{N^2}\left(\n\partial_\varrho\partial_x\Lambda^{s-2}h(t,\cdot,\cdot)\n_{L^\infty_\varrho(L^2_x)}+\n\partial^2_\varrho\partial_x\Lambda^{s-2}h(t,\cdot,\cdot)\n_{L^\infty_\varrho(L^2_x)}\right).
\end{equation*}
The proof is complete.
\end{proof}

\subsection{Stability}\label{section5.2}
In this subsection we will provide the key ingredients towards stability estimates on the difference between the solutions to the continuously stratified system \eqref{contsystem} (after applying the projection operator $P_N$) and the corresponding solutions to  \eqref{multicouches}. These stability estimates are obtained by considering the linearized system satisfied by the difference and their derivatives, carefully estimating the remainders that result from this linearization. In the following Lemma we estimate these remainder terms. 

\begin{lemma}\label{petitsr}
Let $s\in\N$, $s>2+\frac{1}{2}$, there exists $C>0$, such that for any $N\in\N^*$, $\kappa>0$, and any $(\underline{H},H, U)$ and $(\underline h,h,u)$ sufficiently smooth respectively on $\R$ and $\Omega:=(\rho_{\surf},\rho_{\bott})\times\R$, setting 
\begin{align*}
&R'=-(U-P_Nu)\partial_xP_Nh-(H-P_Nh)\partial_xP_Nu,\\
&\boldsymbol{R}'=-\left((U-P_Nu)-\kappa\left(\frac{\partial_xH}{\underline{H}+H}-P_N\left(\frac{\partial_xh}{\underline{h}+h}\right)\right)\right)\partial_xP_Nu,
\end{align*}
the following estimates hold.
\begin{enumerate}
\item 
\begin{align*}
\n R'\n_{H^{s-1,1}}\leq C\n U-P_Nu\n_{H^{s-1,1}}\n \partial_xh\n_{\infty,s-1,1}+C\n H-P_Nh\n_{H^{s-1,1}} \n \partial_xu\n_{\infty,s-1,1}.
\end{align*}
\item
\begin{align*}
\n {\sf S}R'\n_{H^{s,0}}+\n {\sf T}{\sf S}R'\n_{H^{s,0}}\leq C\n U-P_Nu\n_{H^{s,0}}\n \partial_xh\n_{\infty,s,0}+C\n {\sf S}( H-P_Nh)\n_{H^{s,0}} \n \partial_xu\n_{\infty,s-1,1}&\\
+C\n {\sf T}{\sf S}( H-P_Nh)\n_{H^{s,0}} \n \partial_xu\n_{\infty,s-1,1}+C\n  H-P_Nh\n_{H^{s-1,0}} \n \partial_xu\n_{\infty,s,0}.&
\end{align*}
\item
\begin{align*}
\n\boldsymbol{R}'\n_{H^{s,0}}\leq C\n U-P_Nu\n_{H^{s,0}}\n \partial_xu\n_{\infty,s,0}+C\kappa\left\|\frac{\partial_xH}{\underline{H}+H}-P_N\left(\frac{\partial_xh}{\underline{h}+h}\right)\right\|_{H^{s,0}}\n \partial_xu\n_{\infty,s,0}.
\end{align*}
\item
\begin{align*}
\n\boldsymbol{R}'\n_{H^{s,2}}\leq  C\n U-P_Nu\n_{H^{s,2}}\n \partial_xu\n_{\infty,s,2}+C\kappa\left\|\frac{\partial_xH}{\underline{H}+H}-P_N\left(\frac{\partial_xh}{\underline{h}+h}\right)\right\|_{H^{s,2}}\n \partial_xu\n_{\infty,s,2}.
\end{align*}
\end{enumerate}
\end{lemma}

\begin{proof} The proof of 1., 3., and 4. 
follows immediately from Lemma \ref{IL}, Lemma \ref{comnatu} $(2)$ and Lemma  \ref{op P_N}. For 2. we use  Lemma \ref{IL} (1) to infer the identity valid for $0\leq\alpha\leq s$
\begin{align*}
\partial_x^\alpha{\sf S} \mathrm{R}'=&{\sf S}\partial_x^\alpha\bigg((U-P_Nu)(\partial_xP_Nh)\bigg)-(\partial_xP_Nu)({\sf S}\partial_x^\alpha(H-P_Nh))\\
&+{\sf S}_0\left((\Dp(\partial_xP_Nu))(\ru{\sf S}\partial_x^\alpha(H-P_Nh))\right)-{\sf S}\left([\partial_x^\alpha,\partial_xP_Nu](H-P_Nh)\right).
\end{align*}
The estimate on $\n {\sf S}R'\n_{H^{s,0}}$ follows from Lemma \ref{contop}, Lemma \ref{comnatu} $(2)$, and Lemma  \ref{op P_N}. The estimate for $\n {\sf T}{\sf S}R'\n_{H^{s,0}}$ follows in the same way from the analogous identity.
\end{proof}

\begin{lemma}\label{diffquo}
Let $h_*, \underline{M}, M_{\rm c}, M>0$, $s\in\N$ such that  $s> 2+\frac{1}{2}$. There exists ${C(h_*^{-1}, \underline{M},M_{\rm c},M)>0}$ such that for any $N\in\N^*$, and any  $\underline{h}\in W^{2,\infty}((\rho_{\surf},\rho_{\bott}))$, $\underline{H}=P_N(\underline{h})$, and any $h$ and $H$ sufficiently smooth respectively on $\R$ and $\Omega:=(\rho_{\surf},\rho_{\bott})\times\R$, satisfying  
\begin{equation*}
\br\underline{H}\br_{w^{2,\infty}}\leq \underline{M},\quad \n H\n_{H^{s-1,1}}\leq M,\quad \n h\n_{\infty,s-1,1}\leq M_{\rm c},
\end{equation*}
\begin{equation*}
\underset{(i,x)\in\{1,\cdots,N\}\times \mathbb{R}}{\inf}\underline{H}_i+H_i(x)\geq h_*,\: \underset{(x,\varrho)\in\R\times (\rho_{\surf},\rho_{\bott})}{\inf}\underline{h}(\varrho)+h(x,\varrho)\geq h_*,
\end{equation*} 
one has
\begin{align*}\bigg\| &\frac{\partial_xH}{\underline{H}+H}-P_N\bigg(\frac{\partial_xh}{\underline{h}+h}\bigg)\bigg\|_{H^{s,2}}\leq C(h_*^{-1}, \underline{M}, M_{\rm c},M)\times\\
&\hspace{1cm}\bigg(\n\partial_x(H-P_Nh)\n_{H^{s,2}}+\n H-P_Nh\n_{H^{s,2}}(\n H\n_{H^{s,2}}+\n \partial_xh\n_{\infty,s-1,1})\\
&\hspace{1cm}+\n H-P_Nh\n_{H^{s-1,1}}\n \partial_xh\n_{\infty,s,2}+\n H-P_Nh\n_{H^{s-1,1}}\n \partial_x h\n_{\infty, s-1,1}(\n H\n_{H^{s,2}}+\n h\n_{\infty,s,2})\bigg).
\end{align*}
\end{lemma}

\begin{proof}
We notice that 
\begin{align*}
&\frac{\partial_xH}{\underline{H}+H}-P_N\bigg(\frac{\partial_xh}{\underline{h}+h}\bigg)=\partial_x(H-P_Nh)\left(\frac{1}{\underline{H}+H}-\frac{1}{\underline{H}} \right)+\frac{\partial_x(H-P_Nh)}{\underline{H}}\\
&\hspace{2cm}+\frac{\partial_xP_Nh}{\underline{H}^2}\left(P_Nh-H\right)+\partial_xP_Nh\left(\frac{1}{(\underline{H}+H)(\underline{H}+P_Nh)}-\frac{1}{\underline{H}^2} \right)(P_Nh-H),
\end{align*}
and
\begin{align*}
\frac{1}{(\underline{H}+H)(\underline{H}+P_Nh)}-\frac{1}{\underline{H}^2}=&\left(\frac{1}{\underline{H}+H}-\frac{1}{\underline{H}}\right)\left(\frac{1}{\underline{H}+P_Nh}-\frac{1}{\underline{H}}\right)\\
&\hspace{2,5cm}+\frac{1}{\underline{H}}\left[\left(\frac{1}{\underline{H}+H}-\frac{1}{\underline{H}}\right)+\left(\frac{1}{\underline{H}+P_Nh}-\frac{1}{\underline{H}}\right)\right].
\end{align*}

Since $s>2+\frac{1}{2}$, using Lemma \ref{AlgebreBanach} (2), Lemma \ref{composHsk} and Lemma \ref{op P_N} we obtain the desired estimate.
\end{proof}

We can now collect all estimates on remainders which will be used in the proof of our second main result, namely Theorem \ref{mainth}.
\begin{lemma}\label{restdiff}
Let $\rho_{\bott}>\rho_{\surf}>0$,  $s\in\N$, $s>2+\frac{1}{2}$, $h_*,\underline{M}, M_{\rm c},M>0$. Then there exists $C>0$ such that for any $N\in\N^*$, $\kappa\in(0,1]$, and any $(h,u)$ solution to \eqref{contsystem} with $\underline{h},\underline{u}\in W^{2,\infty}((\rho_{\surf},\rho_{\bott}))$  and $(H,U)$ solution to \eqref{multicouches} with $\underline{H}=P_N\underline{h},\underline{U}=P_N\underline{u}$, assuming that these solutions both exist on a time interval $[0,T)$ with $T>0$, and that the following estimates hold for any $t\in [0,T)$
\begin{align*}
&\br\boldsymbol{\rho}\br_{l^\infty}  +\br\boldsymbol{\rho}^{-1}\br_{l^\infty} + \br\underline{H}\br_{w^{2,\infty}}+\br\underline{U}\br_{w^{2,\infty}}+\br\underline{h}\br_{W^{2,\infty}}+\br\underline{u}\br_{W^{2,\infty}}\leq \underline{M},\\
&\n H(t,\cdot)\n_{H^{s-1,1}}+\n U(t,\cdot)\n_{H^{s,2}}+\kappa^\frac{1}{2}\n H(t,\cdot)\n_{H^{s,2}}\leq M,\\
&\n h(t,\cdot)\n_{\infty,s+1,2}+\n \partial_x u(t,\cdot)\n_{\infty,s,2}\leq M_{\rm c},
\end{align*}
\begin{equation*}
\underset{(i,x)\in\{1,\cdots,N\}\times \mathbb{R}}{\inf}\underline{H}_i+H_i(t,x)\geq h_*,\: \underset{(x,\varrho)\in\R\times [\rho_{\surf},\rho_{\bott}]}{\inf}\underline{h}(\varrho)+h(t,x,\varrho)\geq h_*,
\end{equation*} 
then the following holds.
\begin{itemize}
\item For all $\alpha\in\N$, $j\in \{0,1\}$ with $0\leq\alpha\leq s-1-j$, we have
\begin{align*}
\partial_t\partial_x^\alpha\Dp^j(H-P_Nh)+({\sf M}^{j}(\underline{U}+U))\partial_x\partial_x^\alpha\Dp^j(H-P_Nh)=\kappa\partial_x^2\partial_x^\alpha\Dp^j(H-P_Nh)+\mathcal{R}_{\alpha,j}'
\end{align*}

where for every $t\in[0,T]$, $\mathcal{R}_{\alpha,j}'(t)\in l^2(L^2(\R))$ and
\begin{align*}
\n\mathcal{R}_{\alpha,j}'(t,\cdot)\n_{l^2(L^2_x)}\leq& C(\n U-P_Nu\n_{H^{s,2}}+\n H-P_Nh\n_{H^{s-1,1}}).
\end{align*}

\item For all $\alpha\in\N$ with $0\leq\alpha\leq s$, we have
\begin{align*}
&\partial_t\partial_x^\alpha(H-P_Nh)+(\underline{U}+U)\partial_x\partial_x^\alpha(H-P_Nh)
+(\underline{H}+H)\,\partial_x\partial_x^\alpha (U-P_Nu)=\kappa\partial_x^2\partial_x^\alpha(H-P_Nh)+ R_{\alpha,0}',
\\&\partial_t\partial_x^\alpha( U-P_Nu)+\left(\underline{U}+U-\kappa\frac{\partial_xH}{\underline{H}+H}\right)\partial_x \partial_x^\alpha( U-P_Nu)+{\sf \Gamma} \partial_x\partial_x^\alpha(H-P_Nh)=\textbf{R}_{\alpha,0}',
\end{align*}
where for every $t\in[0,T]$, $({\sf S} R_{\alpha,0}'(t),{\sf T}{\sf S} R_{\alpha,0}'(t),\textbf{R}_{\alpha,0}'(t))\in l^2(L^2(\R))^3$ and 
\begin{align*}
\n {\sf S}R_{\alpha,0}'(t,\cdot)&\n_{l^2(L^2_x)}+\n {\sf T}{\sf S}R_{\alpha,0}'(t,\cdot)\n_{l^2(L^2_x)}\leq\\&C(\n U-P_Nu\n_{H^{s,1}}+\n {\sf S}(H-P_Nh)\n_{H^{s,0}}+\n {\sf T}{\sf S}(H-P_Nh)\n_{H^{s,0}}+\n H-P_Nh\n_{H^{s-1,0}}).
\end{align*}
\begin{align*}
\n \boldsymbol{R}_{\alpha,0}'\n_{l^2(L^2_x)}&\leq \n \boldsymbol{R}_N\n_{H^{s,2}}+C(1+\kappa\n\partial_xH\n_{H^{s,1}}) \n U-P_Nu\n_{H^{s,1}}\\
&\quad +C\big(\n H-P_Nh\n_{H^{s-1,1}} + \kappa^\frac{1}{2}\n H-P_Nh\n_{H^{s,2}} +\kappa\n\partial_x(H-P_Nh)\n_{H^{s,2}}\big).
\end{align*}

\item For any $\alpha\in \N$, $j\in\{1,2\}$ such that $0\leq\alpha\leq s-j$, it holds
\begin{align*}
&\partial_t \Dp^j {\sf S}\partial_x^\alpha(H-P_Nh)+\rd({\sf M}^{j-1}(\underline{U}+U))\Dp^j {\sf S}(\partial_x^\alpha\partial_x(H-P_Nh))\\
&\hspace{7cm}=\kappa\partial_x^2(\partial_x^\alpha \Dp^j{\sf S}(H-P_Nh))+R_{\alpha,j}'\\
&\partial_t\Dp^j\partial_x^\alpha (U-P_Nu)+ {\sf M}^j\left(\underline{U}+U-\kappa\frac{\partial_x H}{\underline{H}+H}\right)\partial_x (\Dp^j\partial_x^\alpha (U-P_Nu))=\textbf{R}_{\alpha,j}'
\end{align*}
where for every $t\in[0,T]$, $( R_{\alpha,j}'(t),\textbf{R}_{\alpha,j}'(t))\in l^2(L^2(\R))\times l^2(L^2(\R))$ and 
\begin{align*}
\n R_{\alpha,j}'(t,\cdot)\n_{l^2(L^2_x)}\leq&C(\n U-P_Nu\n_{H^{s,2}}+\n H-P_Nh\n_{H^{s-1,1}}).
\end{align*}
\begin{align*}
\n \boldsymbol{R}'_{\alpha,j}\n_{l^2(L^2_x)}\leq&\n \boldsymbol{R}_N\n_{H^{s,2}}+C(1+\kappa\n\partial_xH\n_{H^{s,2}}) \n U-P_Nu\n_{H^{s,2}}\\
&+C(\n{\sf T}{\sf S}(H-P_Nh)\n_{H^{s,0}}+\n {\sf S}(H-P_Nh)\n_{H^{s,1}})\\
&+C\big(\n H-P_Nh\n_{H^{s-1,1}} + \kappa^\frac{1}{2}\n H-P_Nh\n_{H^{s,2}} +\kappa\n\partial_x(H-P_Nh)\n_{H^{s,2}}\big).
\end{align*}

\item For any $\alpha\in \N$, $j\in\{0,1,2\}$, $0\leq\alpha\leq s-j$ , it holds
\begin{align*}
&\partial_t \Dp^j\partial_x^\alpha(H-P_Nh)+{\sf M}^j(\underline{U}+U)\;\partial_x( \Dp^j\partial_x^\alpha(H-P_Nh))=\kappa\partial_x^2(\partial_x^\alpha  \Dp^j(H-P_Nh))+r_{\alpha,j}'+\partial_x \textbf{r}_{\alpha,j}'
\end{align*}

where for every $t\in[0,T]$, $( r_{\alpha,j}'(t),\textbf{r}_{\alpha,j}'(t))\in l^2(L^2(\R))\times l^2(L^2(\R))$ and 
\begin{align*}
\n r_{\alpha,j}'\n_{l^2(L^2_x)}&\leq C \n H-P_Nh\n_{H^{s,2}},\\
\n\mathbf{r}_{\alpha,j}'\n&\leq C (\n U-P_Nu\n_{H^{s,2}}+\n H-P_Nh\n_{H^{s-1,1}}).
\end{align*}
\end{itemize}
Above, we denote by $\boldsymbol{R}_N$ the term defined in \eqref{R_N}.

\end{lemma}

\begin{proof} 
Under the hypothesis of the Lemma, $(P_Nh,P_Nu)$ is a solution to $($\ref{smoothsoldiscretised}$)$, consequently $(H-P_Nh,U-P_Nu)$ satisfy the following system
\begin{equation}\label{eqofdifference}
\left\{\begin{array}{l}
\partial_t(H-P_Nh)+(\underline{U}+U)\partial_x(H-P_Nh)+(\underline{H}+H)\partial_x(U-P_Nu)=\kappa\partial^2_x(H-P_Nh)+R',\\ 
\partial_t(U-P_Nu)+\left(\underline{U}+U-\kappa\frac{\partial_xH}{\underline{H}+H}\right)\partial_x(U-P_Nu)+{\sf \Gamma}\partial_x(H-P_Nh)=-\boldsymbol{R}_N+\boldsymbol{R}',
\end{array}     
\right.
\end{equation}
where \begin{align*}
\boldsymbol{R}_N&={\sf \Gamma}\partial_xP_Nh-P_N\left(\frac{1}{\varrho}\partial_x\psi\right),\\
R'&=-(U-P_Nu)\partial_xP_Nh-(H-P_Nh)\partial_xP_Nu,\\
\boldsymbol{R}'&=-\left((U-P_Nu)-\kappa\left(\frac{\partial_xH}{\underline{H}+H}-P_N\left(\frac{\partial_xh}{\underline{h}+h}\right)\right)\right)\partial_xP_Nu.
\end{align*}
The contribution from $\boldsymbol{R}_N$ is trivial in this proof. We have estimated $R'$ and $\boldsymbol{R}'$ in Lemma \ref{petitsr} and Lemma \ref{diffquo}. The remaining contributions are estimated following the same steps as in Lemma \ref{restes}.\\

\noindent Estimates for $ \mathcal{R}_{\alpha,j}'$,  for all $\alpha\in\N$, $j\in \{0,1\}$ with $0\leq\alpha\leq s-1-j$, we have
\begin{align*}
\mathcal{R}_{\alpha,j}'=\partial_x^\alpha\Dp^j \mathrm{R}'-\llbracket\partial_x^\alpha\Dp^j,\underline{U}+U\rrbracket\partial_x(H-P_Nh)-\partial_x^\alpha\Dp^j\left((\underline{H}+H)\partial_x(U-P_Nu)\right),
\end{align*}
we obtain an estimate of the above by using the same estimates as $\mathcal{R}_{\alpha,j}$ (adapted to our case) in Lemma \ref{restes} and then applying Lemma \ref{petitsr} and Lemma \ref{op P_N}, consequently we have
\begin{align*}
\n\mathcal{R}_{\alpha,j}'(t,\cdot)\n_{l^2(L^2_x)}\lesssim&(\n \partial_xh\n_{\infty,s-1,1}+\br\underline{H}\br_{w^{1,\infty}}+\n H\n_{H^{s-1,1}})\n U-P_Nu\n_{H^{s,2}}\\
&+(\n \partial_xu\n_{\infty,s-1,1}+\br\Dp\underline{U}\br_{l^\infty}+\n U\n_{H^{s,2}})\n H-P_Nh\n_{H^{s-1,1}}.
\end{align*}
Estimates for ${\sf S}R_{\alpha,0}'$, ${\sf T}{\sf S}R_{\alpha,0}'$ and $\boldsymbol{R}_{\alpha,0}'$ with $0\leq\alpha\leq s$.\\

For all $\alpha\in\N$ with $0\leq\alpha\leq s$, using Abel’s summation Lemma \ref{IL}  we have
\begin{align*}
{\sf S}R_{\alpha,0}'=&\partial_x^\alpha{\sf S}R'-[\partial_x^\alpha,U,\partial_x{\sf S}(H-P_Nh)]-(\partial_x(U-P_Nu))(\partial_x^\alpha {\sf S}H)\\
&+{\sf S}_0\bigg([\partial_x^\alpha;\Dp U,\ru\partial_x{\sf S}(H-P_Nh)]+(\Dp(\partial_x(U-P_Nu)))(\ru{\sf S}\partial_x^\alpha H)\bigg)\\
&-{\sf S}\bigg((\partial_x^\alpha U)(\partial_x(H-P_Nh))+[\partial_x^\alpha;H ,\partial_x(U-P_Nu)]\bigg),
\end{align*}
\begin{align*}
{\sf T}{\sf S}R_{\alpha,0}'=&\partial_x^\alpha{\sf T}{\sf S}R'-[\partial_x^\alpha,U,\partial_x{\sf T}{\sf S}(H-P_Nh)]-(\partial_x(U-P_Nu))(\partial_x^\alpha {\sf T}{\sf S}H)\\
&+{\sf T}{\sf S}_0\bigg([\partial_x^\alpha;\Dp U,\ru\partial_x{\sf S}(H-P_Nh)]+(\Dp(\partial_x(U-P_Nu)))(\ru{\sf S}\partial_x^\alpha H)\bigg)\\
&-{\sf T}{\sf S}\bigg((\partial_x^\alpha U)(\partial_x(H-P_Nh))+[\partial_x^\alpha;H ,\partial_x(U-P_Nu)]\bigg),
\end{align*}
and 
\begin{align*}
\boldsymbol{R}_{\alpha,0}'=&-\partial_x^\alpha\boldsymbol{R}_N+\partial_x^\alpha\boldsymbol{R}'-[\partial_x^\alpha,U]\partial_x(U-P_Nu)+\frac{\kappa}{\underline{H}+H}[\partial_x^\alpha,\partial_xH]\partial_x(U-P_Nu)\\&+
\kappa\left[\partial_x^\alpha,\frac{1}{\underline{H}+H}\right]\big((\partial_xH)(\partial_x(U-P_Nu))\big).
\end{align*}
We obtain an estimate using the previous identities by using the same estimates of ${\sf S}R_{\alpha,0}$, ${\sf T}{\sf S}R_{\alpha,0}$ and $\boldsymbol{R}_{\alpha,0}$ (adapted to our case) in Lemma \ref{restes}. Moreover, applying Lemma \ref{petitsr} we find
\begin{align*}
\n {\sf S}R_{\alpha,0}'(t,\cdot)&\n_{l^2(L^2_x)}+\n {\sf T}{\sf S}R_{\alpha,0}'(t,\cdot)\n_{l^2(L^2_x)}\\&\lesssim(\n\partial_x h\n_{\infty,s,0}+\n {\sf T}{\sf S}H\n_{H^{s,0}}+\n {\sf S}H\n_{H^{s,0}}+\n H\n_{H^{s-1,0}})\n U-P_Nu\n_{H^{s,1}}\\&\hspace{0,5cm}+(\n\partial_x u\n_{\infty,s-1,1}+\n U\n_{H^{s,1}})(\n {\sf S}(H-P_Nh)\n_{H^{s,0}}+\n {\sf T}{\sf S}(H-P_Nh)\n_{H^{s,0}})\\
&\hspace{0,5cm}+(\n\partial_x u\n_{\infty,s,0}+\n U\n_{H^{s,0}})\n H-P_Nh\n_{H^{s-1,0}}.
\end{align*}
In the same way, and  using additionally Lemma \ref{diffquo}, we find
\begin{align*}
\n &\boldsymbol{R}_{\alpha,0}'\n_{l^2(L^2_x)}\lesssim \n \boldsymbol{R}_N\n_{H^{s,2}}\\&
+C(h_*^{-1},\n H\n_{H^{s-1,1}})(\n U\n_{H^{s,1}}+\n \partial_x u\n_{\infty,s,0}+\kappa(\n H\n^2_{H^{s,1}}+\n\partial_xH\n_{H^{s,1}}))\times \n U-P_Nu\n_{H^{s,1}}\\
&+C(h_*^{-1},\underline{M},M_{\rm c},M) \kappa \n \partial_xu\n_{\infty,s,0}\\
&\times\bigg(\n\partial_x(H-P_Nh)\n_{H^{s,2}}+\n H-P_Nh\n_{H^{s,2}}(\n H\n_{H^{s,2}}+\n \partial_xh\n_{\infty,s-1,1})\\
&\hspace{0,5cm}+\n H-P_Nh\n_{H^{s-1,1}}\n \partial_xh\n_{\infty,s,2}+\n H-P_Nh\n_{H^{s-1,1}}\n \partial_x h\n_{\infty,s-1,1}(\n H\n_{H^{s,2}}+\n h\n_{\infty,s,2})\bigg).
\end{align*}
Estimates for $ R_{\alpha,j}'$, $j\in\{1,2\}$, for any $\alpha\in \N$,  such that $0\leq\alpha\leq s-j$.\\

\noindent Recalling that $\Dp {\sf S}=\rd$ and $\Ds {\sf S}=\rd\Dp $ we have
\begin{align*}
R_{\alpha,1}'&= \rd\partial_x^\alpha R'+\rd\left(-[\partial_x^{\alpha+1},\underline{U}+U]\partial_x^\alpha(H-P_Nh)-\partial_x^\alpha\left((\underline{H}+H)\partial_x(U-P_N u)\right)\right)\\
&=\rd\partial_x^\alpha R'+\rd\left(\partial_x^\alpha((\partial_x U)(H-P_Nh))-\partial_x^\alpha((H)(\partial_x (U-P_Nu)))\right)\\
&\hspace{0,5cm}+\rd\left(-[\partial_x^{\alpha+1}, U](H-P_Nh)-\partial_x^\alpha((\underline{H})(\partial_x(U-P_Nu))\right)
\end{align*}
and
\begin{align*}
R_{\alpha,2}'=&\rd\Dp\partial_x^\alpha R'+\rd\bigg(-\llbracket\partial_x^\alpha\Dp,\underline{U}+U\rrbracket\partial_x(H-P_Nh)-\partial_x^\alpha\Dp\left((\underline{H}+H)(\partial_x (U-P_Nu))\right)\bigg).
\end{align*}
Using Lemma \ref{comnatu} $(1)$ to estimate the second and third term of $R_{\alpha,1}'$, and Lemma \ref{petitsr} to estimate the contributions of $R'$ and the same estimates as $R_{\alpha,1}$, $R_{\alpha,2}$  (adapted to our case) in Lemma \ref{restes} for the remaining contributions, we find
\begin{align*}
\n R_{\alpha,j}'(t,\cdot)\n_{l^2(L^2_x)}\lesssim&(\n\partial_x h\n_{\infty,s-1,1}+\n H\n_{H^{s-1,1}}+\br\underline{H}\br_{w^{1,\infty}})\n U-P_Nu\n_{H^{s,2}}\\
&+(\n \partial_xu\n_{\infty,s-1,1}+\n U\n_{H^{s,2}}+\br\Dp\underline{U}\br_{l^\infty})\n H-P_Nh\n_{H^{s-1,1}}.
\end{align*}
Estimates for $\boldsymbol{R}_{\alpha,j}'$, for any $\alpha\in \N$, $j\in\{1,2\}$ such that $0\leq\alpha\leq s-j$, we have
\begin{align*}
\boldsymbol{R}_{\alpha,j}'=&-\partial_x^\alpha\Dp^j\boldsymbol{R}_N+\partial_x^\alpha\Dp^j\boldsymbol{R}'-\left\llbracket\partial_x^\alpha\Dp^j,\underline{U}+U-\kappa\frac{\partial_xH}{\underline{H}+H}\right\rrbracket\partial_x(U-P_Nu)\\&
-\partial_x^\alpha\Dp^j\left({\sf \Gamma}\partial_x(H-P_Nh)\right).
\end{align*}
Where ${\sf\Gamma}$ is as in \eqref{rhogammadecomp}.
Again we use the corresponding estimates for $\boldsymbol{R}_{\alpha,j}$ (adapted to our case) in Lemma \ref{restes}, Lemma \ref{petitsr} and Lemma \ref{diffquo} for the contribution of $\boldsymbol{R}'$, and deduce
\begin{align*}
\n \boldsymbol{R}'_{\alpha,j}&\n_{l^2(L^2_x)}\lesssim\n \boldsymbol{R}_N\n_{H^{s,2}}+(\n \partial_xu\n_{\infty,s,2}+\n U\n_{H^{s,2}}+\br\Dp\underline{U}\br_{w^{1,\infty}})\times \n U-P_Nu\n_{H^{s,2}}\\
&+C(\boldsymbol\rho_1,\br \boldsymbol\rho^{-1}\br_{l^\infty})(\n{\sf T}{\sf S}(H-P_Nh)\n_{H^{s,0}}+\n {\sf S}(H-P_Nh)\n_{H^{s,0}})\\
&+C(h_*^{-1},\br\underline{H}\br_{w^{2,\infty}},\n H\n_{H^{s-1,1}})\kappa(\n H\n^2_{H^{s,2}}+\n\partial_xH\n_{H^{s,2}})\times \n U-P_Nu\n_{H^{s,2}}\\
&+C(h_*^{-1},\underline{M},M_{\rm c},M)\kappa\n \partial_x u\n_{\infty,s,2}\\
&\times\bigg(\n\partial_x(H-P_Nh)\n_{H^{s,2}}+\n H-P_Nh\n_{H^{s,2}}(\n H\n_{H^{s,2}}+\n \partial_xh\n_{\infty,s-1,1})\\
&\hspace{0,5cm}+\n H-P_Nh\n_{H^{s-1,1}}\n \partial_xh\n_{\infty,s,2}+\n H-P_Nh\n_{H^{s-1,1}}\n \partial_x h\n_{\infty,s-1,1}(\n H\n_{H^{s,2}}+\n h\n_{\infty,s,2})\bigg).
\end{align*}

\noindent Estimates for $ r_{\alpha,j}'$ and $\mathbf{r}_{\alpha,j}'$, for any $\alpha\in \N$, $j\in\{0,1,2\}$, $0\leq\alpha\leq s-j$, we have
\begin{align*}
& r_{\alpha,j}'=-\llbracket\partial^{\alpha+1}\Dp^j;U,H-P_Nh\rrbracket+(\partial_x^\alpha\Dp^jU)({\sf M}^j\partial_x(H-P_Nh)).\\
& \mathbf{r}_{\alpha,j}'=-\llbracket\Dp^j,\underline{U}\rrbracket\partial_x^\alpha(H-P_Nh)-\partial_x^\alpha\Dp^j(\underline{H}(U-P_Nu))-(\partial_x^\alpha\Dp^jU)({\sf M}^j(H-P_Nh))\\
&\hspace{1,5cm}-\partial_x^\alpha\Dp^j((U-P_Nu)P_Nh)).
\end{align*}
The estimate of the last term of $\mathbf{r}_{\alpha,j}'$ is obtained using Lemma \ref{IL} and Lemma \ref{op P_N}. The  remaining contributions are obtained using the estimates of  $r_{\alpha,j}$ and $\mathbf{r}_{\alpha,j}$ (adapted to our case) in Lemma \ref{restes}, and we obtain
\begin{align*}
\n r_{\alpha,j}'\n_{l^2(L^2_x)}&\lesssim\n U\n_{H^{s,2}}\n H-P_Nh\n_{H^{s,2}}.\\
\n\mathbf{r}_{\alpha,j}'\n&\lesssim (\br\underline{H}\br_{w^{2,\infty}}+\n h\n_{\infty,s,2})\n U-P_Nu\n_{H^{s,2}}\\
&\hspace{5cm}+(\br\Dp\underline{U}\br_{w^{1,\infty}}+\n U\n_{H^{s,2}})\n H-P_Nh\n_{H^{s-1,1}}.
\end{align*}
This concludes the proof.
\end{proof}

\subsection{Convergence}\label{section5.3}

In this subsection we state and prove our second main result. Specifically, we show that considering any sufficiently regular solutions to the continuously stratified system \eqref{contsystem} that is bounded and satisfies the non-cavitation assumptions on a given time interval, we can construct corresponding  solutions $(H,U)$ to the multi-layer system \eqref{multicouches} on the same time interval provided that the number of layers $N$ is sufficiently large, and we quantify the convergence rate between the continuously stratified and multi-layer solutions as $N$ goes to infinity.

\begin{theorem}\label{mainth}
Let $\rho_{\bott}>\rho_{\surf}>0$, $s\in\N$ such that $s>2+\frac{1}{2}$,  and $\underline{M}, M_{\rm c},h_*,h^*,\kappa>0$. Moreover, consider $\underline{h},\,\underline{u}\in W^{2,\infty}((\rho_{\surf},\rho_{\bott}))$ such that
\begin{equation*}
\br\underline{h}\br_{W^{2,\infty}}+\br\underline{u}\br_{W^{2,\infty}}\leq \underline{M},
\end{equation*}
 and $(h,u)\in C([0,T];X^{\infty,s+1,2}_{\varrho,x})$ solution to \eqref{contsystem} on a time interval $[0, T]$ with $T>0$  such that for all $t\in[0,T]$
\begin{align}
&\underset{(x,\varrho)\in\R\times (\rho_{\surf},\rho_{\bott})}{\inf}\underline{h}(\varrho)+h(t,x,\varrho)\geq h_*,\quad \underset{(x,\varrho)\in\R\times (\rho_{\surf},\rho_{\bott})}{\sup}\underline{h}(\varrho)+h(t,x,\varrho)\leq h^*\label{noncavitationcont}
\end{align}
and
\begin{equation}
 \n h(t,\cdot)\n_{\infty,s+1,2}+\n u(t,\cdot)\n_{\infty,s+1,2}+\n\partial_\varrho\partial_x\Lambda^sh(t,\cdot)\n_{L^\infty_\varrho(L^2_x)}+\n\partial^2_\varrho\partial_x\Lambda^sh(t,\cdot)\n_{L^\infty_\varrho(L^2_x)}\leq M_{{\rm c}}.\label{contenerguregu}
 \end{equation}
Then there exists $c>0$ and $N_0\in\N^*$ such that for all $N\geq N_0$ and any initial data $(H_0, U_0)\in H^{s}(\R)^{2N}$ satisfying
\begin{equation*}
 \Norm{(H_0-P_Nh(0,\cdot),U_0-P_Nu(0,\cdot))}_s(0)\leq c M_{{\rm c}},
\end{equation*}
the solution to \eqref{multicouches} with $\underline{H}=P_N\underline{h}$,  $\underline{U}=P_N\underline{u}$ and satisfying $(H,U)_{t=0}=(H_0,U_0)$ defined in Theorem \ref{timeindepof} is well-defined on the time interval $[0,T]$ and satisfies for any $t\in[0,T]$
\begin{equation}
\underset{(x,i)\in\R\times\{1,\cdots,N\}}{\inf}\underline{H}_i+(H)_i(t,x)\geq \frac{h_*}{2},\quad \underset{(x,i)\R\times\{1,\cdots,N\}}{\sup}\underline{H}_i+(H)_i(t,x)\leq 2h^*,\label{noncavitationd}
\end{equation}
and
\begin{equation}\label{controldiscsol}
\Norm{(H,U)}_s(t)\leq \alpha (1+\sqrt{\kappa t}) M_{{\rm c}},
\end{equation}
with $\alpha$ a universal constant, and the difference between the two solutions satisfies
\begin{align}\label{controldiffsol}
   &\Norm{(H-P_Nh,U-P_Nu)}_s(t)\leq C\Norm{(H-P_Nh,U-P_Nu)}_s(0)\exp\left(\frac{C}{2}(1+\kappa^{-1}(\n\Dp\underline{U}\n_{l^2}^2+(C M_{{\rm c}})^2)t\right)\nonumber\\&\hspace{6cm}+\frac{C}{2N^2}\int\limits_0^t\exp\bigg(\frac{C}{2}(1+\kappa^{-1}(\n\Dp\underline{U}\n_{l^2}^2+(C M_{{\rm c}})^2)(t-\tau)\bigg)d\tau,
\end{align}
where $\Norm{\cdot}_s$ is defined as in Theorem \ref{timeindepof}, and $C$ depends only on $\rho_{\bott},\rho_{\surf},s,\underline{M}, M_{\rm c},h_*,h^*$ (while $N_0$ and $c$ depend also on $T$ and $\kappa$).
\end{theorem}
\begin{proof}
 From Theorem \ref{timeindepof} (or Proposition \ref{WPML}) we set $T_N>0$ the maximal time of existence and uniqueness of the solution $(H, U)\in C([0,T_N); H^{s}(\R)^{2N})$ with $H\in L^2(0,T_N; H^{s+1,2}(\R))$ to the system \eqref{multicouches} with $(H,U)_{t=0}=(H_0,U_0)$. We set $T_*\in[0,T_N]$ the larger value such that for all $t\in[0,T_*)$ we have \eqref{noncavitationd}
 and
 \[\Norm{(H,U)}_s(t)\leq c_0M_{{\rm c}},\]
where $c_0>0$ will be determined below. We can choose $c>0$ sufficiently small such that, by continuity of the energy functional we have that $T_*>0$, and we consider below $t\in (0, T_*)$ such that $t\leq T$.

Notice that the control of the difference $(H-P_Nh,U-P_Nu)(t)$ induces a control of $(H,U)(t)$ by triangular inequality: using Lemma \ref{op P_N}, there exists $\alpha>0$ a universal constant such that
\begin{align*}
\Norm{(H,U)}_s(t)&\leq \Norm{(H-P_Nh,U-P_Nu)}_s(t)+\Norm{(P_Nh,P_Nu)}_s(t)\\
&\leq \Norm{(H-P_Nh,U-P_Nu)}_s(t)+\alpha (1+\sqrt{\kappa t})M_{{\rm c}} .
\end{align*}
In the following, we prove the estimate \eqref{controldiffsol} for $t\in (0, T_*)$, that is assuming \eqref{noncavitationd}
 and $\Norm{(H,U)}_s(t)\leq c_0M_{{\rm c}}$. From \eqref{controldiffsol} we infer that under the assumptions
\begin{align}
   & \Norm{(H-P_Nh,U-P_Nu)}_s(0)\leq \epsilon\frac{ M_{{\rm c}}}{2C}\exp\left(-\frac{C}{2}(1+\kappa^{-1}(\n\Dp\underline{U}\n_{l^2}^2+(c_0M_{{\rm c}})^2)T\right),\label{ivdependency}\\
   &\frac{1}{N^2}\frac{\exp\left(\frac{C}{2}(1+\kappa^{-1}(\n\Dp\underline{U}\n_{l^2}^2+(c_0M_{{\rm c}})^2)T\right)-1}{(1+\kappa^{-1}(\n\Dp\underline{U}\n_{l^2}^2+(c_0M_{{\rm c}})^2)}\leq \epsilon\frac{M_{{\rm c}}}{2},\label{Ndependency}
\end{align}
 where $\epsilon>0$ is arbitrary, we have 
\begin{align*}
\Norm{(H,U)}_s(t)\leq (\epsilon+\alpha(1+\sqrt{\kappa t})) M_{{\rm c}}.
\end{align*}
Similarly, choosing $\epsilon$ sufficiently small, we can infer \eqref{noncavitationd} with $h_*/2$ (respectively $2h^*$) replaced with $3h_*/4$ (respectively $3h^*/2$) from \eqref{controldiffsol}-\eqref{ivdependency}-\eqref{Ndependency}.
Hence setting $c_0=2(\epsilon+\alpha(1+\sqrt T))$, the usual continuity argument implies that $T_N\geq T_\star \geq T$, and the conclusions of the Theorem hold.
\medskip

Let us now establish \eqref{controldiffsol}, assuming \eqref{noncavitationd}
 and $\Norm{(H,U)}_s(t)\leq c_0M_{{\rm c}}$. To this aim we follow the proof of Theorem \ref{timeindepof}, using the stability estimates analogous to Lemma \ref{esttransdiff,trans} and Lemma \ref{energy} together with the estimates on remainders of Lemma \ref{ResteN} and Lemma \ref{restdiff}. We use below the results and notations of Lemma \ref{restdiff}.

Following the proof of Lemma \ref{energy} and integrating in time the differential inequality instead of using Gronwall's lemma (and using the fact that $\br{\sf S}\br_{l^2\to l^2}\leq \frac{1}{\sqrt{N}}\br{\sf T}{\sf S}\br_{l^2\to l^2}+\br{\sf C}{\sf S}\br_{l^2\to l^2}$) we obtain
\begin{align}\label{1diff}
&\n {\sf S}(H-P_Nh)(t,\cdot)\n^2_{H^{s,0}}+ \n {\sf T}{\sf S}(H-P_Nh)(t,\cdot)\n^2_{H^{s,0}}+ \n (U-P_Nu)(t,\cdot)\n^2_{H^{s,0}}\nonumber\\&\quad+\kappa\n\partial_x{\sf S}(H-P_Nh)\n^2_{L^2(0,t;H^{s,0})}+\kappa\n\partial_x{\sf T}{\sf S}(H-P_Nh)\n^2_{L^2(0,t;H^{s,0})}\nonumber\\
&\quad\leq c\left(\n {\sf S}(H-P_Nh)(0,\cdot)\n^2_{H^{s,0}}+ \n (U-P_Nu)(0,\cdot)\n^2_{H^{s,0}}+ \n {\sf T}{\sf S}(H-P_Nh)(0,\cdot)\n^2_{H^{s,0}}\right)\nonumber\\
&\quad\quad+c\int_0^tC(\rho_{\surf},h_*,c_0M_{{\rm c}})\big(1+\kappa^{-1}\n \Dp (\underline{U}+U)(\tau,\cdot)\n_{L^\infty_x(l^2)}^2\big)\bigg(\Norm{(H-P_Nh,U-P_Nu)}_s(\tau)\bigg)^2d\tau\nonumber\\
&\quad\quad+c\int_0^t\bigg(\underset{0\leq\alpha\leq s}{\sum}(\n {\sf S} R_{\alpha,0} '(\tau,\cdot)\n_{l^2(L^2_x)}+\n {\sf T}{\sf S}R_{\alpha,0} '(\tau,\cdot)\n_{l^2(L^2_x)}+\n\mathbf{R}_{\alpha,0} '(\tau,\cdot)\n_{l^2(L^2_x)})\bigg)\nonumber\\
&\hspace{8.5cm}\times\bigg(\Norm{(H-P_Nh,U-P_Nu)}_s(\tau)\bigg)d\tau,
\end{align}
where $c>1$ depends on $h_*,\,h^*,\,\rho_{\surf}$ and $\rho_{\bott}$.

Similarly, integrating in time the differential inequalities that yield Lemma \ref{esttransdiff,trans} and recalling that, by definition,
\begin{align*}
\n (H-P_Nh)(t,\cdot)\n_{H^{s-1,1}}+ \n (U-P_Nu)(t,\cdot)\n_{H^{s,2}}+\kappa^\frac{1}{2} \n (H-P_Nh)(t,\cdot)\n_{H^{s,2}} \qquad & \\
\leq \Norm{(H-P_Nh,U-P_Nu)}_s(t),&
\end{align*}
we find that for $j\in\{1,2\}$, 
\begin{align}\label{2diff}
&\n\Dp^j{\sf S}(H-P_Nh)(t,\cdot)\n^2_{H^{s-j,0}}+\kappa \n\partial_x \Dp^j{\sf S}(H-P_Nh)\n^2_{L^2(0,t;H^{s-j,0})}\nonumber\\
&\quad\lesssim  \n \Dp^j{\sf S}(H-P_Nh)(0)\n^2_{H^{s-j,0}}+\int_0^t\n U(\tau,\cdot)\n_{H^{s,2}}\bigg(\Norm{(H-P_Nh,U-P_Nu)}_s(\tau)\bigg)^2d\tau\nonumber\\
&\quad\quad+\int_0^t\bigg(\sum\limits_{0\leq\alpha\leq s-j}\n R_{\alpha,j}'(\tau,\cdot)\n_{l^2(L^2_x)}\bigg)(\tau)\Norm{(H-P_Nh,U-P_Nu)}_s(\tau)d\tau,
\end{align}
for $j\in\{1,2\}$,  
\begin{align}\label{3diff}
&\n\Dp^j(U-P_Nu)(t,\cdot)\n^2_{H^{s-j,0}} \lesssim  \n \Dp^j(U-P_Nu)(0)\n^2_{H^{s-j,0}} \nonumber\\
&\quad
+\int_0^t\bigg\|\partial_x\bigg(U-\kappa\frac{\partial_xH}{\underline{H}+H}\bigg)(\tau,\cdot)\bigg\|_{l^\infty(L^\infty_x)}\bigg(\Norm{(H-P_Nh,U-P_Nu)}_s(\tau)\bigg)^2d\tau\nonumber\\
&\quad+\int_0^t\left(\underset{0\leq\alpha \leq s-j}{\sum}\n \mathbf{R}'_{\alpha,j}(\tau,\cdot)\n_{l^2(L^2_x)}\right)\bigg(\Norm{(H-P_Nh,U-P_Nu)}_s(\tau)\bigg)d\tau,
\end{align}
for $j\in\{0,1\}$,
\begin{align}\label{4diff}
\n&\Dp^j(H-P_Nh)(t,\cdot)\n^2_{H^{s-j-1,0}}+\kappa \n\partial_x \Dp^j(H-P_Nh)\n^2_{L^2(0,t;H^{s-j-1,0})}\nonumber\\
&\lesssim  \n \Dp^j(H-P_Nh)(0)\n^2_{H^{s-j-1,0}}+\int_0^t\n U(\tau,\cdot)\n_{H^{s,2}}\bigg(\Norm{(H-P_Nh,U-P_Nu)}_s(\tau)\bigg)^2d\tau\nonumber\\
&+\int_0^t\left(\underset{0\leq\alpha \leq s-j-1}{\sum}\n \mathcal{R}_{\alpha,j}'(\tau,\cdot)\n_{l^2(L^2_x)}\right)\bigg(\Norm{(H-P_Nh,U-P_Nu)}_s(\tau)\bigg)d\tau,
\end{align}
for $j\in\{0,1,2\}$, 
\begin{align}\label{5diff}
\kappa\n&\Dp^j(H-P_Nh)(t,\cdot)\n^2_{H^{s-j,0}}+\frac{\kappa^2}{2} \n\partial_x \Dp^j(H-P_Nh)\n^2_{L^2(0,t;H^{s-j,0})}\nonumber\\
&\lesssim \kappa \n \Dp^j(H-P_Nh)(0)\n^2_{H^{s-j,0}}+\int_0^t\n U(\tau,\cdot)\n_{H^{s,2}}\bigg(\Norm{(H-P_Nh,U-P_Nu)}_s(\tau)\bigg)^2d\tau\nonumber\\
&\quad +\kappa^\frac{1}{2}\int_0^t\Bigg(\sum\limits_{0\leq\alpha\leq s-j}\n r_{\alpha,j}'(\tau,\cdot)\n_{l^2(L^2_x)}\Bigg)\Norm{(H-P_Nh,U-P_Nu)}_s(\tau)d\tau\nonumber\\
&\quad +\int_0^t\Bigg(\sum\limits_{0\leq\alpha\leq s-j}\n \mathbf{r}_{\alpha,j}'(\tau,\cdot)\n_{l^2(L^2_x)}\Bigg)^2d\tau,
\end{align}
where we used Cauchy-Schwarz and Young inequalities for the contribution of $\mathbf{r}_{\alpha,j}'$.

Moreover as seen in the proof of Theorem \ref{timeindepof} we have 
\begin{equation}\label{exp1}
\left\|\partial_x \left(U-\kappa\frac{\partial_xH}{\underline{H}+H}(\tau,\cdot)\right)\right\|_{l^\infty(L^\infty_x)}\leq C (\n U(\tau,\cdot)\n_{H^{s,2}}+\kappa\n \partial_xH(\tau,\cdot)\n_{H^{s,2}}+\kappa\n H(\tau,\cdot)\n_{H^{s,2}}^2),
\end{equation}
with $C=C(h_\star^{-1},c_0M_{{\rm c}})$ and 
\begin{equation}\label{exp2}
\n \Dp (\underline{U}+U)\n_{L^\infty_x(l^2)}^2\lesssim \n\Dp\underline{U}\n_{l^2}^2+\n U\n_{H^{s,2}}^2\leq  \n\Dp\underline{U}\n_{l^2}^2+(c_0M_{{\rm c}})^2.
\end{equation}
Hence gathering estimates \eqref{1diff}-\eqref{5diff} with \eqref{exp1} and \eqref{exp2}, together with the remainder estimates obtained in Lemma \ref{ResteN} and Lemma \ref{restdiff}, and using the control of $(h,u)$ in \eqref{contenerguregu} and that for all $\tau\in[0,t]$, $ \Norm{(H,U)}_s(\tau)\leq c_0 M_{{\rm c}}$, we obtain
\begin{align*}
&\bigg(\Norm{(H-P_Nh,U-P_Nu)}_s(t)\bigg)^2\leq C^2\bigg(\Norm{(H-P_Nh,U-P_Nu)}_s(0)\bigg)^2\\
&\hspace{1cm}+C\int_0^t(1+\kappa^{-1}(\n\Dp\underline{U}\n_{l^2}^2+(c_0M_{{\rm c}})^2)\bigg(\Norm{(H-P_Nh,U-P_Nu)}_s(\tau)\bigg)^2 d\tau\\
&\hspace{1cm}+\frac{C}{N^2}\int_0^t\Norm{(H-P_Nh,U-P_Nu)}_s(\tau)d\tau \\
&\hspace{1cm}+C\sup_{\tau\in[0,t]}\bigg(\Norm{(H-P_Nh,U-P_Nu)}_s(\tau)\bigg)\left(\int_0^t\bigg(\Norm{(H-P_Nh,U-P_Nu)}_s(\tau)\bigg)^2d\tau\right)^{1/2},
\end{align*}
 where $C:=C(\rho_{\bott},\rho_{\surf}^{-1},h^*,h_*^{-1},\underline{M},c_0M_{{\rm c}})>1$ and the last term stems from the use of Cauchy-Schwarz inequality in contributions involving either $\n \partial_x(H-P_N h) \n_{H^{s,2}}$ or $\n \partial_x H \n_{H^{s,2}}$. Using Young inequality and augmenting $C$, this contribution can be absorbed in terms of the other ones, and we have simply
 \begin{align*}
&\sup_{\tau\in[0,t]}\bigg(\Norm{(H-P_Nh,U-P_Nu)}_s(t)\bigg)^2\leq C^2\bigg(\Norm{(H-P_Nh,U-P_Nu)}_s(0)\bigg)^2\\
&\hspace{1cm}+C\int_0^t(1+\kappa^{-1}(\n\Dp\underline{U}\n_{l^2}^2+(c_0M_{{\rm c}})^2)\sup_{\tau'\in[0,\tau]}\bigg(\Norm{(H-P_Nh,U-P_Nu)}_s(\tau')\bigg)^2 d\tau\\
&\hspace{1cm}+\frac{C}{N^2}\int_0^t\sup_{\tau\in[0,\tau']}\Norm{(H-P_Nh,U-P_Nu)}_s(\tau')d\tau .
\end{align*}

 By Gronwall's lemma we infer that $\sup_{\tau\in[0,t]}\Norm{(H-P_Nh,U-P_Nu)}_s(t)^2\leq Z(t)$ for all $t\in[0,T_*)$, where $Z$ is the solution to
\begin{equation*}\label{Zsystem}\left\{\begin{array}{l}
Z'(\tau)=C(1+\kappa^{-1}(\n\Dp\underline{U}\n_{l^2}^2+(c_0M_{{\rm c}})^2)Z(\tau)+\frac{C}{N^2}Z^\frac{1}{2}(\tau)\quad \tau\in[0,T_*),\\
Z(0)=C^2\Norm{(H-P_Nh,U-P_Nu)}_s(0)^2,\end{array}
\right.
\end{equation*}
whose explicit solution is
\begin{align*} &Z(t)=\Bigg(C\Norm{(H-P_Nh,U-P_Nu)}_s(0)\exp\left(\frac{C}{2}(1+\kappa^{-1}(\n\Dp\underline{U}\n_{l^2}^2+(C M_{{\rm c}})^2)t\right)\nonumber\\&\hspace{6cm}+\frac{C}{2N^2}\int\limits_0^t\exp\bigg(\frac{C}{2}(1+\kappa^{-1}(\n\Dp\underline{U}\n_{l^2}^2+(C M_{{\rm c}})^2)(t-\tau)\bigg)d\tau\Bigg)^2.
\end{align*}
Consequently we obtain \eqref{controldiffsol}, and the proof is complete.
\end{proof}
\begin{remark}
\begin{itemize}
    \item The existence of $T>0$  and $(h,u)$ sufficiently smooth  solutions to the system \eqref{contsystem}  over the time interval $[0,T]$ satisfying  \eqref{noncavitationcont} and \eqref{contenerguregu}, as considered in the previous theorem, results from Theorem 1.1 in \cite{DB}.
    \item The dependency of $c$ and $N_0$ on $T$ and $\kappa$ is expressed  in \eqref{ivdependency} and \eqref{Ndependency}. 
\end{itemize}
\end{remark}
\begin{remark}
We used the operator 
\[
P_N:	\begin{array}{ccc}
	\mathcal{C}([\rho_{\surf},\rho_{\bott}])& \rightarrow & \mathbb{R}^N\\
	f&\mapsto &\left(f(\boldsymbol{\rho}_i\right))_{1\leq i\leq N}
\end{array}
\]
to map functions defined on $[\rho_{\surf},\rho_{\bott}]$ to $N$-dimensional vectors because it enjoys the property $P_N(fg)=P_N(f)P_N(g)$. However we can adapt the proof of Thorem \ref{mainth} so as to replace $P_N$ with
\[
\overline{P_N}:	\begin{array}{ccc}
	\mathcal{C}([\rho_{\surf},\rho_{\bott}])& \rightarrow & \mathbb{R}^N\\
	f&\mapsto &\left(\frac1{\boldsymbol{\rho}_{i+1/2}-\boldsymbol{\rho}_{i-1/2}}\int_{\boldsymbol{\rho}_{i-1/2}}^{\boldsymbol{\rho}_{i+1/2}}f(\varrho)d\varrho\right)_{1\leq i\leq N}
\end{array}
\]
In this case the physical meaning of the discretized objects is more explicit. For instance, $\big(\overline{P_N}(\underline u+u)\big)_i$ is the $i-$th layer-averaged horizontal velocity, whereas $\big(\overline{P_N}(\underline h+h))_i=N(\eta_{i-1}-\eta_i)$ is the rescaled layer depth; see Figure \ref{dessin}.
\end{remark}
\paragraph{Acknowledgments} The author thanks the Centre Henri Lebesgue ANR-11-LABX-0020-01 for its stimulating mathematical research programs.

\bibliographystyle{abbrv}
\bibliography{refs}

\end{document}